\documentclass[11pt]{amsart}

\pdfoutput=1

\newif\ifDRAFT
\DRAFTfalse

\usepackage[T1]{fontenc}
\usepackage[english]{babel}
\usepackage[autostyle]{csquotes}

\usepackage[margin = 1in]{geometry}

\usepackage{amsfonts, amsmath, amsthm, amssymb, mathtools}
\allowdisplaybreaks

\usepackage{booktabs}

\usepackage{cite}
\usepackage{hyperref}
\usepackage{xcolor}
\hypersetup{
    colorlinks,
    linkcolor={red!50!black},
    citecolor={blue!50!black},
    urlcolor={blue!80!black}
}

\usepackage{graphicx}
\graphicspath{{fig/}}

\usepackage{textcomp}

\newtheorem{theorem}{Theorem}[section]
\newtheorem{proposition}[theorem]{Proposition}
\newtheorem{lemma}[theorem]{Lemma}
\newtheorem{corollary}[theorem]{Corollary}
\newtheorem{conjecture}[theorem]{Conjecture}
\theoremstyle{definition}
\newtheorem{definition}[theorem]{Definition}
\theoremstyle{remark}
\newtheorem{remark}[theorem]{Remark}
\theoremstyle{definition}
\newtheorem{example}[theorem]{Example}

\newcommand{\incl}{\xhookrightarrow{}}

\newcommand{\defeq}{:=}

\DeclareMathOperator{\Hom}{Hom}

\DeclareMathOperator{\End}{End}

\DeclareMathOperator{\tr}{tr}

\DeclareMathOperator{\Sym}{Sym}

\DeclareMathOperator{\ad}{ad}
\DeclareMathOperator{\rank}{rank}
\newcommand{\id}{\mathrm{id}}
\renewcommand{\phi}{\varphi}
\DeclarePairedDelimiter\abs{\lvert}{\rvert}

\newcommand{\N}{\mathbb{N}}
\newcommand{\Z}{\mathbb{Z}}

\newcommand{\R}{\mathbb{R}}
\newcommand{\C}{\mathbb{C}}

\newcommand{\Flow}{\mathsf{Flow}}

\newcommand{\Sc}{\mathsf{S}}
\newcommand{\Br}{\mathsf{Br}}
\newcommand{\VG}{\mathsf{VG}}

\newcommand{\csum}{\mathbin{\#}}

\usepackage{adjustbox}

\newcommand{\BrI}{\raisebox{-6pt}{\marginbox{1.5pt 0pt}{
\begingroup%
  \makeatletter%
  \providecommand\color[2][]{%
    \errmessage{(Inkscape) Color is used for the text in Inkscape, but the package 'color.sty' is not loaded}%
    \renewcommand\color[2][]{}%
  }%
  \providecommand\transparent[1]{%
    \errmessage{(Inkscape) Transparency is used (non-zero) for the text in Inkscape, but the package 'transparent.sty' is not loaded}%
    \renewcommand\transparent[1]{}%
  }%
  \providecommand\rotatebox[2]{#2}%
  \newcommand*\fsize{\dimexpr\f@size pt\relax}%
  \newcommand*\lineheight[1]{\fontsize{\fsize}{#1\fsize}\selectfont}%
  \ifx\svgwidth\undefined%
    \setlength{\unitlength}{19.06299186bp}%
    \ifx\svgscale\undefined%
      \relax%
    \else%
      \setlength{\unitlength}{\unitlength * \real{\svgscale}}%
    \fi%
  \else%
    \setlength{\unitlength}{\svgwidth}%
  \fi%
  \global\let\svgwidth\undefined%
  \global\let\svgscale\undefined%
  \makeatother%
  \begin{picture}(1,0.99130448)%
    \lineheight{1}%
    \setlength\tabcolsep{0pt}%
    \put(0,0){\includegraphics[width=\unitlength,page=1]{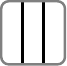}}%
  \end{picture}%
\endgroup%
}}}
\newcommand{\BrE}{\raisebox{-6pt}{\marginbox{1.5pt 0pt}{
\begingroup%
  \makeatletter%
  \providecommand\color[2][]{%
    \errmessage{(Inkscape) Color is used for the text in Inkscape, but the package 'color.sty' is not loaded}%
    \renewcommand\color[2][]{}%
  }%
  \providecommand\transparent[1]{%
    \errmessage{(Inkscape) Transparency is used (non-zero) for the text in Inkscape, but the package 'transparent.sty' is not loaded}%
    \renewcommand\transparent[1]{}%
  }%
  \providecommand\rotatebox[2]{#2}%
  \newcommand*\fsize{\dimexpr\f@size pt\relax}%
  \newcommand*\lineheight[1]{\fontsize{\fsize}{#1\fsize}\selectfont}%
  \ifx\svgwidth\undefined%
    \setlength{\unitlength}{19.06299186bp}%
    \ifx\svgscale\undefined%
      \relax%
    \else%
      \setlength{\unitlength}{\unitlength * \real{\svgscale}}%
    \fi%
  \else%
    \setlength{\unitlength}{\svgwidth}%
  \fi%
  \global\let\svgwidth\undefined%
  \global\let\svgscale\undefined%
  \makeatother%
  \begin{picture}(1,0.99999993)%
    \lineheight{1}%
    \setlength\tabcolsep{0pt}%
    \put(0,0){\includegraphics[width=\unitlength,page=1]{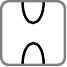}}%
  \end{picture}%
\endgroup%
}}}
\newcommand{\BrX}{\raisebox{-6pt}{\marginbox{1.5pt 0pt}{
\begingroup%
  \makeatletter%
  \providecommand\color[2][]{%
    \errmessage{(Inkscape) Color is used for the text in Inkscape, but the package 'color.sty' is not loaded}%
    \renewcommand\color[2][]{}%
  }%
  \providecommand\transparent[1]{%
    \errmessage{(Inkscape) Transparency is used (non-zero) for the text in Inkscape, but the package 'transparent.sty' is not loaded}%
    \renewcommand\transparent[1]{}%
  }%
  \providecommand\rotatebox[2]{#2}%
  \newcommand*\fsize{\dimexpr\f@size pt\relax}%
  \newcommand*\lineheight[1]{\fontsize{\fsize}{#1\fsize}\selectfont}%
  \ifx\svgwidth\undefined%
    \setlength{\unitlength}{19.06299186bp}%
    \ifx\svgscale\undefined%
      \relax%
    \else%
      \setlength{\unitlength}{\unitlength * \real{\svgscale}}%
    \fi%
  \else%
    \setlength{\unitlength}{\svgwidth}%
  \fi%
  \global\let\svgwidth\undefined%
  \global\let\svgscale\undefined%
  \makeatother%
  \begin{picture}(1,0.99999993)%
    \lineheight{1}%
    \setlength\tabcolsep{0pt}%
    \put(0,0){\includegraphics[width=\unitlength,page=1]{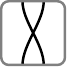}}%
  \end{picture}%
\endgroup%
}}}
\newcommand{\ScVfour}{\raisebox{-6pt}{\marginbox{1.5pt 0pt}{
\begingroup%
  \makeatletter%
  \providecommand\color[2][]{%
    \errmessage{(Inkscape) Color is used for the text in Inkscape, but the package 'color.sty' is not loaded}%
    \renewcommand\color[2][]{}%
  }%
  \providecommand\transparent[1]{%
    \errmessage{(Inkscape) Transparency is used (non-zero) for the text in Inkscape, but the package 'transparent.sty' is not loaded}%
    \renewcommand\transparent[1]{}%
  }%
  \providecommand\rotatebox[2]{#2}%
  \newcommand*\fsize{\dimexpr\f@size pt\relax}%
  \newcommand*\lineheight[1]{\fontsize{\fsize}{#1\fsize}\selectfont}%
  \ifx\svgwidth\undefined%
    \setlength{\unitlength}{19.06299186bp}%
    \ifx\svgscale\undefined%
      \relax%
    \else%
      \setlength{\unitlength}{\unitlength * \real{\svgscale}}%
    \fi%
  \else%
    \setlength{\unitlength}{\svgwidth}%
  \fi%
  \global\let\svgwidth\undefined%
  \global\let\svgscale\undefined%
  \makeatother%
  \begin{picture}(1,0.99999993)%
    \lineheight{1}%
    \setlength\tabcolsep{0pt}%
    \put(0,0){\includegraphics[width=\unitlength,page=1]{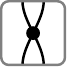}}%
  \end{picture}%
\endgroup%
}}}

\newcommand{\BrCap}{\raisebox{-3pt}{\marginbox{1.5pt 0pt}{
\begingroup%
  \makeatletter%
  \providecommand\color[2][]{%
    \errmessage{(Inkscape) Color is used for the text in Inkscape, but the package 'color.sty' is not loaded}%
    \renewcommand\color[2][]{}%
  }%
  \providecommand\transparent[1]{%
    \errmessage{(Inkscape) Transparency is used (non-zero) for the text in Inkscape, but the package 'transparent.sty' is not loaded}%
    \renewcommand\transparent[1]{}%
  }%
  \providecommand\rotatebox[2]{#2}%
  \newcommand*\fsize{\dimexpr\f@size pt\relax}%
  \newcommand*\lineheight[1]{\fontsize{\fsize}{#1\fsize}\selectfont}%
  \ifx\svgwidth\undefined%
    \setlength{\unitlength}{19.06299321bp}%
    \ifx\svgscale\undefined%
      \relax%
    \else%
      \setlength{\unitlength}{\unitlength * \real{\svgscale}}%
    \fi%
  \else%
    \setlength{\unitlength}{\svgwidth}%
  \fi%
  \global\let\svgwidth\undefined%
  \global\let\svgscale\undefined%
  \makeatother%
  \begin{picture}(1,0.71912428)%
    \lineheight{1}%
    \setlength\tabcolsep{0pt}%
    \put(0,0){\includegraphics[width=\unitlength,page=1]{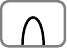}}%
  \end{picture}%
\endgroup%
}}}
\newcommand{\BrCup}{\raisebox{-3pt}{\marginbox{1.5pt 0pt}{
\begingroup%
  \makeatletter%
  \providecommand\color[2][]{%
    \errmessage{(Inkscape) Color is used for the text in Inkscape, but the package 'color.sty' is not loaded}%
    \renewcommand\color[2][]{}%
  }%
  \providecommand\transparent[1]{%
    \errmessage{(Inkscape) Transparency is used (non-zero) for the text in Inkscape, but the package 'transparent.sty' is not loaded}%
    \renewcommand\transparent[1]{}%
  }%
  \providecommand\rotatebox[2]{#2}%
  \newcommand*\fsize{\dimexpr\f@size pt\relax}%
  \newcommand*\lineheight[1]{\fontsize{\fsize}{#1\fsize}\selectfont}%
  \ifx\svgwidth\undefined%
    \setlength{\unitlength}{19.06299321bp}%
    \ifx\svgscale\undefined%
      \relax%
    \else%
      \setlength{\unitlength}{\unitlength * \real{\svgscale}}%
    \fi%
  \else%
    \setlength{\unitlength}{\svgwidth}%
  \fi%
  \global\let\svgwidth\undefined%
  \global\let\svgscale\undefined%
  \makeatother%
  \begin{picture}(1,0.71912428)%
    \lineheight{1}%
    \setlength\tabcolsep{0pt}%
    \put(0,0){\includegraphics[width=\unitlength,page=1]{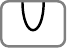}}%
  \end{picture}%
\endgroup%
}}}
\newcommand{\BrIs}{\raisebox{-3pt}{\marginbox{1.5pt 0pt}{
\begingroup%
  \makeatletter%
  \providecommand\color[2][]{%
    \errmessage{(Inkscape) Color is used for the text in Inkscape, but the package 'color.sty' is not loaded}%
    \renewcommand\color[2][]{}%
  }%
  \providecommand\transparent[1]{%
    \errmessage{(Inkscape) Transparency is used (non-zero) for the text in Inkscape, but the package 'transparent.sty' is not loaded}%
    \renewcommand\transparent[1]{}%
  }%
  \providecommand\rotatebox[2]{#2}%
  \newcommand*\fsize{\dimexpr\f@size pt\relax}%
  \newcommand*\lineheight[1]{\fontsize{\fsize}{#1\fsize}\selectfont}%
  \ifx\svgwidth\undefined%
    \setlength{\unitlength}{12.75830607bp}%
    \ifx\svgscale\undefined%
      \relax%
    \else%
      \setlength{\unitlength}{\unitlength * \real{\svgscale}}%
    \fi%
  \else%
    \setlength{\unitlength}{\svgwidth}%
  \fi%
  \global\let\svgwidth\undefined%
  \global\let\svgscale\undefined%
  \makeatother%
  \begin{picture}(1,1.07448914)%
    \lineheight{1}%
    \setlength\tabcolsep{0pt}%
    \put(0,0){\includegraphics[width=\unitlength,page=1]{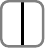}}%
  \end{picture}%
\endgroup%
}}}

\newcommand{\BrSOne}{\raisebox{-10pt}{\marginbox{1.5pt 0pt}{
\begingroup%
  \makeatletter%
  \providecommand\color[2][]{%
    \errmessage{(Inkscape) Color is used for the text in Inkscape, but the package 'color.sty' is not loaded}%
    \renewcommand\color[2][]{}%
  }%
  \providecommand\transparent[1]{%
    \errmessage{(Inkscape) Transparency is used (non-zero) for the text in Inkscape, but the package 'transparent.sty' is not loaded}%
    \renewcommand\transparent[1]{}%
  }%
  \providecommand\rotatebox[2]{#2}%
  \newcommand*\fsize{\dimexpr\f@size pt\relax}%
  \newcommand*\lineheight[1]{\fontsize{\fsize}{#1\fsize}\selectfont}%
  \ifx\svgwidth\undefined%
    \setlength{\unitlength}{25.3133083bp}%
    \ifx\svgscale\undefined%
      \relax%
    \else%
      \setlength{\unitlength}{\unitlength * \real{\svgscale}}%
    \fi%
  \else%
    \setlength{\unitlength}{\svgwidth}%
  \fi%
  \global\let\svgwidth\undefined%
  \global\let\svgscale\undefined%
  \makeatother%
  \begin{picture}(1,0.9343837)%
    \lineheight{1}%
    \setlength\tabcolsep{0pt}%
    \put(0,0){\includegraphics[width=\unitlength,page=1]{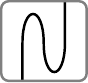}}%
  \end{picture}%
\endgroup%
}}}
\newcommand{\BrSTwo}{\raisebox{-10pt}{\marginbox{1.5pt 0pt}{
\begingroup%
  \makeatletter%
  \providecommand\color[2][]{%
    \errmessage{(Inkscape) Color is used for the text in Inkscape, but the package 'color.sty' is not loaded}%
    \renewcommand\color[2][]{}%
  }%
  \providecommand\transparent[1]{%
    \errmessage{(Inkscape) Transparency is used (non-zero) for the text in Inkscape, but the package 'transparent.sty' is not loaded}%
    \renewcommand\transparent[1]{}%
  }%
  \providecommand\rotatebox[2]{#2}%
  \newcommand*\fsize{\dimexpr\f@size pt\relax}%
  \newcommand*\lineheight[1]{\fontsize{\fsize}{#1\fsize}\selectfont}%
  \ifx\svgwidth\undefined%
    \setlength{\unitlength}{25.3133083bp}%
    \ifx\svgscale\undefined%
      \relax%
    \else%
      \setlength{\unitlength}{\unitlength * \real{\svgscale}}%
    \fi%
  \else%
    \setlength{\unitlength}{\svgwidth}%
  \fi%
  \global\let\svgwidth\undefined%
  \global\let\svgscale\undefined%
  \makeatother%
  \begin{picture}(1,0.9343837)%
    \lineheight{1}%
    \setlength\tabcolsep{0pt}%
    \put(0,0){\includegraphics[width=\unitlength,page=1]{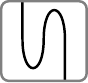}}%
  \end{picture}%
\endgroup%
}}}

\newcommand{\BrGlKilling}{\raisebox{-5pt}{\marginbox{1.5pt 0pt}{
\begingroup%
  \makeatletter%
  \providecommand\color[2][]{%
    \errmessage{(Inkscape) Color is used for the text in Inkscape, but the package 'color.sty' is not loaded}%
    \renewcommand\color[2][]{}%
  }%
  \providecommand\transparent[1]{%
    \errmessage{(Inkscape) Transparency is used (non-zero) for the text in Inkscape, but the package 'transparent.sty' is not loaded}%
    \renewcommand\transparent[1]{}%
  }%
  \providecommand\rotatebox[2]{#2}%
  \newcommand*\fsize{\dimexpr\f@size pt\relax}%
  \newcommand*\lineheight[1]{\fontsize{\fsize}{#1\fsize}\selectfont}%
  \ifx\svgwidth\undefined%
    \setlength{\unitlength}{29.10137747bp}%
    \ifx\svgscale\undefined%
      \relax%
    \else%
      \setlength{\unitlength}{\unitlength * \real{\svgscale}}%
    \fi%
  \else%
    \setlength{\unitlength}{\svgwidth}%
  \fi%
  \global\let\svgwidth\undefined%
  \global\let\svgscale\undefined%
  \makeatother%
  \begin{picture}(1,0.6142505)%
    \lineheight{1}%
    \setlength\tabcolsep{0pt}%
    \put(0,0){\includegraphics[width=\unitlength,page=1]{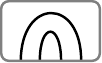}}%
  \end{picture}%
\endgroup%
}}}
\newcommand{\BrGlCasimir}{\raisebox{-5pt}{\marginbox{1.5pt 0pt}{
\begingroup%
  \makeatletter%
  \providecommand\color[2][]{%
    \errmessage{(Inkscape) Color is used for the text in Inkscape, but the package 'color.sty' is not loaded}%
    \renewcommand\color[2][]{}%
  }%
  \providecommand\transparent[1]{%
    \errmessage{(Inkscape) Transparency is used (non-zero) for the text in Inkscape, but the package 'transparent.sty' is not loaded}%
    \renewcommand\transparent[1]{}%
  }%
  \providecommand\rotatebox[2]{#2}%
  \newcommand*\fsize{\dimexpr\f@size pt\relax}%
  \newcommand*\lineheight[1]{\fontsize{\fsize}{#1\fsize}\selectfont}%
  \ifx\svgwidth\undefined%
    \setlength{\unitlength}{29.10137747bp}%
    \ifx\svgscale\undefined%
      \relax%
    \else%
      \setlength{\unitlength}{\unitlength * \real{\svgscale}}%
    \fi%
  \else%
    \setlength{\unitlength}{\svgwidth}%
  \fi%
  \global\let\svgwidth\undefined%
  \global\let\svgscale\undefined%
  \makeatother%
  \begin{picture}(1,0.6142505)%
    \lineheight{1}%
    \setlength\tabcolsep{0pt}%
    \put(0,0){\includegraphics[width=\unitlength,page=1]{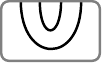}}%
  \end{picture}%
\endgroup%
}}}
\newcommand{\BrLieOne}{\raisebox{-8pt}{\marginbox{1.5pt 0pt}{
\begingroup%
  \makeatletter%
  \providecommand\color[2][]{%
    \errmessage{(Inkscape) Color is used for the text in Inkscape, but the package 'color.sty' is not loaded}%
    \renewcommand\color[2][]{}%
  }%
  \providecommand\transparent[1]{%
    \errmessage{(Inkscape) Transparency is used (non-zero) for the text in Inkscape, but the package 'transparent.sty' is not loaded}%
    \renewcommand\transparent[1]{}%
  }%
  \providecommand\rotatebox[2]{#2}%
  \newcommand*\fsize{\dimexpr\f@size pt\relax}%
  \newcommand*\lineheight[1]{\fontsize{\fsize}{#1\fsize}\selectfont}%
  \ifx\svgwidth\undefined%
    \setlength{\unitlength}{25.3133083bp}%
    \ifx\svgscale\undefined%
      \relax%
    \else%
      \setlength{\unitlength}{\unitlength * \real{\svgscale}}%
    \fi%
  \else%
    \setlength{\unitlength}{\svgwidth}%
  \fi%
  \global\let\svgwidth\undefined%
  \global\let\svgscale\undefined%
  \makeatother%
  \begin{picture}(1,0.9343837)%
    \lineheight{1}%
    \setlength\tabcolsep{0pt}%
    \put(0,0){\includegraphics[width=\unitlength,page=1]{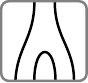}}%
  \end{picture}%
\endgroup%
}}}
\newcommand{\BrLieTwo}{\raisebox{-8pt}{\marginbox{1.5pt 0pt}{
\begingroup%
  \makeatletter%
  \providecommand\color[2][]{%
    \errmessage{(Inkscape) Color is used for the text in Inkscape, but the package 'color.sty' is not loaded}%
    \renewcommand\color[2][]{}%
  }%
  \providecommand\transparent[1]{%
    \errmessage{(Inkscape) Transparency is used (non-zero) for the text in Inkscape, but the package 'transparent.sty' is not loaded}%
    \renewcommand\transparent[1]{}%
  }%
  \providecommand\rotatebox[2]{#2}%
  \newcommand*\fsize{\dimexpr\f@size pt\relax}%
  \newcommand*\lineheight[1]{\fontsize{\fsize}{#1\fsize}\selectfont}%
  \ifx\svgwidth\undefined%
    \setlength{\unitlength}{25.3133083bp}%
    \ifx\svgscale\undefined%
      \relax%
    \else%
      \setlength{\unitlength}{\unitlength * \real{\svgscale}}%
    \fi%
  \else%
    \setlength{\unitlength}{\svgwidth}%
  \fi%
  \global\let\svgwidth\undefined%
  \global\let\svgscale\undefined%
  \makeatother%
  \begin{picture}(1,0.9343837)%
    \lineheight{1}%
    \setlength\tabcolsep{0pt}%
    \put(0,0){\includegraphics[width=\unitlength,page=1]{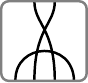}}%
  \end{picture}%
\endgroup%
}}}

\newcommand{\BrXDouble}{\raisebox{-8pt}{\marginbox{1.5pt 0pt}{
\begingroup%
  \makeatletter%
  \providecommand\color[2][]{%
    \errmessage{(Inkscape) Color is used for the text in Inkscape, but the package 'color.sty' is not loaded}%
    \renewcommand\color[2][]{}%
  }%
  \providecommand\transparent[1]{%
    \errmessage{(Inkscape) Transparency is used (non-zero) for the text in Inkscape, but the package 'transparent.sty' is not loaded}%
    \renewcommand\transparent[1]{}%
  }%
  \providecommand\rotatebox[2]{#2}%
  \newcommand*\fsize{\dimexpr\f@size pt\relax}%
  \newcommand*\lineheight[1]{\fontsize{\fsize}{#1\fsize}\selectfont}%
  \ifx\svgwidth\undefined%
    \setlength{\unitlength}{25.3133083bp}%
    \ifx\svgscale\undefined%
      \relax%
    \else%
      \setlength{\unitlength}{\unitlength * \real{\svgscale}}%
    \fi%
  \else%
    \setlength{\unitlength}{\svgwidth}%
  \fi%
  \global\let\svgwidth\undefined%
  \global\let\svgscale\undefined%
  \makeatother%
  \begin{picture}(1,0.9343837)%
    \lineheight{1}%
    \setlength\tabcolsep{0pt}%
    \put(0,0){\includegraphics[width=\unitlength,page=1]{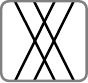}}%
  \end{picture}%
\endgroup%
}}}

\newcommand{\BrAlt}{\raisebox{-6pt}{\marginbox{1.5pt 0pt}{
\begingroup%
  \makeatletter%
  \providecommand\color[2][]{%
    \errmessage{(Inkscape) Color is used for the text in Inkscape, but the package 'color.sty' is not loaded}%
    \renewcommand\color[2][]{}%
  }%
  \providecommand\transparent[1]{%
    \errmessage{(Inkscape) Transparency is used (non-zero) for the text in Inkscape, but the package 'transparent.sty' is not loaded}%
    \renewcommand\transparent[1]{}%
  }%
  \providecommand\rotatebox[2]{#2}%
  \newcommand*\fsize{\dimexpr\f@size pt\relax}%
  \newcommand*\lineheight[1]{\fontsize{\fsize}{#1\fsize}\selectfont}%
  \ifx\svgwidth\undefined%
    \setlength{\unitlength}{19.06299186bp}%
    \ifx\svgscale\undefined%
      \relax%
    \else%
      \setlength{\unitlength}{\unitlength * \real{\svgscale}}%
    \fi%
  \else%
    \setlength{\unitlength}{\svgwidth}%
  \fi%
  \global\let\svgwidth\undefined%
  \global\let\svgscale\undefined%
  \makeatother%
  \begin{picture}(1,0.99130448)%
    \lineheight{1}%
    \setlength\tabcolsep{0pt}%
    \put(0,0){\includegraphics[width=\unitlength,page=1]{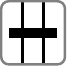}}%
  \end{picture}%
\endgroup%
}}}
\newcommand{\BrJ}{\raisebox{-6pt}{\marginbox{1.5pt 0pt}{
\begingroup%
  \makeatletter%
  \providecommand\color[2][]{%
    \errmessage{(Inkscape) Color is used for the text in Inkscape, but the package 'color.sty' is not loaded}%
    \renewcommand\color[2][]{}%
  }%
  \providecommand\transparent[1]{%
    \errmessage{(Inkscape) Transparency is used (non-zero) for the text in Inkscape, but the package 'transparent.sty' is not loaded}%
    \renewcommand\transparent[1]{}%
  }%
  \providecommand\rotatebox[2]{#2}%
  \newcommand*\fsize{\dimexpr\f@size pt\relax}%
  \newcommand*\lineheight[1]{\fontsize{\fsize}{#1\fsize}\selectfont}%
  \ifx\svgwidth\undefined%
    \setlength{\unitlength}{19.06299186bp}%
    \ifx\svgscale\undefined%
      \relax%
    \else%
      \setlength{\unitlength}{\unitlength * \real{\svgscale}}%
    \fi%
  \else%
    \setlength{\unitlength}{\svgwidth}%
  \fi%
  \global\let\svgwidth\undefined%
  \global\let\svgscale\undefined%
  \makeatother%
  \begin{picture}(1,0.99130448)%
    \lineheight{1}%
    \setlength\tabcolsep{0pt}%
    \put(0,0){\includegraphics[width=\unitlength,page=1]{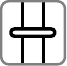}}%
  \end{picture}%
\endgroup%
}}}

\newcommand{\ScTri}{\raisebox{-3pt}{\marginbox{1.5pt 0pt}{
\begingroup%
  \makeatletter%
  \providecommand\color[2][]{%
    \errmessage{(Inkscape) Color is used for the text in Inkscape, but the package 'color.sty' is not loaded}%
    \renewcommand\color[2][]{}%
  }%
  \providecommand\transparent[1]{%
    \errmessage{(Inkscape) Transparency is used (non-zero) for the text in Inkscape, but the package 'transparent.sty' is not loaded}%
    \renewcommand\transparent[1]{}%
  }%
  \providecommand\rotatebox[2]{#2}%
  \newcommand*\fsize{\dimexpr\f@size pt\relax}%
  \newcommand*\lineheight[1]{\fontsize{\fsize}{#1\fsize}\selectfont}%
  \ifx\svgwidth\undefined%
    \setlength{\unitlength}{23.96091938bp}%
    \ifx\svgscale\undefined%
      \relax%
    \else%
      \setlength{\unitlength}{\unitlength * \real{\svgscale}}%
    \fi%
  \else%
    \setlength{\unitlength}{\svgwidth}%
  \fi%
  \global\let\svgwidth\undefined%
  \global\let\svgscale\undefined%
  \makeatother%
  \begin{picture}(1,0.57212585)%
    \lineheight{1}%
    \setlength\tabcolsep{0pt}%
    \put(0,0){\includegraphics[width=\unitlength,page=1]{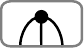}}%
  \end{picture}%
\endgroup%
}}}
\newcommand{\ScTriFlip}{\raisebox{-3pt}{\marginbox{1.5pt 0pt}{
\begingroup%
  \makeatletter%
  \providecommand\color[2][]{%
    \errmessage{(Inkscape) Color is used for the text in Inkscape, but the package 'color.sty' is not loaded}%
    \renewcommand\color[2][]{}%
  }%
  \providecommand\transparent[1]{%
    \errmessage{(Inkscape) Transparency is used (non-zero) for the text in Inkscape, but the package 'transparent.sty' is not loaded}%
    \renewcommand\transparent[1]{}%
  }%
  \providecommand\rotatebox[2]{#2}%
  \newcommand*\fsize{\dimexpr\f@size pt\relax}%
  \newcommand*\lineheight[1]{\fontsize{\fsize}{#1\fsize}\selectfont}%
  \ifx\svgwidth\undefined%
    \setlength{\unitlength}{23.96091938bp}%
    \ifx\svgscale\undefined%
      \relax%
    \else%
      \setlength{\unitlength}{\unitlength * \real{\svgscale}}%
    \fi%
  \else%
    \setlength{\unitlength}{\svgwidth}%
  \fi%
  \global\let\svgwidth\undefined%
  \global\let\svgscale\undefined%
  \makeatother%
  \begin{picture}(1,0.57212585)%
    \lineheight{1}%
    \setlength\tabcolsep{0pt}%
    \put(0,0){\includegraphics[width=\unitlength,page=1]{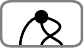}}%
  \end{picture}%
\endgroup%
}}}
\newcommand{\ScTriDual}{\raisebox{-3pt}{\marginbox{1.5pt 0pt}{
\begingroup%
  \makeatletter%
  \providecommand\color[2][]{%
    \errmessage{(Inkscape) Color is used for the text in Inkscape, but the package 'color.sty' is not loaded}%
    \renewcommand\color[2][]{}%
  }%
  \providecommand\transparent[1]{%
    \errmessage{(Inkscape) Transparency is used (non-zero) for the text in Inkscape, but the package 'transparent.sty' is not loaded}%
    \renewcommand\transparent[1]{}%
  }%
  \providecommand\rotatebox[2]{#2}%
  \newcommand*\fsize{\dimexpr\f@size pt\relax}%
  \newcommand*\lineheight[1]{\fontsize{\fsize}{#1\fsize}\selectfont}%
  \ifx\svgwidth\undefined%
    \setlength{\unitlength}{23.96091938bp}%
    \ifx\svgscale\undefined%
      \relax%
    \else%
      \setlength{\unitlength}{\unitlength * \real{\svgscale}}%
    \fi%
  \else%
    \setlength{\unitlength}{\svgwidth}%
  \fi%
  \global\let\svgwidth\undefined%
  \global\let\svgscale\undefined%
  \makeatother%
  \begin{picture}(1,0.57212585)%
    \lineheight{1}%
    \setlength\tabcolsep{0pt}%
    \put(0,0){\includegraphics[width=\unitlength,page=1]{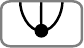}}%
  \end{picture}%
\endgroup%
}}}
\newcommand{\ScTriFlipDual}{\raisebox{-3pt}{\marginbox{1.5pt 0pt}{
\begingroup%
  \makeatletter%
  \providecommand\color[2][]{%
    \errmessage{(Inkscape) Color is used for the text in Inkscape, but the package 'color.sty' is not loaded}%
    \renewcommand\color[2][]{}%
  }%
  \providecommand\transparent[1]{%
    \errmessage{(Inkscape) Transparency is used (non-zero) for the text in Inkscape, but the package 'transparent.sty' is not loaded}%
    \renewcommand\transparent[1]{}%
  }%
  \providecommand\rotatebox[2]{#2}%
  \newcommand*\fsize{\dimexpr\f@size pt\relax}%
  \newcommand*\lineheight[1]{\fontsize{\fsize}{#1\fsize}\selectfont}%
  \ifx\svgwidth\undefined%
    \setlength{\unitlength}{23.96091938bp}%
    \ifx\svgscale\undefined%
      \relax%
    \else%
      \setlength{\unitlength}{\unitlength * \real{\svgscale}}%
    \fi%
  \else%
    \setlength{\unitlength}{\svgwidth}%
  \fi%
  \global\let\svgwidth\undefined%
  \global\let\svgscale\undefined%
  \makeatother%
  \begin{picture}(1,0.57212585)%
    \lineheight{1}%
    \setlength\tabcolsep{0pt}%
    \put(0,0){\includegraphics[width=\unitlength,page=1]{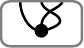}}%
  \end{picture}%
\endgroup%
}}}

\newcommand{\IHXadA}{\raisebox{-8pt}{\marginbox{1.5pt 0pt}{
\begingroup%
  \makeatletter%
  \providecommand\color[2][]{%
    \errmessage{(Inkscape) Color is used for the text in Inkscape, but the package 'color.sty' is not loaded}%
    \renewcommand\color[2][]{}%
  }%
  \providecommand\transparent[1]{%
    \errmessage{(Inkscape) Transparency is used (non-zero) for the text in Inkscape, but the package 'transparent.sty' is not loaded}%
    \renewcommand\transparent[1]{}%
  }%
  \providecommand\rotatebox[2]{#2}%
  \newcommand*\fsize{\dimexpr\f@size pt\relax}%
  \newcommand*\lineheight[1]{\fontsize{\fsize}{#1\fsize}\selectfont}%
  \ifx\svgwidth\undefined%
    \setlength{\unitlength}{22.32655767bp}%
    \ifx\svgscale\undefined%
      \relax%
    \else%
      \setlength{\unitlength}{\unitlength * \real{\svgscale}}%
    \fi%
  \else%
    \setlength{\unitlength}{\svgwidth}%
  \fi%
  \global\let\svgwidth\undefined%
  \global\let\svgscale\undefined%
  \makeatother%
  \begin{picture}(1,1.32691446)%
    \lineheight{1}%
    \setlength\tabcolsep{0pt}%
    \put(0,0){\includegraphics[width=\unitlength,page=1]{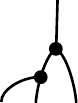}}%
  \end{picture}%
\endgroup%
}}}
\newcommand{\IHXadB}{\raisebox{-8pt}{\marginbox{1.5pt 0pt}{
\begingroup%
  \makeatletter%
  \providecommand\color[2][]{%
    \errmessage{(Inkscape) Color is used for the text in Inkscape, but the package 'color.sty' is not loaded}%
    \renewcommand\color[2][]{}%
  }%
  \providecommand\transparent[1]{%
    \errmessage{(Inkscape) Transparency is used (non-zero) for the text in Inkscape, but the package 'transparent.sty' is not loaded}%
    \renewcommand\transparent[1]{}%
  }%
  \providecommand\rotatebox[2]{#2}%
  \newcommand*\fsize{\dimexpr\f@size pt\relax}%
  \newcommand*\lineheight[1]{\fontsize{\fsize}{#1\fsize}\selectfont}%
  \ifx\svgwidth\undefined%
    \setlength{\unitlength}{22.51837607bp}%
    \ifx\svgscale\undefined%
      \relax%
    \else%
      \setlength{\unitlength}{\unitlength * \real{\svgscale}}%
    \fi%
  \else%
    \setlength{\unitlength}{\svgwidth}%
  \fi%
  \global\let\svgwidth\undefined%
  \global\let\svgscale\undefined%
  \makeatother%
  \begin{picture}(1,1.31561152)%
    \lineheight{1}%
    \setlength\tabcolsep{0pt}%
    \put(0,0){\includegraphics[width=\unitlength,page=1]{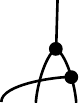}}%
  \end{picture}%
\endgroup%
}}}
\newcommand{\IHXadC}{\raisebox{-8pt}{\marginbox{1.5pt 0pt}{
\begingroup%
  \makeatletter%
  \providecommand\color[2][]{%
    \errmessage{(Inkscape) Color is used for the text in Inkscape, but the package 'color.sty' is not loaded}%
    \renewcommand\color[2][]{}%
  }%
  \providecommand\transparent[1]{%
    \errmessage{(Inkscape) Transparency is used (non-zero) for the text in Inkscape, but the package 'transparent.sty' is not loaded}%
    \renewcommand\transparent[1]{}%
  }%
  \providecommand\rotatebox[2]{#2}%
  \newcommand*\fsize{\dimexpr\f@size pt\relax}%
  \newcommand*\lineheight[1]{\fontsize{\fsize}{#1\fsize}\selectfont}%
  \ifx\svgwidth\undefined%
    \setlength{\unitlength}{22.32155614bp}%
    \ifx\svgscale\undefined%
      \relax%
    \else%
      \setlength{\unitlength}{\unitlength * \real{\svgscale}}%
    \fi%
  \else%
    \setlength{\unitlength}{\svgwidth}%
  \fi%
  \global\let\svgwidth\undefined%
  \global\let\svgscale\undefined%
  \makeatother%
  \begin{picture}(1,1.32721177)%
    \lineheight{1}%
    \setlength\tabcolsep{0pt}%
    \put(0,0){\includegraphics[width=\unitlength,page=1]{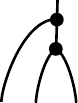}}%
  \end{picture}%
\endgroup%
}}}

\newcommand{\Crleft}{\raisebox{-6pt}{\marginbox{1.5pt 0pt}{
\begingroup%
  \makeatletter%
  \providecommand\color[2][]{%
    \errmessage{(Inkscape) Color is used for the text in Inkscape, but the package 'color.sty' is not loaded}%
    \renewcommand\color[2][]{}%
  }%
  \providecommand\transparent[1]{%
    \errmessage{(Inkscape) Transparency is used (non-zero) for the text in Inkscape, but the package 'transparent.sty' is not loaded}%
    \renewcommand\transparent[1]{}%
  }%
  \providecommand\rotatebox[2]{#2}%
  \newcommand*\fsize{\dimexpr\f@size pt\relax}%
  \newcommand*\lineheight[1]{\fontsize{\fsize}{#1\fsize}\selectfont}%
  \ifx\svgwidth\undefined%
    \setlength{\unitlength}{10.82395958bp}%
    \ifx\svgscale\undefined%
      \relax%
    \else%
      \setlength{\unitlength}{\unitlength * \real{\svgscale}}%
    \fi%
  \else%
    \setlength{\unitlength}{\svgwidth}%
  \fi%
  \global\let\svgwidth\undefined%
  \global\let\svgscale\undefined%
  \makeatother%
  \begin{picture}(1,1.80573815)%
    \lineheight{1}%
    \setlength\tabcolsep{0pt}%
    \put(0,0){\includegraphics[width=\unitlength,page=1]{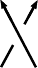}}%
  \end{picture}%
\endgroup%
}}}
\newcommand{\Crright}{\raisebox{-6pt}{\marginbox{1.5pt 0pt}{
\begingroup%
  \makeatletter%
  \providecommand\color[2][]{%
    \errmessage{(Inkscape) Color is used for the text in Inkscape, but the package 'color.sty' is not loaded}%
    \renewcommand\color[2][]{}%
  }%
  \providecommand\transparent[1]{%
    \errmessage{(Inkscape) Transparency is used (non-zero) for the text in Inkscape, but the package 'transparent.sty' is not loaded}%
    \renewcommand\transparent[1]{}%
  }%
  \providecommand\rotatebox[2]{#2}%
  \newcommand*\fsize{\dimexpr\f@size pt\relax}%
  \newcommand*\lineheight[1]{\fontsize{\fsize}{#1\fsize}\selectfont}%
  \ifx\svgwidth\undefined%
    \setlength{\unitlength}{10.82395958bp}%
    \ifx\svgscale\undefined%
      \relax%
    \else%
      \setlength{\unitlength}{\unitlength * \real{\svgscale}}%
    \fi%
  \else%
    \setlength{\unitlength}{\svgwidth}%
  \fi%
  \global\let\svgwidth\undefined%
  \global\let\svgscale\undefined%
  \makeatother%
  \begin{picture}(1,1.80573815)%
    \lineheight{1}%
    \setlength\tabcolsep{0pt}%
    \put(0,0){\includegraphics[width=\unitlength,page=1]{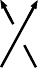}}%
  \end{picture}%
\endgroup%
}}}
\newcommand{\Crzero}{\raisebox{-6pt}{\marginbox{1.5pt 0pt}{
\begingroup%
  \makeatletter%
  \providecommand\color[2][]{%
    \errmessage{(Inkscape) Color is used for the text in Inkscape, but the package 'color.sty' is not loaded}%
    \renewcommand\color[2][]{}%
  }%
  \providecommand\transparent[1]{%
    \errmessage{(Inkscape) Transparency is used (non-zero) for the text in Inkscape, but the package 'transparent.sty' is not loaded}%
    \renewcommand\transparent[1]{}%
  }%
  \providecommand\rotatebox[2]{#2}%
  \newcommand*\fsize{\dimexpr\f@size pt\relax}%
  \newcommand*\lineheight[1]{\fontsize{\fsize}{#1\fsize}\selectfont}%
  \ifx\svgwidth\undefined%
    \setlength{\unitlength}{10.8110676bp}%
    \ifx\svgscale\undefined%
      \relax%
    \else%
      \setlength{\unitlength}{\unitlength * \real{\svgscale}}%
    \fi%
  \else%
    \setlength{\unitlength}{\svgwidth}%
  \fi%
  \global\let\svgwidth\undefined%
  \global\let\svgscale\undefined%
  \makeatother%
  \begin{picture}(1,1.80802073)%
    \lineheight{1}%
    \setlength\tabcolsep{0pt}%
    \put(0,0){\includegraphics[width=\unitlength,page=1]{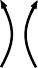}}%
  \end{picture}%
\endgroup%
}}}
\newcommand{\Crid}{\raisebox{-6pt}{\marginbox{1.5pt 0pt}{
\begingroup%
  \makeatletter%
  \providecommand\color[2][]{%
    \errmessage{(Inkscape) Color is used for the text in Inkscape, but the package 'color.sty' is not loaded}%
    \renewcommand\color[2][]{}%
  }%
  \providecommand\transparent[1]{%
    \errmessage{(Inkscape) Transparency is used (non-zero) for the text in Inkscape, but the package 'transparent.sty' is not loaded}%
    \renewcommand\transparent[1]{}%
  }%
  \providecommand\rotatebox[2]{#2}%
  \newcommand*\fsize{\dimexpr\f@size pt\relax}%
  \newcommand*\lineheight[1]{\fontsize{\fsize}{#1\fsize}\selectfont}%
  \ifx\svgwidth\undefined%
    \setlength{\unitlength}{11.51953029bp}%
    \ifx\svgscale\undefined%
      \relax%
    \else%
      \setlength{\unitlength}{\unitlength * \real{\svgscale}}%
    \fi%
  \else%
    \setlength{\unitlength}{\svgwidth}%
  \fi%
  \global\let\svgwidth\undefined%
  \global\let\svgscale\undefined%
  \makeatother%
  \begin{picture}(1,1.77482522)%
    \lineheight{1}%
    \setlength\tabcolsep{0pt}%
    \put(0,0){\includegraphics[width=\unitlength,page=1]{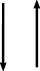}}%
  \end{picture}%
\endgroup%
}}}
\newcommand{\Crcap}{\raisebox{-6pt}{\marginbox{1.5pt 0pt}{
\begingroup%
  \makeatletter%
  \providecommand\color[2][]{%
    \errmessage{(Inkscape) Color is used for the text in Inkscape, but the package 'color.sty' is not loaded}%
    \renewcommand\color[2][]{}%
  }%
  \providecommand\transparent[1]{%
    \errmessage{(Inkscape) Transparency is used (non-zero) for the text in Inkscape, but the package 'transparent.sty' is not loaded}%
    \renewcommand\transparent[1]{}%
  }%
  \providecommand\rotatebox[2]{#2}%
  \newcommand*\fsize{\dimexpr\f@size pt\relax}%
  \newcommand*\lineheight[1]{\fontsize{\fsize}{#1\fsize}\selectfont}%
  \ifx\svgwidth\undefined%
    \setlength{\unitlength}{11.51816753bp}%
    \ifx\svgscale\undefined%
      \relax%
    \else%
      \setlength{\unitlength}{\unitlength * \real{\svgscale}}%
    \fi%
  \else%
    \setlength{\unitlength}{\svgwidth}%
  \fi%
  \global\let\svgwidth\undefined%
  \global\let\svgscale\undefined%
  \makeatother%
  \begin{picture}(1,1.77504669)%
    \lineheight{1}%
    \setlength\tabcolsep{0pt}%
    \put(0,0){\includegraphics[width=\unitlength,page=1]{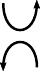}}%
  \end{picture}%
\endgroup%
}}}
\newcommand{\Crdright}{\raisebox{-6pt}{\marginbox{1.5pt 0pt}{
\begingroup%
  \makeatletter%
  \providecommand\color[2][]{%
    \errmessage{(Inkscape) Color is used for the text in Inkscape, but the package 'color.sty' is not loaded}%
    \renewcommand\color[2][]{}%
  }%
  \providecommand\transparent[1]{%
    \errmessage{(Inkscape) Transparency is used (non-zero) for the text in Inkscape, but the package 'transparent.sty' is not loaded}%
    \renewcommand\transparent[1]{}%
  }%
  \providecommand\rotatebox[2]{#2}%
  \newcommand*\fsize{\dimexpr\f@size pt\relax}%
  \newcommand*\lineheight[1]{\fontsize{\fsize}{#1\fsize}\selectfont}%
  \ifx\svgwidth\undefined%
    \setlength{\unitlength}{17.11504129bp}%
    \ifx\svgscale\undefined%
      \relax%
    \else%
      \setlength{\unitlength}{\unitlength * \real{\svgscale}}%
    \fi%
  \else%
    \setlength{\unitlength}{\svgwidth}%
  \fi%
  \global\let\svgwidth\undefined%
  \global\let\svgscale\undefined%
  \makeatother%
  \begin{picture}(1,1.20492949)%
    \lineheight{1}%
    \setlength\tabcolsep{0pt}%
    \put(0,0){\includegraphics[width=\unitlength,page=1]{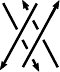}}%
  \end{picture}%
\endgroup%
}}}

\ifDRAFT
\usepackage[yyyymmdd,hhmmss]{datetime}
\usepackage{background}
\SetBgContents{DRAFT \today{} \currenttime{}}
\SetBgScale{1}
\SetBgAngle{0}
\SetBgColor{black}
\SetBgOpacity{0.5}
\SetBgPosition{current page.south}
\SetBgVshift{1.5cm}
\fi

\def\thickness{0.25mm}
\def\dthickness{0.75mm}
\usepackage{tikz}
\usetikzlibrary{shapes,positioning}
\usetikzlibrary{decorations.markings}
\newcommand{\gSide}{\begin{tikzpicture}[baseline=-0.5ex]
  \def\rad{0.3}
  \draw [gray] (0,0) circle [radius=\rad];
  \begin{scope} [radius=\rad]
    \clip (0,0) circle;
    \coordinate (A) at (45:\rad);
    \coordinate (B) at (135:\rad);
    \coordinate (C) at (225:\rad);
    \coordinate (D) at (315:\rad);
    
    \draw [line width=\thickness] (A) to[out=225,in=315] (B);
    \draw [line width=\thickness] (C) to[out=45,in=135] (D);
  \end{scope}
\end{tikzpicture}
}
\newcommand{\gUp}{\begin{tikzpicture}[baseline=-0.5ex]
  \def\rad{0.3}
  \draw [gray] (0,0) circle [radius=\rad];
  \begin{scope} [radius=\rad]
    \clip (0,0) circle;
    \coordinate (A) at (45:\rad);
    \coordinate (B) at (135:\rad);
    \coordinate (C) at (225:\rad);
    \coordinate (D) at (315:\rad);
    
    \draw [line width=\thickness] (A) to[out=225,in=135] (D);
    \draw [line width=\thickness] (B) to[out=315,in=45] (C);
  \end{scope}
\end{tikzpicture}
}

\newcommand{\gX}{\begin{tikzpicture}[baseline=-0.5ex]
  \def\rad{0.3}
  \draw [gray] (0,0) circle [radius=\rad];
  \begin{scope} [radius=\rad]
    \clip (0,0) circle;
    \coordinate (A) at (45:\rad);
    \coordinate (B) at (135:\rad);
    \coordinate (C) at (225:\rad);
    \coordinate (D) at (315:\rad);
    
    \draw [line width=\thickness] (A) to[out=225,in=45] (C);
    \draw [line width=\thickness] (B) to[out=315,in=135] (D);
    \filldraw (0,0) circle (0.7mm);
  \end{scope}
\end{tikzpicture}
}
\newcommand{\gV}{\begin{tikzpicture}[baseline=-0.5ex]
  \def\rad{0.3}
  \draw [gray] (0,0) circle [radius=\rad];
  \begin{scope} [radius=\rad]
    \clip (0,0) circle;
    \coordinate (A) at (45:\rad);
    \coordinate (B) at (135:\rad);
    \coordinate (C) at (225:\rad);
    \coordinate (D) at (315:\rad);
    
    \draw [line width=\thickness] (A) to[out=225,in=45] (C);
    \draw [line width=\thickness] (B) to[out=315,in=135] (D);
  \end{scope}
\end{tikzpicture}
}

\newcommand{\gRight}{\begin{tikzpicture}[baseline=-0.5ex]
  \def\rad{0.3}
  \draw [gray] (0,0) circle [radius=\rad];
  \begin{scope} [radius=\rad]
    \clip (0,0) circle;
    \coordinate (A) at (45:\rad);
    \coordinate (B) at (135:\rad);
    \coordinate (C) at (225:\rad);
    \coordinate (D) at (315:\rad);
    
    \draw [line width=\thickness] (B) to[out=315,in=135] (D);
    \draw [line width=\dthickness,double=black,draw=white,double distance=\thickness] (A) to[out=225,in=45] (C);
  \end{scope}
\end{tikzpicture}
}
\newcommand{\gLeft}{\begin{tikzpicture}[baseline=-0.5ex]
  \def\rad{0.3}
  \draw [gray] (0,0) circle [radius=\rad];
  \begin{scope} [radius=\rad]
    \clip (0,0) circle;
    \coordinate (A) at (45:\rad);
    \coordinate (B) at (135:\rad);
    \coordinate (C) at (225:\rad);
    \coordinate (D) at (315:\rad);
    
    \draw [line width=\thickness] (A) to[out=225,in=45] (C);
    \draw [line width=\dthickness,double=black,draw=white,double distance=\thickness] (B) to[out=315,in=135] (D);
  \end{scope}
\end{tikzpicture}
}

\begin{document}

\title{Planar diagrams for local invariants of graphs in surfaces}

\author{Calvin McPhail-Snyder}
\address{Department of Mathematics, University of California, Berkeley, California 94720-3840}
\email{cmcs@math.berkeley.edu}

\author{Kyle A. Miller}
\address{Department of Mathematics, University of California, Berkeley, California 94720-3840}
\email{kmill@math.berkeley.edu}


\date{October 2018}

\begin{abstract}
  In order to apply quantum topology methods to nonplanar graphs, we define a planar diagram category that describes the local topology of embeddings of graphs into surfaces.
  These \emph{virtual graphs} are a categorical interpretation of ribbon graphs.
  We describe an extension of the flow polynomial to virtual graphs, the $S$-polynomial, and formulate the $\mathfrak{sl}(N)$ Penrose polynomial for non-cubic graphs, giving contraction-deletion relations.
  The $S$-polynomial is used to define an extension of the Yamada polynomial to virtual spatial graphs, and with it we obtain a sufficient condition for non-classicality of virtual spatial graphs.
  We conjecture the existence of local relations for the $S$-polynomial at squares of integers. 
\end{abstract}

\maketitle
\section{Introduction}
The study of invariants of topological objects like knots and links can frequently be simplified by the use of diagram categories.
These are categories whose objects are intervals with marked points and whose morphism spaces are of diagrams drawn between them, for example of graphs or tangles.
Numerical invariants of these objects correspond to functors from these categories to simpler linear categories with local relations.
Perhaps the most famous example is the use of the Kauffman bracket (equivalently, the $\mathcal{R}$-matrix of $\mathcal{U}_q(\mathfrak{sl}_2)$) to construct the Jones polynomial.

Many of the examples in the literature correspond to diagrams drawn in the plane, such as the tangle diagrams used in Reshetikhin-Turaev invariants\cite{Reshetikhin1990}.
Similarly, in \cite{Fendley2009} Fendley and Krushkal consider a category of planar graphs whose vertices are of degree $3$ (cubic graphs).
A natural generalization is to ask what happens when the objects are nonplanar --- for example, one might be interested in invariants of graphs on tori.
The diagrams of the relevant category are no longer drawn in the plane, but instead on a compact oriented surface with boundary, as in Figure~\ref{fig:surface-graph}, and analyzing composition laws can become more difficult.
\begin{figure}[htb]
  \centering
  \includegraphics{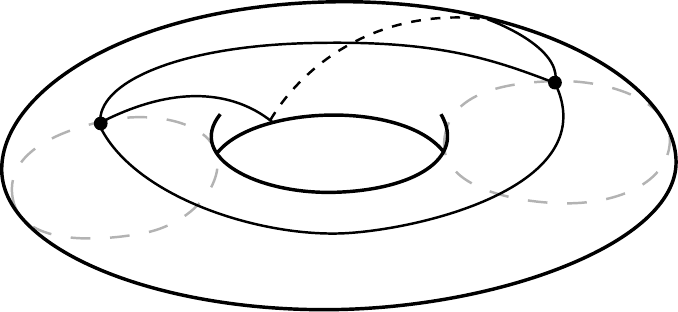}
  \caption{A theta graph $\Theta_3$ cellularly embedded in a torus.}
  \label{fig:surface-graph}
\end{figure}
In this paper, we explain how to represent the local information of graphs on surfaces in terms of planar diagrams for objects we call \emph{virtual graphs}, so named for their suitability in representing diagrams of virtual links and virtual spatial graphs.
Virtual graphs are equivalence classes that unite the notions of ribbon graphs and cellular embeddings.

When considering invariants of planar graphs or similar objects, it makes sense to adopt a functional or categorical viewpoint.
Graphs drawn in planar surfaces with boundary can have edges incident to the boundary, and one might compose diagrams by gluing boundaries while attaching corresponding boundary edges.
This idea has been formalized in a number of ways, such as planar algebras, operads, or tensor categories.

We consider a decorated $(1+1)$-dimensional oriented cobordism category: objects are closed $1$-manifolds (circles) with marked points, and morphisms are oriented cobordisms between them, decorated with embedded graphs meeting the marked boundary points in half-edges.
We refer to such graphs as \emph{surface graphs}, and in Definition~\ref{def:sym-surface-graph-category} we give a variation, the \emph{stable surface graph category}, where we impose a certain symmetric monoidal structure.

From this perspective, the data of a numerical invariant such as the flow polynomial, the Penrose polynomials, or the $S$-polynomial of Definition~\ref{def:s-polynomial} corresponds to a monoidal functor (possibly with other properties) from the surface graph category to the category of vector spaces.
Such a functor is a type of $(1+1)$-dimensional topological quantum field theory.
These theories are most naturally described in terms of diagrams drawn on surfaces, and, like the case of knots, links, and tangles, for many purposes it is easier to work with projected two-dimensional planar diagrams.
We use the viewpoint of virtual graphs to study invariants that naturally respect the symmetric monoidal structure at the graph level.

In the following, all surfaces are compact and oriented, possibly with boundary.
Everything is with respect to the piecewise linear (PL) category.

\subsection{Overview}
\label{sec:overview}

In Section~\ref{sec:virtual-graphs-introduction} we introduce \emph{virtual graphs} as a setting for invariants of ribbon graphs (combinatorial maps) and virtual spatial graph theory.
Section~\ref{sec:flow-and-s-polynomials} gives a state sum definition for the $S$-polynomial, which is an extension of the flow polynomial to virtual graphs.
We introduce a category of virtual graphs modulo edge subdivision in Section~\ref{sec:cat-virtual-graphs} and use it to produce a trace-preserving functor to the Brauer category (Section~\ref{sec:brauer-category}) that computes the $S$-polynomial in Section~\ref{sec:virtual-graph-inv-categories}.
Using this, we give formulas for edge and vertex connect sums as well as a connection between the $S$-polynomial and the flow polynomial when evaluated at $0$ and $4$.

Section~\ref{sec:penrose-polys} relates the $\mathfrak{sl}(N)$ and $\mathfrak{so}(N)$ Penrose polynomials to the $S$-polynomial, with an extension of the $\mathfrak{sl}(N)$ polynomial to non-cubic virtual graphs in Section~\ref{sec:sln-rels}.
In Section~\ref{sec:yamada-polynomial} we give an extension of the Yamada polynomial to virtual spatial graphs that when paired with Fleming and Mellor's extension can be used to tell that some virtual spatial graphs are not equivalent to classical spatial graphs, extending Miyazawa's result for virtual links.
Section~\ref{sec:golden-conj-vsg} shows that an interpretation of Agol and Krushkal's golden inequality conjecture does not hold for virtual spatial graphs.

\begin{table}[htb]
  \centering
  \caption{Correspondence between the traditional notation and the notation used in this paper, where $G$ is a combinatorial map or ribbon graph.}
  \label{tab:notation-correspondence}
  \begin{tabular}{ll}
    \toprule
    Combinatorial & Topological\\
    \midrule
    $v(G)$ & $\abs{V(G)}$\\
    $e(G)$ & $\abs{E(G)}$\\
    $k(g)$ & $b_0(G)$\\
    $r(g)$ & $\abs{V(G)}-b_0(G)$\\
    $n(G)$ & $b_1(G)$\\
    $bc(G)$ & $b_0(\partial G)=\abs{F(G)}$\\
    $k(G)+n(G)-bc(G)$ & $2g(G)$\\
    \bottomrule
  \end{tabular}
\end{table}

\section{Virtual graphs and virtual spatial graphs}
\label{sec:virtual-graphs-introduction}
In this section we define \emph{virtual graphs}, which provide a planar calculus for the local information of graphs embedded in surfaces.
The motivation for these virtual graph diagrams is to extend objects like the flow category in \cite{Agol2016} to nonplanar graphs, to define a Temperley-Lieb-like category for surfaces, and to give a setting for objects such as $3$-graphs\cite{Duzhin1998} or ($4$-valent) virtual graphs and the category $GraphCat$ of cubic graphs\cite{Kauffman2015}.
Virtual graphs are represented by objects frequently called \emph{combinatorial maps} or \emph{ribbon graphs}, but we wish to de-emphasize a particular surface embedding.

We define \emph{virtual spatial graphs} as virtual graphs with special $4$-valent vertices representing classical crossings subject to Reidemeister moves, instead of using Gauss codes as in \cite{Fleming2007}, and we review the interpretation of virtual spatial graphs as ribbon graphs embedded in thickened surfaces.

\subsection{Virtual graphs}
By a \emph{graph} we mean a compact $1$-dimensional CW complex with a given cell structure, which in other words is a finite abstract graph topologized so its vertices are points and its edges are closed unit intervals.

\begin{definition}
  A \emph{surface graph} $G\incl\Sigma$ is a graph $G$ embedded in the interior of a compact oriented surface $\Sigma$, possibly with boundary.
\end{definition}

There are two special kinds of surface graphs $G\incl\Sigma$:
\begin{enumerate}
\item
  A surface graph is a \emph{cellular} or \emph{combinatorial} embedding if $\Sigma-\nu(G)$ (the complement of a regular neighborhood) is a disjoint union of closed disks.
\item
  A surface graph is a \emph{ribbon graph} if $G\incl\Sigma$ is a homotopy equivalence.
\end{enumerate}

Cellular embeddings have, up to isotopy, a well-defined (\emph{geometric} or \emph{Poincar\'{e}}) \emph{dual} surface graph $G^*\incl\Sigma$ obtained by placing a single vertex inside each disk of $\Sigma-\nu(G)$ and joining such vertices by an edge when the corresponding disks are adjacent to the same edge of $G$.
In particular, the edges of $G$ and $G^*$ are in one-to-one correspondence, and we can arrange for each dual pair to transversely intersect at a single point.
This is the usual dual graph in the case of a graph embedded in the plane.

Cellular embeddings and ribbon graphs are intimately related.
Given a cellular embedding $G\incl\Sigma$, the restriction $G\incl\operatorname{cl}(\nu(G))$ is a ribbon graph.
Conversely, given a ribbon graph $G\incl\Sigma$, the extension $G\incl\Sigma'$ obtained by gluing disks into the boundary of $\Sigma$ to obtain a closed manifold $\Sigma'$ yields a cellular embedding.
Up to composition with an orientation-preserving surface homeomorphism, these are inverse operations.

\begin{definition}
  A \emph{stabilization move} of a surface graph $G \incl \Sigma$ is the surface graph $G\incl \Sigma'$ given by composition with an orientation-preserving embedding $\Sigma\incl\Sigma'$ into a compact oriented surface $\Sigma'$.
  \emph{Stable equivalence} is the equivalence relation on surface graphs generated by stabilization moves.  See Figure~\ref{fig:stab-example}.
\end{definition}
\begin{figure}[htb]
  \centering
\begingroup%
  \makeatletter%
  \providecommand\color[2][]{%
    \errmessage{(Inkscape) Color is used for the text in Inkscape, but the package 'color.sty' is not loaded}%
    \renewcommand\color[2][]{}%
  }%
  \providecommand\transparent[1]{%
    \errmessage{(Inkscape) Transparency is used (non-zero) for the text in Inkscape, but the package 'transparent.sty' is not loaded}%
    \renewcommand\transparent[1]{}%
  }%
  \providecommand\rotatebox[2]{#2}%
  \newcommand*\fsize{\dimexpr\f@size pt\relax}%
  \newcommand*\lineheight[1]{\fontsize{\fsize}{#1\fsize}\selectfont}%
  \ifx\svgwidth\undefined%
    \setlength{\unitlength}{240.01567751bp}%
    \ifx\svgscale\undefined%
      \relax%
    \else%
      \setlength{\unitlength}{\unitlength * \real{\svgscale}}%
    \fi%
  \else%
    \setlength{\unitlength}{\svgwidth}%
  \fi%
  \global\let\svgwidth\undefined%
  \global\let\svgscale\undefined%
  \makeatother%
  \begin{picture}(1,0.27738891)%
    \lineheight{1}%
    \setlength\tabcolsep{0pt}%
    \put(0,0){\includegraphics[width=\unitlength,page=1]{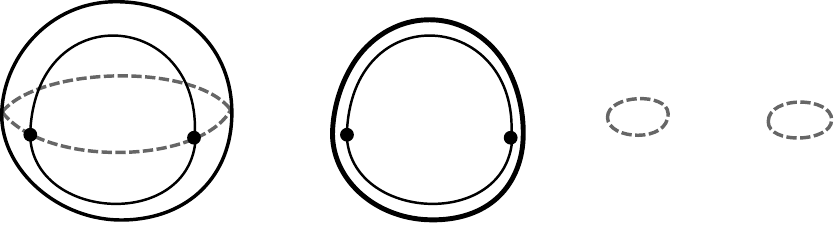}}%
    \put(0.31526062,0.13877835){\color[rgb]{0,0,0}\makebox(0,0)[lt]{\lineheight{0}\smash{\begin{tabular}[t]{l}$\sim$\end{tabular}}}}%
    \put(0,0){\includegraphics[width=\unitlength,page=2]{stabilization.pdf}}%
    \put(0.66285699,0.13877835){\color[rgb]{0,0,0}\makebox(0,0)[lt]{\lineheight{0}\smash{\begin{tabular}[t]{l}$\sim$\end{tabular}}}}%
    \put(0,0){\includegraphics[width=\unitlength,page=3]{stabilization.pdf}}%
  \end{picture}%
\endgroup%

  \caption{Example of stable equivalences, where the annulus in the middle is embedded in a sphere (left) and a torus (right).}
  \label{fig:stab-example}
\end{figure}
This definition is analogous to the one for virtual link diagrams in \cite{Carter2002}.
The closed-surface representatives are all related by $0$-surgery and $1$-surgery in the complement of $G$.
Note that ambient isotopies induce stable equivalence.

\begin{figure}[htb]
  \centering
  \includegraphics{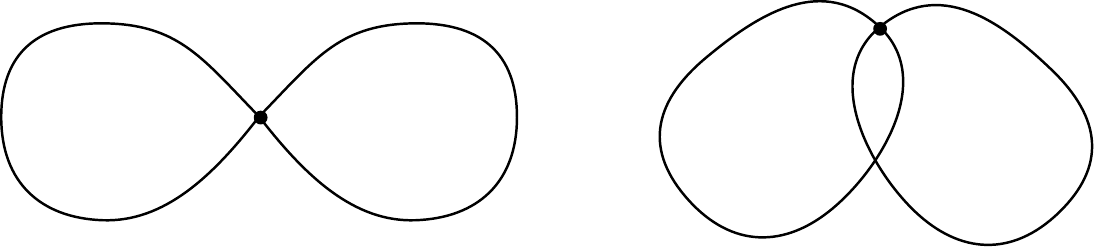}
  \caption{Planar and nonplanar virtual graphs.}
  \label{fig:virtual-graph-example}
\end{figure}

Stable equivalence allows us to ignore irrelevant topology.
For contraction-deletion-like rules, for example, stable equivalence makes edge deletion as simple as restricting to a subgraph.
\begin{definition}
  A \emph{virtual graph} is a stable equivalence class $[G\incl\Sigma]$ for a surface graph $G\incl\Sigma$.
  By abuse of notation, we refer $G$ as the virtual graph.
\end{definition}
Every virtual graph can be represented by a ribbon graph by destabilizing an arbitrary representative surface graph to a closed regular neighborhood of the graph.

The data for an arrow presentation of a ribbon graph (see Section~\ref{subsec:edge-contraction}) may also be given by the counterclockwise cyclic ordering of the half-edges of $G$ incident to each vertex, the totality of which is called a \emph{rotation system} for $G$.
Specifically, following \cite{Duzhin1998}, with $\mathcal{H}$ being a finite set (of \emph{half-edges}), let $\alpha,\sigma\in\Sym(\mathcal{H})$ be permutations such that the orbits of $\alpha$ are all of size $2$.
The orbits of $\alpha$ and $\sigma$ respectively give the edges and vertices of a graph, and $\sigma$ is its rotation system.

A virtual graph can be specified by a \emph{virtual graph diagram}, which is an immersion of a graph $G$ in the plane such that (1) no two vertices coincide, (2) no vertex coincides with the interior of an edge, and (3) edges intersect transversely.
Such intersections are called \emph{virtual crossings}.
One may imagine that the diagram portrays the image of $G$ through an orientation-preserving planar immersion of a ribbon graph representative.

The rotation system is obtained from such a diagram by reading off the half edges incident to a vertex in counterclockwise order.
Virtual graphs are in bijective correspondence with virtual graph diagrams up to isotopy, modulo the moves in Figure~\ref{fig:virt-graph-moves}.
As an example, the two virtual graphs in Figure~\ref{fig:inequivalent-theta-graphs} have the same underlying graph but different rotation systems, and hence are inequivalent.

\begin{figure}[tb]
  \centering
  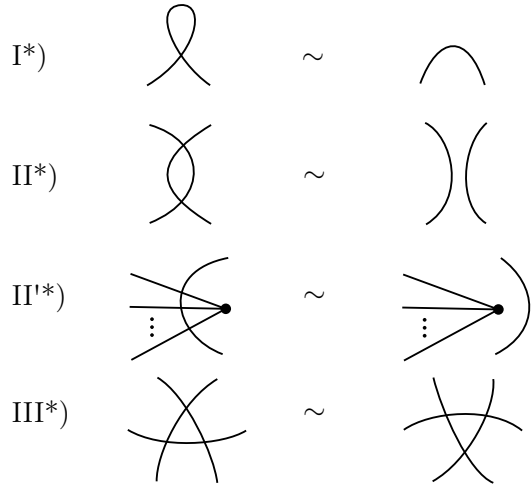
  \caption{Diagrammatic moves for virtual graph equivalence.}
  \label{fig:virt-graph-moves}
\end{figure}

\begin{figure}[tb]
  \centering
  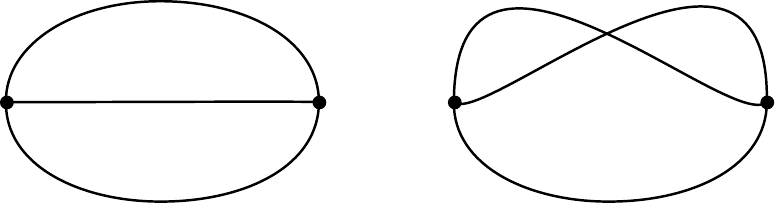
  \caption{Two inequivalent virtual graphs whose underlying graph is the \emph{theta graph} $\Theta_3$.
  The second is represented by the surface graph in Figure~\ref{fig:surface-graph}.}
  \label{fig:inequivalent-theta-graphs}
\end{figure}

\begin{figure}[tb]
  \centering
\begingroup%
  \makeatletter%
  \providecommand\color[2][]{%
    \errmessage{(Inkscape) Color is used for the text in Inkscape, but the package 'color.sty' is not loaded}%
    \renewcommand\color[2][]{}%
  }%
  \providecommand\transparent[1]{%
    \errmessage{(Inkscape) Transparency is used (non-zero) for the text in Inkscape, but the package 'transparent.sty' is not loaded}%
    \renewcommand\transparent[1]{}%
  }%
  \providecommand\rotatebox[2]{#2}%
  \newcommand*\fsize{\dimexpr\f@size pt\relax}%
  \newcommand*\lineheight[1]{\fontsize{\fsize}{#1\fsize}\selectfont}%
  \ifx\svgwidth\undefined%
    \setlength{\unitlength}{183.44884197bp}%
    \ifx\svgscale\undefined%
      \relax%
    \else%
      \setlength{\unitlength}{\unitlength * \real{\svgscale}}%
    \fi%
  \else%
    \setlength{\unitlength}{\svgwidth}%
  \fi%
  \global\let\svgwidth\undefined%
  \global\let\svgscale\undefined%
  \makeatother%
  \begin{picture}(1,0.37857322)%
    \lineheight{1}%
    \setlength\tabcolsep{0pt}%
    \put(0,0){\includegraphics[width=\unitlength,page=1]{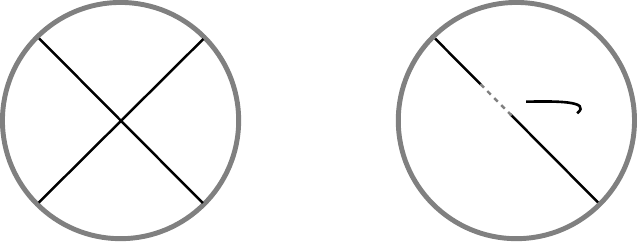}}%
    \put(0.45050188,0.17913877){\color[rgb]{0,0,0}\makebox(0,0)[lt]{\lineheight{1.25}\smash{\begin{tabular}[t]{l}$\longrightarrow$\end{tabular}}}}%
    \put(0,0){\includegraphics[width=\unitlength,page=2]{disk-to-torus.pdf}}%
  \end{picture}%
\endgroup%

  \caption{Resolving a virtual crossing by a connect sum with a torus.}
  \label{fig:disk-to-torus}
\end{figure}

One can also obtain a representative surface graph from a virtual graph diagram by a straightforward construction that preserves the rotation system, illustrated in Figure~\ref{fig:disk-to-torus}: replace a small disk neighborhood of each virtual crossing by a punctured torus through which the edges are routed so they no longer cross.

As a first invariant, the \emph{genus} of a virtual graph $G$ is
\begin{equation*}
  g(G)=\min_{G\incl\Sigma}g(\Sigma),
\end{equation*}
where $g(\Sigma)$ is the sum of the genera of the connected components if $\Sigma$ is disconnected.
In particular, if $G\incl\Sigma$ is a cellular embedding, $g(G)=g(\Sigma)$.
A \emph{planar} virtual graph is one whose genus is zero, and an abstract graph is planar if and only if it is the underlying graph of some planar virtual graph.
Virtual graphs up to move VI* in Figure~\ref{fig:vert-pliable-moves} are the same as abstract graphs.

With $G$ a virtual graph and $G\incl\Sigma$ a representative cellular embedding, the dual virtual graph $G^*$ is defined to be the dual surface graph $G^*\incl\Sigma$ as a virtual graph.
This is independent of the choice of representative.

\subsection{Virtual spatial graphs}
A \emph{spatial graph} $G$ is an embedding of a ribbon graph $G\hookrightarrow \Sigma$ in $S^3$, where $\Sigma$ is the \emph{ribbon structure} of the spatial graph.
The data for spatial graphs can be given by link diagrams extended to represent the vertices of $G$.
The ribbon structure of $G$ can be read off from such a diagram through the so-called blackboard framing, where the diagram is thought of as the image of $G$ through an orientation-preserving generic immersive projection of a ribbon graph onto the equatorial $S^2$.
(This is the most restrictive definition for a spatial graph with unoriented edges, where the diagrams are up to \emph{rigid vertex isotopy} and \emph{regular isotopy}.
The least restrictive, an embedding of the graph itself, is up to \emph{pliable vertex isotopy} and the full set of Reidemeister moves.
See \cite{Yamada1989} and \cite{Fleming2007}.)

The combinatorial data of a spatial graph diagram is a virtual graph with distinguished $4$-valent \emph{classical crossings} marked with which pair of opposite incident half edges corresponds to the over-strand.
By disregarding the planarity of the diagram, we obtain \emph{virtual spatial graphs}.
This is equivalent to the approach in \cite{Fleming2007}, which extends the original Gauss code approach for virtual links in \cite{Kauffman1999} to define virtual spatial graphs.

Virtual spatial graphs are related by the virtual graph moves in Figure~\ref{fig:virt-graph-moves}, where II\textquotesingle{}* applies to the classical crossings as well (illustrated in Figure~\ref{fig:virt-spatial-graph-moves}), along with the Reidemeister moves for framed link diagrams up to regular isotopy in Figure~\ref{fig:regular-isotopy} and the moves for rigid vertex isotopy in Figure~\ref{fig:vert-crossing-moves}.

\begin{figure}[htb]
  \centering
\begingroup%
  \makeatletter%
  \providecommand\color[2][]{%
    \errmessage{(Inkscape) Color is used for the text in Inkscape, but the package 'color.sty' is not loaded}%
    \renewcommand\color[2][]{}%
  }%
  \providecommand\transparent[1]{%
    \errmessage{(Inkscape) Transparency is used (non-zero) for the text in Inkscape, but the package 'transparent.sty' is not loaded}%
    \renewcommand\transparent[1]{}%
  }%
  \providecommand\rotatebox[2]{#2}%
  \newcommand*\fsize{\dimexpr\f@size pt\relax}%
  \newcommand*\lineheight[1]{\fontsize{\fsize}{#1\fsize}\selectfont}%
  \ifx\svgwidth\undefined%
    \setlength{\unitlength}{347.1572154bp}%
    \ifx\svgscale\undefined%
      \relax%
    \else%
      \setlength{\unitlength}{\unitlength * \real{\svgscale}}%
    \fi%
  \else%
    \setlength{\unitlength}{\svgwidth}%
  \fi%
  \global\let\svgwidth\undefined%
  \global\let\svgscale\undefined%
  \makeatother%
  \begin{picture}(1,0.10668493)%
    \lineheight{1}%
    \setlength\tabcolsep{0pt}%
    \put(0.16338226,0.05805904){\color[rgb]{0,0,0}\makebox(0,0)[lt]{\lineheight{0}\smash{\begin{tabular}[t]{l}$\sim$\end{tabular}}}}%
    \put(0,0){\includegraphics[width=\unitlength,page=1]{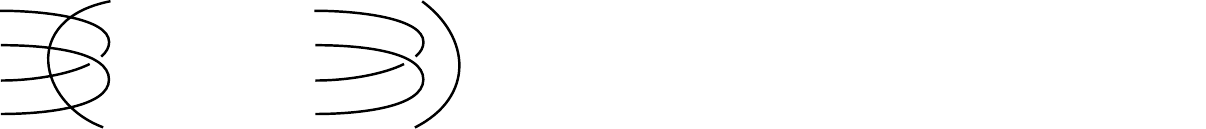}}%
    \put(0.78125635,0.05805904){\color[rgb]{0,0,0}\makebox(0,0)[lt]{\lineheight{0}\smash{\begin{tabular}[t]{l}$\sim$\end{tabular}}}}%
    \put(0,0){\includegraphics[width=\unitlength,page=2]{virtual-moves-classical.pdf}}%
  \end{picture}%
\endgroup%

  \caption{How move II\textquotesingle{}* applies to classical crossings.}
  \label{fig:virt-spatial-graph-moves}
\end{figure}

\begin{figure}[htb]
  \centering
\begingroup%
  \makeatletter%
  \providecommand\color[2][]{%
    \errmessage{(Inkscape) Color is used for the text in Inkscape, but the package 'color.sty' is not loaded}%
    \renewcommand\color[2][]{}%
  }%
  \providecommand\transparent[1]{%
    \errmessage{(Inkscape) Transparency is used (non-zero) for the text in Inkscape, but the package 'transparent.sty' is not loaded}%
    \renewcommand\transparent[1]{}%
  }%
  \providecommand\rotatebox[2]{#2}%
  \newcommand*\fsize{\dimexpr\f@size pt\relax}%
  \newcommand*\lineheight[1]{\fontsize{\fsize}{#1\fsize}\selectfont}%
  \ifx\svgwidth\undefined%
    \setlength{\unitlength}{192.16763985bp}%
    \ifx\svgscale\undefined%
      \relax%
    \else%
      \setlength{\unitlength}{\unitlength * \real{\svgscale}}%
    \fi%
  \else%
    \setlength{\unitlength}{\svgwidth}%
  \fi%
  \global\let\svgwidth\undefined%
  \global\let\svgscale\undefined%
  \makeatother%
  \begin{picture}(1,0.70213275)%
    \lineheight{1}%
    \setlength\tabcolsep{0pt}%
    \put(-0.00204289,0.58855716){\color[rgb]{0,0,0}\makebox(0,0)[lt]{\lineheight{0}\smash{\begin{tabular}[t]{l}I\textquotesingle)\end{tabular}}}}%
    \put(-0.00204289,0.35345118){\color[rgb]{0,0,0}\makebox(0,0)[lt]{\lineheight{0}\smash{\begin{tabular}[t]{l}II)\end{tabular}}}}%
    \put(-0.00204289,0.11834532){\color[rgb]{0,0,0}\makebox(0,0)[lt]{\lineheight{0}\smash{\begin{tabular}[t]{l}III)\end{tabular}}}}%
    \put(0,0){\includegraphics[width=\unitlength,page=1]{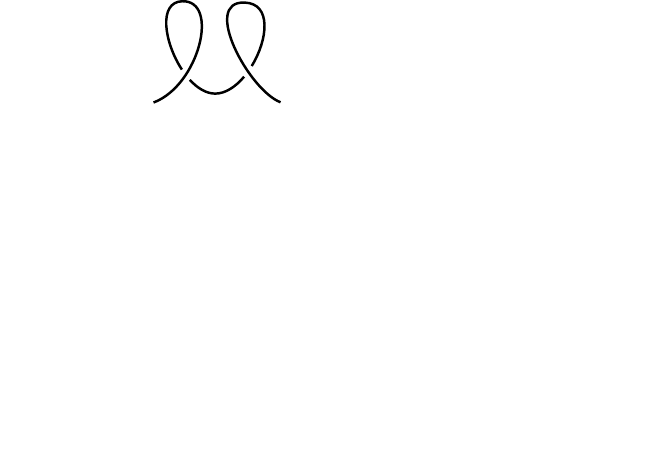}}%
    \put(0.56702644,0.59114367){\color[rgb]{0,0,0}\makebox(0,0)[lt]{\lineheight{0}\smash{\begin{tabular}[t]{l}$\sim$\end{tabular}}}}%
    \put(0,0){\includegraphics[width=\unitlength,page=2]{regular-isotopy-moves.pdf}}%
    \put(0.5694797,0.35963156){\color[rgb]{0,0,0}\makebox(0,0)[lt]{\lineheight{0}\smash{\begin{tabular}[t]{l}$\sim$\end{tabular}}}}%
    \put(0.5694797,0.12406687){\color[rgb]{0,0,0}\makebox(0,0)[lt]{\lineheight{0}\smash{\begin{tabular}[t]{l}$\sim$\end{tabular}}}}%
    \put(0,0){\includegraphics[width=\unitlength,page=3]{regular-isotopy-moves.pdf}}%
  \end{picture}%
\endgroup%

  \caption{Reidemeister moves for framed link diagrams up to regular isotopy.}
  \label{fig:regular-isotopy}
\end{figure}

\begin{figure}[htb]
  \centering
\begingroup%
  \makeatletter%
  \providecommand\color[2][]{%
    \errmessage{(Inkscape) Color is used for the text in Inkscape, but the package 'color.sty' is not loaded}%
    \renewcommand\color[2][]{}%
  }%
  \providecommand\transparent[1]{%
    \errmessage{(Inkscape) Transparency is used (non-zero) for the text in Inkscape, but the package 'transparent.sty' is not loaded}%
    \renewcommand\transparent[1]{}%
  }%
  \providecommand\rotatebox[2]{#2}%
  \newcommand*\fsize{\dimexpr\f@size pt\relax}%
  \newcommand*\lineheight[1]{\fontsize{\fsize}{#1\fsize}\selectfont}%
  \ifx\svgwidth\undefined%
    \setlength{\unitlength}{286.84943667bp}%
    \ifx\svgscale\undefined%
      \relax%
    \else%
      \setlength{\unitlength}{\unitlength * \real{\svgscale}}%
    \fi%
  \else%
    \setlength{\unitlength}{\svgwidth}%
  \fi%
  \global\let\svgwidth\undefined%
  \global\let\svgscale\undefined%
  \makeatother%
  \begin{picture}(1,0.13705821)%
    \lineheight{1}%
    \setlength\tabcolsep{0pt}%
    \put(-0.00273717,0.0645559){\color[rgb]{0,0,0}\makebox(0,0)[lt]{\lineheight{0}\smash{\begin{tabular}[t]{l}IV)\end{tabular}}}}%
    \put(0.37849666,0.06628867){\color[rgb]{0,0,0}\makebox(0,0)[lt]{\lineheight{0}\smash{\begin{tabular}[t]{l}$\sim$\end{tabular}}}}%
    \put(0,0){\includegraphics[width=\unitlength,page=1]{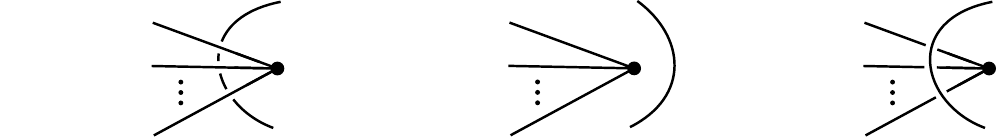}}%
    \put(0.75091701,0.06628867){\color[rgb]{0,0,0}\makebox(0,0)[lt]{\lineheight{0}\smash{\begin{tabular}[t]{l}$\sim$\end{tabular}}}}%
  \end{picture}%
\endgroup%

  \caption{Moves for rigid vertex isotopy.}
  \label{fig:vert-crossing-moves}
\end{figure}
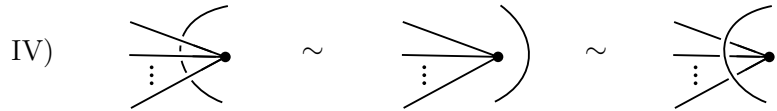

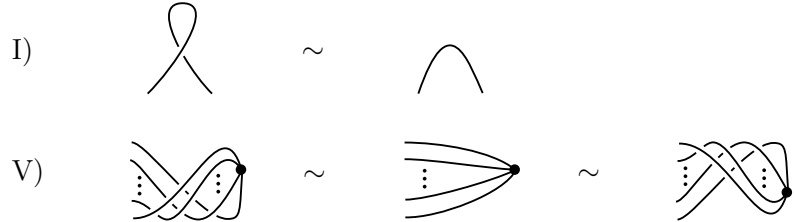
\begin{figure}[htb]
  \centering
\begingroup%
  \makeatletter%
  \providecommand\color[2][]{%
    \errmessage{(Inkscape) Color is used for the text in Inkscape, but the package 'color.sty' is not loaded}%
    \renewcommand\color[2][]{}%
  }%
  \providecommand\transparent[1]{%
    \errmessage{(Inkscape) Transparency is used (non-zero) for the text in Inkscape, but the package 'transparent.sty' is not loaded}%
    \renewcommand\transparent[1]{}%
  }%
  \providecommand\rotatebox[2]{#2}%
  \newcommand*\fsize{\dimexpr\f@size pt\relax}%
  \newcommand*\lineheight[1]{\fontsize{\fsize}{#1\fsize}\selectfont}%
  \ifx\svgwidth\undefined%
    \setlength{\unitlength}{294.0252283bp}%
    \ifx\svgscale\undefined%
      \relax%
    \else%
      \setlength{\unitlength}{\unitlength * \real{\svgscale}}%
    \fi%
  \else%
    \setlength{\unitlength}{\svgwidth}%
  \fi%
  \global\let\svgwidth\undefined%
  \global\let\svgscale\undefined%
  \makeatother%
  \begin{picture}(1,0.27985358)%
    \lineheight{1}%
    \setlength\tabcolsep{0pt}%
    \put(-0.00021257,0.20714151){\color[rgb]{0,0,0}\makebox(0,0)[lt]{\lineheight{1.25}\smash{\begin{tabular}[t]{l}I)\end{tabular}}}}%
    \put(-0.00021257,0.0521216){\color[rgb]{0,0,0}\makebox(0,0)[lt]{\lineheight{1.25}\smash{\begin{tabular}[t]{l}V)\end{tabular}}}}%
    \put(0.37171712,0.208832){\color[rgb]{0,0,0}\makebox(0,0)[lt]{\lineheight{1.25}\smash{\begin{tabular}[t]{l}$\sim$\end{tabular}}}}%
    \put(0.37332051,0.05071922){\color[rgb]{0,0,0}\makebox(0,0)[lt]{\lineheight{1.25}\smash{\begin{tabular}[t]{l}$\sim$\end{tabular}}}}%
    \put(0.7248452,0.05141107){\color[rgb]{0,0,0}\makebox(0,0)[lt]{\lineheight{1.25}\smash{\begin{tabular}[t]{l}$\sim$\end{tabular}}}}%
    \put(0,0){\includegraphics[width=\unitlength,page=1]{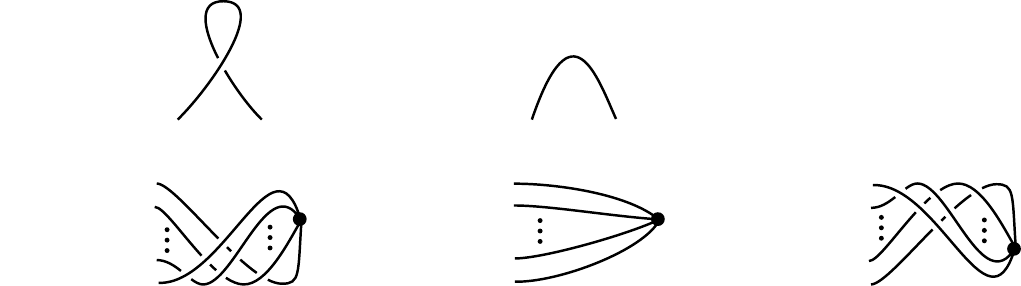}}%
  \end{picture}%
\endgroup%

  \caption{Moves for rigid vertex isotopy for \emph{flat vertex graphs}, spatial graphs whose edges are not framed and whose vertices are unoriented disks.}
  \label{fig:vert-crossing-moves-unframed}
\end{figure}

\begin{figure}[htb]
  \centering
\begingroup%
  \makeatletter%
  \providecommand\color[2][]{%
    \errmessage{(Inkscape) Color is used for the text in Inkscape, but the package 'color.sty' is not loaded}%
    \renewcommand\color[2][]{}%
  }%
  \providecommand\transparent[1]{%
    \errmessage{(Inkscape) Transparency is used (non-zero) for the text in Inkscape, but the package 'transparent.sty' is not loaded}%
    \renewcommand\transparent[1]{}%
  }%
  \providecommand\rotatebox[2]{#2}%
  \newcommand*\fsize{\dimexpr\f@size pt\relax}%
  \newcommand*\lineheight[1]{\fontsize{\fsize}{#1\fsize}\selectfont}%
  \ifx\svgwidth\undefined%
    \setlength{\unitlength}{298.6959752bp}%
    \ifx\svgscale\undefined%
      \relax%
    \else%
      \setlength{\unitlength}{\unitlength * \real{\svgscale}}%
    \fi%
  \else%
    \setlength{\unitlength}{\svgwidth}%
  \fi%
  \global\let\svgwidth\undefined%
  \global\let\svgscale\undefined%
  \makeatother%
  \begin{picture}(1,0.28177008)%
    \lineheight{1}%
    \setlength\tabcolsep{0pt}%
    \put(-0.00020924,0.23274579){\color[rgb]{0,0,0}\makebox(0,0)[lt]{\lineheight{0}\smash{\begin{tabular}[t]{l}VI)\end{tabular}}}}%
    \put(-0.00020924,0.06006261){\color[rgb]{0,0,0}\makebox(0,0)[lt]{\lineheight{0}\smash{\begin{tabular}[t]{l}VI*)\end{tabular}}}}%
    \put(0.36590453,0.23440984){\color[rgb]{0,0,0}\makebox(0,0)[lt]{\lineheight{0}\smash{\begin{tabular}[t]{l}$\sim$\end{tabular}}}}%
    \put(0,0){\includegraphics[width=\unitlength,page=1]{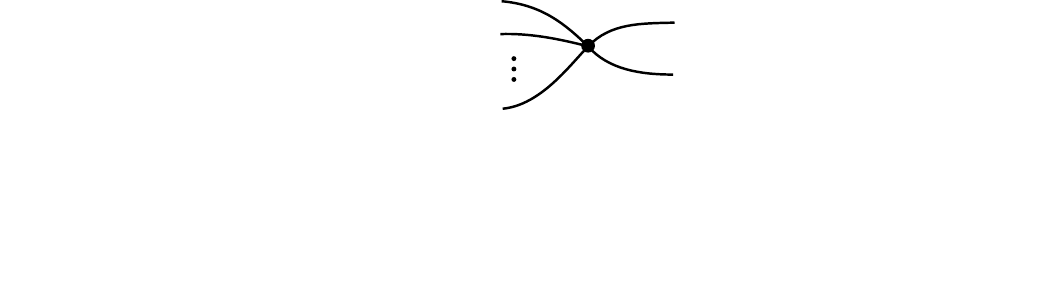}}%
    \put(0.72355437,0.23440984){\color[rgb]{0,0,0}\makebox(0,0)[lt]{\lineheight{0}\smash{\begin{tabular}[t]{l}$\sim$\end{tabular}}}}%
    \put(0,0){\includegraphics[width=\unitlength,page=2]{pliable-moves.pdf}}%
    \put(0.36590453,0.05864567){\color[rgb]{0,0,0}\makebox(0,0)[lt]{\lineheight{0}\smash{\begin{tabular}[t]{l}$\sim$\end{tabular}}}}%
    \put(0,0){\includegraphics[width=\unitlength,page=3]{pliable-moves.pdf}}%
  \end{picture}%
\endgroup%

  \caption{Moves for (VI) pliable vertex isotopy and (VI*) virtual pliable vertex isotopy.}
  \label{fig:vert-pliable-moves}
\end{figure}

Being a virtual graph, a virtual spatial graph can be represented as a surface graph $G\incl\Sigma$ with distinguished $4$-valent vertices for classical crossings --- that is, as a spatial graph diagram on a compact surface.
By thickening the surface $\Sigma$ and teasing apart the classical crossings, we may obtain a ribbon graph embedding $G\incl \Sigma\times I$ from the blackboard framing, where $I\defeq[0,1]$.
The definition of stable equivalence extends to thickened surfaces by requiring that the stabilization move $\Sigma\times I\incl \Sigma'\times I$ be induced from a map $\Sigma\incl \Sigma'$.

\begin{theorem}
  \label{thm:vsg-char}
  Virtual spatial graphs are in bijective correspondence with ribbon graphs in thickened closed oriented surfaces modulo stable equivalence.  Furthermore, each has a unique representative in a thickened closed surface (possibly disconnected) of minimal genus, up to stable equivalence induced by orientation-preserving self-homeomorphism of the surface.
\end{theorem}
\begin{proof}
  The fact that virtual spatial graphs correspond to spatial graph diagrams on compact surfaces up to stable equivalence is a result of Carter, Kamada, and Saito \cite{Carter2002}.
  Kuperberg applied JSJ theory to characterize the unique minimal representative, where minimality is in the sense that every properly embedded annulus in the graph complement that connects the two boundary components of the thickened surface bounds a ball \cite{Kuperberg2003}.
\end{proof}

In particular, once a virtual graph is represented as a diagram on a minimal-genus closed surface, equivalence is completely generated by the moves for regular isotopy and rigid vertex isotopy, along with the action by the mapping class group.
The \emph{virtual genus} of a virtual spatial graph is the genus of this surface.
A consequence of Theorem~\ref{thm:vsg-char} is that if two spatial graphs are equivalent as virtual spatial graphs, they are indeed equivalent as spatial graphs.

\begin{corollary}
  \label{cor:virtual-graph-genus}
  Let $G$ be a virtual graph.
  The virtual genus of $G$ as a virtual spatial graph equals the genus of $G$.
\end{corollary}
\begin{proof}
  Consider a cellular embedding $G\incl\Sigma$ of a nonplanar virtual graph and an extension to $G\incl\Sigma\times I$ by composition with $x\mapsto (x,\frac{1}{2})$.
  Let $A\subset\Sigma\times I$ be a properly embedded annulus in the complement of $G$ connecting the two boundary components of $\Sigma\times I$.
  Consider the intersection between $A$ and $\Sigma_0=\Sigma\times\frac{1}{2}$, which is a collection of disjoint circles.
  Each circle must bound a disk in $\Sigma_0-G$ due to it being from a cellular embedding.
  Of the intersection circles that bound a disk in $A$, take the innermost.
  We can remove at least this intersection by isotoping $\Sigma_0$ along a ball bounded by this disk and one in $\Sigma_0-G$, where the ball exists by the irreducibility of $\Sigma\times I$.
  Hence, we may assume $A-\Sigma_0$ is a union of annuli $A_1,\dots,A_n$, alternating sides of $\Sigma_0$, with $A_i\cap A_{i+1}$ bounding a disk $D_i\subset \Sigma_0-G$ for $1\leq i<n$.
  $A_1\cup D_1$ and $A_n\cup D_{n-1}$ are disks incident to the boundary, and hence bound disks $D_0,D_n\subset\Sigma\times\partial I$, respectively.
  Each sphere $D_{i-1}\cup A_i\cup D_i$ bounds a ball not intersecting $G$, the union of which is a ball that is one of the components of $\Sigma\times I-A$.
\end{proof}
\begin{remark}
  We give an algebraic proof of a weaker statement of this corollary in Remark~\ref{rem:alg-virt-graph-genus}.  
\end{remark}

\section{Two invariants of virtual graphs}
\label{sec:flow-and-s-polynomials}
In this section we define two closely related polynomial invariants of virtual graphs.
These invariants will later be considered in a more general context in Section~\ref{sec:virtual-graph-inv-categories}, but for sake of a concrete definition we discuss them on their own first.

\subsection{The flow polynomial}
Since virtual graphs are abstract graphs equipped with additional data, abstract graph invariants such as the Tutte-Whitney polynomial --- and specializations like the chromatic polynomial and the flow polynomial --- are also invariants of virtual graphs up to move VI*.
Recall the definition of the flow polynomial:
\begin{definition}
  \label{def:flow-polynomial}
  Let $G$ be a graph and $Q$ an indeterminate.
  The \emph{flow polynomial} is given by the state sum formula
  \begin{equation*}
    F_G(Q) = \sum_{T \subset E(G)} (-1)^{\abs{T}} Q^{b_1(G-T)}.
  \end{equation*}
  Here $b_1(G-T) = \rank H_1(G-T)$ is the first Betti number.
  Each ``state'' is a choice of which edges to exclude.
\end{definition}

The combinatorial interpretation of the flow polynomial is that it counts the number of nowhere-zero $Q$-flows of a graph:
\begin{definition}
  Let $G$ be a graph and let $Q$ be a nonnegative integer.
  Fix an arbitrary orientation of the edges of $G$.
  A \emph{nowhere-zero $Q$-flow} of $G$ is a coloring of its edges by nonzero elements of $\Z/Q\Z$ satisfying Kirchhoff's Law: the signed sum (incoming sum minus outgoing sum) of the colorings at each vertex is zero.
  That is, a nowhere-zero $Q$-flow is a simplicial $1$-cycle with coefficients in $\Z/Q\Z$, all of whose coefficients are nonzero.
\end{definition}
A particular simplicial $1$-cycle can be economically and freely determined by giving its coefficients at only the edges outside a spanning tree of the graph, the number of which is counted by the first Betti number.
The state sum can be viewed as an application of the inclusion-exclusion principle, where excluding a set $T$ of edges is the same as forcing those edges to have zero for their coefficients, and the remaining degrees of freedom for such a simplicial $1$-cycle is $b_1(G-T)$.
Alternatively, one could instead check that both the state sum and the combinatorial interpretations are invariant under the local relations in Definition~\ref{def:flow-category}.

One motivation for counting nowhere-zero $Q$-flows is that they are dual to $Q$-colorings: If $G$ is a connected planar graph with dual graph $G^*$, then $F_G(Q)=Q^{-1}\chi_{G^*}(Q)$, where $\chi_{G^*}(Q)$ is the chromatic polynomial of the dual graph $G^*$.

\subsection{The \texorpdfstring{$S$}{S}-polynomial}
We now define a graph polynomial closely related to the flow polynomial.
\begin{definition}
  \label{def:s-polynomial}
  Let $G$ be a virtual graph and $Q$ an indeterminate.
  The \emph{$S$-polynomial} is given by the state sum formula
  \begin{equation*}
    S_G(Q) = \sum_{T \subset E(G)} (-1)^{\abs{T}} Q^{b_1(G-T) - g(G-T)}.
  \end{equation*}
\end{definition}
The only difference from the flow polynomial is the $g(G-T)$ in the exponent, which incorporates topological information about the rotation system of the virtual graph.
By construction, the $S$-polynomial is an invariant of virtual graphs.
It is immediate from this state sum definition that $S_G(Q)=F_G(Q)$ in the case that $G$ is a planar virtual graph, which is a way that the $S$-polynomial is an extension of the flow polynomial.

The motivations for the $S$-polynomial are discussed in more detail in Section \ref{sec:virtual-graph-inv-categories}.
Primarily, this invariant is the result of applying the functor in \cite{Fendley2009} to graphs on arbitrary surfaces.
Secondarily, unlike the flow polynomial it is sensitive to the rotation system and can, for instance, distinguish the graphs in Figure~\ref{fig:inequivalent-theta-graphs}.
This gives a new extension of the Yamada polynomial in Section~\ref{sec:yamada-polynomial} to virtual spatial graphs.

The $S$-polynomial is a specialization of the Krushkal polynomial reformulated for cellular embeddings, the definition of which is given after the following.
Given a surface graph $G\incl\Sigma$, the \emph{complementary genus} $g^\perp(G)$ (with $\Sigma$ implicit) is the genus of the surface $\Sigma-\nu(G)$.
If $G\incl\Sigma$ is a cellular embedding, then recall that the dual graph $G^*\incl\Sigma$ is from reversing the roles of vertices and faces, with the edges in $G$ and $G^*$ coming in dual pairs.
If $T\subset E(G)$, then $g^\perp(G-T)$ is the genus of the induced subgraph of $G^*$ whose edges are the edges dual to those in $T$.
\begin{definition}
  The \emph{Krushkal polynomial} when reformulated for a cellular embedding $G\incl\Sigma$ is given by the following state sum formula {\cite[section 4.2]{Krushkal2011}}:
  \begin{equation*}
    P'_{G,\Sigma}(X,Y,A,B) = \sum_{T\subset E(G)}X^{b_0(G-T)-b_0(G)} Y^{b_1(G-T)} A^{2g(G-T)} B^{2g^\perp(G-T)}.
  \end{equation*}
  With $G$ thought of as a virtual graph, define $P'_{G}=P'_{G,\Sigma}$.
\end{definition}

\begin{theorem}
  For $G$ a virtual graph,
  \begin{equation*}
    S_G(Q) = (-1)^{b_1(G)} P'_G(-1,-Q,Q^{-1/2},1).
  \end{equation*}
\end{theorem}
\begin{proof}
  Observe that
  \begin{align*}
    P'_G(-1,-Q,Q^{-1/2},1) &= \sum_{T \subset E(G)} (-1)^{b_0(G-T) - b_0(G)} (-Q)^{b_1(G-T)} Q^{-g(G-T)}\\
    &=\sum_{T\subset E(G)} (-1)^{b_0(G-T)-b_1(G-T)-b_0(G)} Q^{b_1(G-T)-g(G-T)}.
  \end{align*}
  Using the Euler characteristics of $G-T$ and $G$ as graphs,
  \begin{align*}
    b_0(G-T) - b_1(G-T) &= \abs{V(G-T)} - \abs{E(G-T)}\\
                        &= \abs{V(G)}-\abs{E(G)} + \abs{T}\\
                        &= b_0(G)-b_1(G)+\abs{T}.
  \end{align*}
  Thus,
  \begin{align*}
    P'_G(-1,-Q,Q^{-1/2},1) &= \sum_{T\subset E(G)} (-1)^{-b_1(G)+\abs{T}}Q^{b_1(G-T)-g(G-T)}.
  \end{align*}
  We can now pull out the overall factor of $(-1)^{b_1(G)}$ from the sum to give the desired formula.
\end{proof}

\section{Category of virtual graphs}
\label{sec:cat-virtual-graphs}

In this section, we describe a category for virtual graphs up to edge subdivision.
This is more generally a planar algebra, however describing the category is sufficient. 

\begin{definition}
  \label{def:virtual-graph-category}
  The \emph{virtual graph category} $\mathsf{VG}^R$ over the ring $R$ is a monoidal category whose objects are finite ordered sets $[n]=\{0,1,\dots,n-1\}$ for $n\in\N$ and whose morphism sets $\mathsf{VG}^R([m],[n])$ are formal $R$-linear combinations of virtual graph diagrams, modulo edge subdivision, on oriented disks with $m+n$ marked $1$-valent vertices along the boundary.
  Edges incident to the boundary are called \emph{boundary edges}.
  Non-boundary edges are called \emph{internal edges}.
  We imagine the $[m]$ points to be at the bottom of a diagram and the $[n]$ points to be at the top.
  
  Composition $\mathsf{VG}^R([n],[\ell])\times\mathsf{VG}^R([m],[n])\to\mathsf{VG}^R([m],[\ell])$ is defined on individual virtual graphs by attaching the disks in an orientation-respecting way along a pair of arcs such that vertices $0,1,\dots,n-1$ at the top of the second virtual graph are glued to the respective vertices $0,1,\dots,n-1$ at the bottom of the first.
  The composition extends by linearity.

  Monoidal composition is by the usual horizontal gluing along a pair of arcs disjoint from the marked points on the sides of the diagrams.
\end{definition}

Ignoring edge subdivision and requiring there to be edges incident to the boundary vertices together allow for an identity in $\mathsf{VG}^R([n],[n])$, namely the graph with $n$ paths connecting vertex $i$ at the bottom to vertex $i$ at the top, for all $0\leq i <n$.
Virtual crossings satisfy move II*, so one may view a virtual crossing between two edges as a transposition, giving an embedding of the symmetric group ring $R[S_n]$ into $\mathsf{VG}^R([n],[n])$.
This structure makes $\mathsf{VG}^R$ a symmetric monoidal category.
One could compare the case of a braided monoidal category, where the braiding is a crossing that is not necessarily an involution.

For now we will focus on virtual graphs, but there is a similar category $\mathsf{VSG}^R$ of virtual spatial graphs, extending $\mathsf{VG}^R$ with special marked $4$-valent vertices called \emph{classical crossings}.

The virtual graph category can be viewed as a planar diagram model for the category of graphs in surfaces with imposed symmetric monoidal structure at the level of graphs.

\begin{definition}
  \label{def:surface-graph-category}
  The \emph{surface graph category} over a ring $R$ is a monoidal category whose objects are closed oriented $1$-manifolds with finitely many marked points, and whose morphisms are formal $R$-linear combinations of oriented cobordisms with embedded surface graphs up to edge subdivision, intersecting the marked points transversely, with the cobordisms up to homeomorphism.
\end{definition}

This is a symmetric monoidal category, with disjoint union being the monoidal product.
The category satisfies laws similar to that of a Frobenius algebra.
The full subcategory whose objects are $1$-manifolds with exactly one marked point per connected component is a symmetric monoidal category, where an element of the symmetric group may be realized as a cobordism consisting of disjoint cylinders, each with an embedded path.
The morphism sets of the surface graph category can be embedded in this subcategory as bimodules over symmetric groups by composition with the pants maps $p_{a,b,0}$ and their duals $p^*_{a,b,0}$ from Figure~\ref{fig:surface-graph-cat-pants-map}.
If we wished to impose a condition on the category for these embeddings to be functorial, then we would require that $p_{a,b,0}\circ p^*_{a,b,0}=\id$, which permits $1$-surgeries that do not disconnect the surface.
If in addition we want $p^*_{a,b,0}\circ p_{a,b,0}=\id$, then arbitrary $1$-surgeries may be performed.
See Figure~\ref{fig:surface-graph-cat-pants-map-iso} for these compositions.
Hence, imposing that $p_{a,b,0}$ be an isomorphism generates stable equivalence.

\begin{figure}[htb]
  \centering
\begingroup%
  \makeatletter%
  \providecommand\color[2][]{%
    \errmessage{(Inkscape) Color is used for the text in Inkscape, but the package 'color.sty' is not loaded}%
    \renewcommand\color[2][]{}%
  }%
  \providecommand\transparent[1]{%
    \errmessage{(Inkscape) Transparency is used (non-zero) for the text in Inkscape, but the package 'transparent.sty' is not loaded}%
    \renewcommand\transparent[1]{}%
  }%
  \providecommand\rotatebox[2]{#2}%
  \newcommand*\fsize{\dimexpr\f@size pt\relax}%
  \newcommand*\lineheight[1]{\fontsize{\fsize}{#1\fsize}\selectfont}%
  \ifx\svgwidth\undefined%
    \setlength{\unitlength}{83.23028855bp}%
    \ifx\svgscale\undefined%
      \relax%
    \else%
      \setlength{\unitlength}{\unitlength * \real{\svgscale}}%
    \fi%
  \else%
    \setlength{\unitlength}{\svgwidth}%
  \fi%
  \global\let\svgwidth\undefined%
  \global\let\svgscale\undefined%
  \makeatother%
  \begin{picture}(1,0.82196687)%
    \lineheight{1}%
    \setlength\tabcolsep{0pt}%
    \put(0,0){\includegraphics[width=\unitlength,page=1]{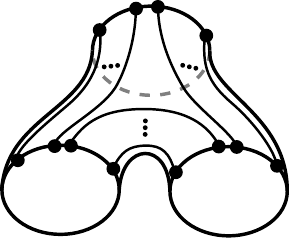}}%
    \put(0.35842239,0.63473432){\color[rgb]{0,0,0}\makebox(0,0)[lt]{\lineheight{1.25}\smash{\begin{tabular}[t]{l}$a$\end{tabular}}}}%
    \put(0.61041255,0.63325199){\color[rgb]{0,0,0}\makebox(0,0)[lt]{\lineheight{1.25}\smash{\begin{tabular}[t]{l}$b$\end{tabular}}}}%
    \put(0.52740418,0.35161587){\color[rgb]{0,0,0}\makebox(0,0)[lt]{\lineheight{1.25}\smash{\begin{tabular}[t]{l}$c$\end{tabular}}}}%
  \end{picture}%
\endgroup%

  \caption{Pants map $p_{a,b,c}$ in the surface graph category from a disjoint union of circles with $a+c$ and $b+c$ marked points each to a circle with $a+b$ marked points.}
  \label{fig:surface-graph-cat-pants-map}
\end{figure}
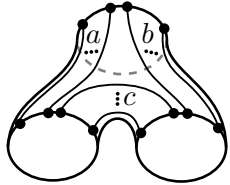

\begin{figure}[htb]
  \centering
\begingroup%
  \makeatletter%
  \providecommand\color[2][]{%
    \errmessage{(Inkscape) Color is used for the text in Inkscape, but the package 'color.sty' is not loaded}%
    \renewcommand\color[2][]{}%
  }%
  \providecommand\transparent[1]{%
    \errmessage{(Inkscape) Transparency is used (non-zero) for the text in Inkscape, but the package 'transparent.sty' is not loaded}%
    \renewcommand\transparent[1]{}%
  }%
  \providecommand\rotatebox[2]{#2}%
  \newcommand*\fsize{\dimexpr\f@size pt\relax}%
  \newcommand*\lineheight[1]{\fontsize{\fsize}{#1\fsize}\selectfont}%
  \ifx\svgwidth\undefined%
    \setlength{\unitlength}{191.33578552bp}%
    \ifx\svgscale\undefined%
      \relax%
    \else%
      \setlength{\unitlength}{\unitlength * \real{\svgscale}}%
    \fi%
  \else%
    \setlength{\unitlength}{\svgwidth}%
  \fi%
  \global\let\svgwidth\undefined%
  \global\let\svgscale\undefined%
  \makeatother%
  \begin{picture}(1,0.55794853)%
    \lineheight{1}%
    \setlength\tabcolsep{0pt}%
    \put(0,0){\includegraphics[width=\unitlength,page=1]{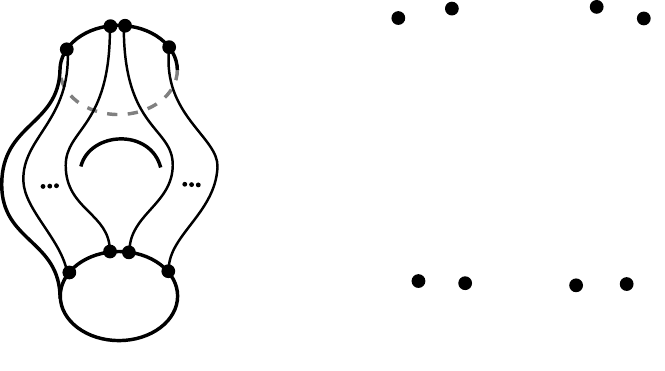}}%
    \put(0.05806966,0.30305388){\color[rgb]{0,0,0}\makebox(0,0)[lt]{\lineheight{1.25}\smash{\begin{tabular}[t]{l}$a$\end{tabular}}}}%
    \put(0.27729905,0.30305388){\color[rgb]{0,0,0}\makebox(0,0)[lt]{\lineheight{1.25}\smash{\begin{tabular}[t]{l}$b$\end{tabular}}}}%
    \put(0,0){\includegraphics[width=\unitlength,page=2]{pants-map-iso.pdf}}%
    \put(0.73768108,0.31377905){\color[rgb]{0,0,0}\makebox(0,0)[lt]{\lineheight{1.25}\smash{\begin{tabular}[t]{l}$a$\end{tabular}}}}%
    \put(0,0){\includegraphics[width=\unitlength,page=3]{pants-map-iso.pdf}}%
    \put(0.81634569,0.31248943){\color[rgb]{0,0,0}\makebox(0,0)[lt]{\lineheight{1.25}\smash{\begin{tabular}[t]{l}$b$\end{tabular}}}}%
  \end{picture}%
\endgroup%

  \caption{The compositions $p_{a,b,0}\circ p^*_{a,b,0}$ and $p_{a,b,0}^*\circ p_{a,b,0}$.}
  \label{fig:surface-graph-cat-pants-map-iso}
\end{figure}

\begin{definition}
  \label{def:sym-surface-graph-category}
  The \emph{stable surface graph category} over a ring $R$ is the surface graph category modulo stable equivalence of the cobordisms.
\end{definition}

The isomorphism classes of objects in this category are determined by the number of marked points, and one could view them as stable equivalence classes of $0$-manifolds embedded in $1$-manifolds.
There is an equivalence of categories between the stable surface graph category and the virtual graph category.

\section{The Brauer category}
\label{sec:brauer-category}

Recall the definition of the Temperley-Lieb algebra $TL_n^c$, where $n\in\N$ and $c$ is an indeterminate, sometimes chosen to be a specific complex number.
$TL_n^c$ is an algebra over $\mathbb{C}(c)$ spanned by diagrams, which are oriented disks with properly embedded $1$-manifolds (\emph{strings}) up to isotopy rel boundary, where the disks have the same $2n$ equally spaced boundary points, and the boundary points are partitioned into two sets of consecutive ``top'' and ``bottom'' points.
(A manifold $N$ is \emph{properly} embedded in a manifold $M$ if $N\cap\partial M=\partial N$ and this intersection is transverse.)
A diagram with a closed loop bounding a disk (in the complement of the strings) is $c$ times the same diagram without the loop.

Composition is given  by gluing the bottom of the first diagram to the top of the second.
There is a trace defined by gluing the top to the bottom of a diagram; the resulting diagram is a collection of circles embedded in a sphere, so it evaluates to the empty diagram scaled by a power of $c$.
The Temperley-Lieb algebra is the endomorphism ring of a monoidal category, similar to that of Definition~\ref{def:surface-graph-category}, but where the objects are instead arcs with marked points and where the morphisms are $1$-manifolds embedded in disks cobounding these arcs.

Instead of disks we could consider diagrams drawn on surfaces.
This is a version of the surface graph category but with only embedded $1$-manifolds, rather than graphs of arbitrary valence.
Even after imposing the relation that closed loops bounding disks can be replaced by multiplication by $c$, diagrams without boundary points do not necessarily evaluate to a scalar because of lingering essential loops.
One may compare this to skein modules of dimension greater than one.
If, however, we instead consider the stable surface graph category in Definition~\ref{def:sym-surface-graph-category}, then every loop bounds a disk in some representative surface graph of the stable equivalence class, and so they can be replaced with multiplication by $c$.
This leads to a ``virtual Temperley-Lieb category:''

\begin{definition}
  \label{def:brauer-category}
  The \emph{Brauer category} $\Br^c$ is the subcategory of $\VG^{\C(c)}$ generated by all virtual graphs that as surface graphs are properly embedded $1$-manifolds, modulo multiplication by $c$ being equivalent to inserting a closed loop.  Diagrammatically, this is the Temperley-Lieb category after allowing proper immersions of $1$-manifolds rather than requiring proper embeddings.
\end{definition}

The \emph{Brauer algebra} $Br^c_n=\Br^c([n],[n])$ is the endomorphism algebra for $[n]$, and it was introduced by Brauer in \cite{Brauer1937} for the Schur-Weyl duality of the orthogonal group.
It is worth reviewing the map $\Br^N([m],[n])\twoheadrightarrow \Hom_{SO(N)}(V^{\otimes m},V^{\otimes n})$, where $V$ is an inner product space, $N=\dim V$, and $m+n$ is even.
A virtual graph diagram, drawn with $[m]$ at the bottom and $[n]$ at the top, is interpreted as an element of $(V^{\otimes m})^*\otimes V^{n}$ according to the following piece-by-piece correspondence, where $\{e_i\}_i$ is an orthonormal basis for $V$ and $\{e^i\}_i$ the corresponding dual basis for $V^*$:
\begin{align*}
  \raisebox{-3pt}{\marginbox{1.5pt 0pt}{
\begingroup%
  \makeatletter%
  \providecommand\color[2][]{%
    \errmessage{(Inkscape) Color is used for the text in Inkscape, but the package 'color.sty' is not loaded}%
    \renewcommand\color[2][]{}%
  }%
  \providecommand\transparent[1]{%
    \errmessage{(Inkscape) Transparency is used (non-zero) for the text in Inkscape, but the package 'transparent.sty' is not loaded}%
    \renewcommand\transparent[1]{}%
  }%
  \providecommand\rotatebox[2]{#2}%
  \newcommand*\fsize{\dimexpr\f@size pt\relax}%
  \newcommand*\lineheight[1]{\fontsize{\fsize}{#1\fsize}\selectfont}%
  \ifx\svgwidth\undefined%
    \setlength{\unitlength}{12.75830607bp}%
    \ifx\svgscale\undefined%
      \relax%
    \else%
      \setlength{\unitlength}{\unitlength * \real{\svgscale}}%
    \fi%
  \else%
    \setlength{\unitlength}{\svgwidth}%
  \fi%
  \global\let\svgwidth\undefined%
  \global\let\svgscale\undefined%
  \makeatother%
  \begin{picture}(1,1.07448914)%
    \lineheight{1}%
    \setlength\tabcolsep{0pt}%
    \put(0,0){\includegraphics[width=\unitlength,page=1]{br1-id.pdf}}%
  \end{picture}%
\endgroup%
}}&\mapsto \sum_{i}e^i\otimes e_i &
  \raisebox{-6pt}{\marginbox{1.5pt 0pt}{
\begingroup%
  \makeatletter%
  \providecommand\color[2][]{%
    \errmessage{(Inkscape) Color is used for the text in Inkscape, but the package 'color.sty' is not loaded}%
    \renewcommand\color[2][]{}%
  }%
  \providecommand\transparent[1]{%
    \errmessage{(Inkscape) Transparency is used (non-zero) for the text in Inkscape, but the package 'transparent.sty' is not loaded}%
    \renewcommand\transparent[1]{}%
  }%
  \providecommand\rotatebox[2]{#2}%
  \newcommand*\fsize{\dimexpr\f@size pt\relax}%
  \newcommand*\lineheight[1]{\fontsize{\fsize}{#1\fsize}\selectfont}%
  \ifx\svgwidth\undefined%
    \setlength{\unitlength}{19.06299186bp}%
    \ifx\svgscale\undefined%
      \relax%
    \else%
      \setlength{\unitlength}{\unitlength * \real{\svgscale}}%
    \fi%
  \else%
    \setlength{\unitlength}{\svgwidth}%
  \fi%
  \global\let\svgwidth\undefined%
  \global\let\svgscale\undefined%
  \makeatother%
  \begin{picture}(1,0.99999993)%
    \lineheight{1}%
    \setlength\tabcolsep{0pt}%
    \put(0,0){\includegraphics[width=\unitlength,page=1]{br2-x.pdf}}%
  \end{picture}%
\endgroup%
}}&\mapsto \sum_{i,j}e^i\otimes e^j\otimes e_j\otimes e_i\\
  \raisebox{-3pt}{\marginbox{1.5pt 0pt}{
\begingroup%
  \makeatletter%
  \providecommand\color[2][]{%
    \errmessage{(Inkscape) Color is used for the text in Inkscape, but the package 'color.sty' is not loaded}%
    \renewcommand\color[2][]{}%
  }%
  \providecommand\transparent[1]{%
    \errmessage{(Inkscape) Transparency is used (non-zero) for the text in Inkscape, but the package 'transparent.sty' is not loaded}%
    \renewcommand\transparent[1]{}%
  }%
  \providecommand\rotatebox[2]{#2}%
  \newcommand*\fsize{\dimexpr\f@size pt\relax}%
  \newcommand*\lineheight[1]{\fontsize{\fsize}{#1\fsize}\selectfont}%
  \ifx\svgwidth\undefined%
    \setlength{\unitlength}{19.06299321bp}%
    \ifx\svgscale\undefined%
      \relax%
    \else%
      \setlength{\unitlength}{\unitlength * \real{\svgscale}}%
    \fi%
  \else%
    \setlength{\unitlength}{\svgwidth}%
  \fi%
  \global\let\svgwidth\undefined%
  \global\let\svgscale\undefined%
  \makeatother%
  \begin{picture}(1,0.71912428)%
    \lineheight{1}%
    \setlength\tabcolsep{0pt}%
    \put(0,0){\includegraphics[width=\unitlength,page=1]{br2-cap.pdf}}%
  \end{picture}%
\endgroup%
}}&\mapsto \sum_{i}e^i\otimes e^i &
  \raisebox{-3pt}{\marginbox{1.5pt 0pt}{
\begingroup%
  \makeatletter%
  \providecommand\color[2][]{%
    \errmessage{(Inkscape) Color is used for the text in Inkscape, but the package 'color.sty' is not loaded}%
    \renewcommand\color[2][]{}%
  }%
  \providecommand\transparent[1]{%
    \errmessage{(Inkscape) Transparency is used (non-zero) for the text in Inkscape, but the package 'transparent.sty' is not loaded}%
    \renewcommand\transparent[1]{}%
  }%
  \providecommand\rotatebox[2]{#2}%
  \newcommand*\fsize{\dimexpr\f@size pt\relax}%
  \newcommand*\lineheight[1]{\fontsize{\fsize}{#1\fsize}\selectfont}%
  \ifx\svgwidth\undefined%
    \setlength{\unitlength}{19.06299321bp}%
    \ifx\svgscale\undefined%
      \relax%
    \else%
      \setlength{\unitlength}{\unitlength * \real{\svgscale}}%
    \fi%
  \else%
    \setlength{\unitlength}{\svgwidth}%
  \fi%
  \global\let\svgwidth\undefined%
  \global\let\svgscale\undefined%
  \makeatother%
  \begin{picture}(1,0.71912428)%
    \lineheight{1}%
    \setlength\tabcolsep{0pt}%
    \put(0,0){\includegraphics[width=\unitlength,page=1]{br2-cup.pdf}}%
  \end{picture}%
\endgroup%
}}&\mapsto \sum_{i}e_i\otimes e_i
\end{align*}
Glued strings are contracted using the natural pairing between $V$ and $V^*$.
Notice that $\raisebox{-3pt}{\marginbox{1.5pt 0pt}{}}$ is the inner product (evaluation map) and $\raisebox{-3pt}{\marginbox{1.5pt 0pt}{}}$ is its ``Casimir'' (coevaluation map), and so they are related by the topological identity
\begin{equation*}
  (\raisebox{-3pt}{\marginbox{1.5pt 0pt}{}}\otimes \raisebox{-3pt}{\marginbox{1.5pt 0pt}{}})\circ(\raisebox{-3pt}{\marginbox{1.5pt 0pt}{}}\otimes \raisebox{-3pt}{\marginbox{1.5pt 0pt}{}})=\raisebox{-3pt}{\marginbox{1.5pt 0pt}{}}=(\raisebox{-3pt}{\marginbox{1.5pt 0pt}{}}\otimes\raisebox{-3pt}{\marginbox{1.5pt 0pt}{}})\circ(\raisebox{-3pt}{\marginbox{1.5pt 0pt}{}}\otimes\raisebox{-3pt}{\marginbox{1.5pt 0pt}{}}),
\end{equation*}
which in a less linear form can be represented as
\begin{equation*}
  \raisebox{-10pt}{\marginbox{1.5pt 0pt}{
\begingroup%
  \makeatletter%
  \providecommand\color[2][]{%
    \errmessage{(Inkscape) Color is used for the text in Inkscape, but the package 'color.sty' is not loaded}%
    \renewcommand\color[2][]{}%
  }%
  \providecommand\transparent[1]{%
    \errmessage{(Inkscape) Transparency is used (non-zero) for the text in Inkscape, but the package 'transparent.sty' is not loaded}%
    \renewcommand\transparent[1]{}%
  }%
  \providecommand\rotatebox[2]{#2}%
  \newcommand*\fsize{\dimexpr\f@size pt\relax}%
  \newcommand*\lineheight[1]{\fontsize{\fsize}{#1\fsize}\selectfont}%
  \ifx\svgwidth\undefined%
    \setlength{\unitlength}{25.3133083bp}%
    \ifx\svgscale\undefined%
      \relax%
    \else%
      \setlength{\unitlength}{\unitlength * \real{\svgscale}}%
    \fi%
  \else%
    \setlength{\unitlength}{\svgwidth}%
  \fi%
  \global\let\svgwidth\undefined%
  \global\let\svgscale\undefined%
  \makeatother%
  \begin{picture}(1,0.9343837)%
    \lineheight{1}%
    \setlength\tabcolsep{0pt}%
    \put(0,0){\includegraphics[width=\unitlength,page=1]{br-s-one.pdf}}%
  \end{picture}%
\endgroup%
}}=\raisebox{-3pt}{\marginbox{1.5pt 0pt}{}}=\raisebox{-10pt}{\marginbox{1.5pt 0pt}{
\begingroup%
  \makeatletter%
  \providecommand\color[2][]{%
    \errmessage{(Inkscape) Color is used for the text in Inkscape, but the package 'color.sty' is not loaded}%
    \renewcommand\color[2][]{}%
  }%
  \providecommand\transparent[1]{%
    \errmessage{(Inkscape) Transparency is used (non-zero) for the text in Inkscape, but the package 'transparent.sty' is not loaded}%
    \renewcommand\transparent[1]{}%
  }%
  \providecommand\rotatebox[2]{#2}%
  \newcommand*\fsize{\dimexpr\f@size pt\relax}%
  \newcommand*\lineheight[1]{\fontsize{\fsize}{#1\fsize}\selectfont}%
  \ifx\svgwidth\undefined%
    \setlength{\unitlength}{25.3133083bp}%
    \ifx\svgscale\undefined%
      \relax%
    \else%
      \setlength{\unitlength}{\unitlength * \real{\svgscale}}%
    \fi%
  \else%
    \setlength{\unitlength}{\svgwidth}%
  \fi%
  \global\let\svgwidth\undefined%
  \global\let\svgscale\undefined%
  \makeatother%
  \begin{picture}(1,0.9343837)%
    \lineheight{1}%
    \setlength\tabcolsep{0pt}%
    \put(0,0){\includegraphics[width=\unitlength,page=1]{br-s-two.pdf}}%
  \end{picture}%
\endgroup%
}}.
\end{equation*} %
The images of $\raisebox{-3pt}{\marginbox{1.5pt 0pt}{}}^{\otimes a}\otimes\raisebox{-6pt}{\marginbox{1.5pt 0pt}{}}\otimes \raisebox{-3pt}{\marginbox{1.5pt 0pt}{}}^{\otimes b}$, with $a+2+b=n=m$, are induced by and generate the right action of the symmetric group $S_n$ on the tensor power $V^{\otimes n}$.
Loops are the composition $\raisebox{-3pt}{\marginbox{1.5pt 0pt}{}}\circ\raisebox{-3pt}{\marginbox{1.5pt 0pt}{}}$, and hence they evaluate to the dimension of $V$.
This correspondence in fact determines an additive functor from $\Br^N$ to the full subcategory of $\operatorname{Rep}(SO(N))$ (equivalently $\operatorname{Rep}(\mathfrak{so}(N))$) generated by tensor powers of $V$.

The Brauer algebra is semisimple for generic values of $c$ (see \cite{Wenzl1988}), failing only at integers.
The algebra $Br^c_2$ will make a prominent appearance.
It has the basis $\{\raisebox{-6pt}{\marginbox{1.5pt 0pt}{
\begingroup%
  \makeatletter%
  \providecommand\color[2][]{%
    \errmessage{(Inkscape) Color is used for the text in Inkscape, but the package 'color.sty' is not loaded}%
    \renewcommand\color[2][]{}%
  }%
  \providecommand\transparent[1]{%
    \errmessage{(Inkscape) Transparency is used (non-zero) for the text in Inkscape, but the package 'transparent.sty' is not loaded}%
    \renewcommand\transparent[1]{}%
  }%
  \providecommand\rotatebox[2]{#2}%
  \newcommand*\fsize{\dimexpr\f@size pt\relax}%
  \newcommand*\lineheight[1]{\fontsize{\fsize}{#1\fsize}\selectfont}%
  \ifx\svgwidth\undefined%
    \setlength{\unitlength}{19.06299186bp}%
    \ifx\svgscale\undefined%
      \relax%
    \else%
      \setlength{\unitlength}{\unitlength * \real{\svgscale}}%
    \fi%
  \else%
    \setlength{\unitlength}{\svgwidth}%
  \fi%
  \global\let\svgwidth\undefined%
  \global\let\svgscale\undefined%
  \makeatother%
  \begin{picture}(1,0.99130448)%
    \lineheight{1}%
    \setlength\tabcolsep{0pt}%
    \put(0,0){\includegraphics[width=\unitlength,page=1]{br2-id.pdf}}%
  \end{picture}%
\endgroup%
}},\raisebox{-6pt}{\marginbox{1.5pt 0pt}{
\begingroup%
  \makeatletter%
  \providecommand\color[2][]{%
    \errmessage{(Inkscape) Color is used for the text in Inkscape, but the package 'color.sty' is not loaded}%
    \renewcommand\color[2][]{}%
  }%
  \providecommand\transparent[1]{%
    \errmessage{(Inkscape) Transparency is used (non-zero) for the text in Inkscape, but the package 'transparent.sty' is not loaded}%
    \renewcommand\transparent[1]{}%
  }%
  \providecommand\rotatebox[2]{#2}%
  \newcommand*\fsize{\dimexpr\f@size pt\relax}%
  \newcommand*\lineheight[1]{\fontsize{\fsize}{#1\fsize}\selectfont}%
  \ifx\svgwidth\undefined%
    \setlength{\unitlength}{19.06299186bp}%
    \ifx\svgscale\undefined%
      \relax%
    \else%
      \setlength{\unitlength}{\unitlength * \real{\svgscale}}%
    \fi%
  \else%
    \setlength{\unitlength}{\svgwidth}%
  \fi%
  \global\let\svgwidth\undefined%
  \global\let\svgscale\undefined%
  \makeatother%
  \begin{picture}(1,0.99999993)%
    \lineheight{1}%
    \setlength\tabcolsep{0pt}%
    \put(0,0){\includegraphics[width=\unitlength,page=1]{br2-e.pdf}}%
  \end{picture}%
\endgroup%
}},\raisebox{-6pt}{\marginbox{1.5pt 0pt}{}}\}$, and the primitive central idempotents for this algebra are
\begin{align*}
  p_1&=\frac{1}{c}\raisebox{-6pt}{\marginbox{1.5pt 0pt}{}}\\
  p_2&=\frac{1}{2}\raisebox{-6pt}{\marginbox{1.5pt 0pt}{}}-\frac{1}{2}\raisebox{-6pt}{\marginbox{1.5pt 0pt}{}}\\
  p_3&=\frac{1}{2}\raisebox{-6pt}{\marginbox{1.5pt 0pt}{}}-\frac{1}{c}\raisebox{-6pt}{\marginbox{1.5pt 0pt}{}}+\frac{1}{2}\raisebox{-6pt}{\marginbox{1.5pt 0pt}{}}.
\end{align*}
The element $p_2+p_3=\raisebox{-6pt}{\marginbox{1.5pt 0pt}{}}-\frac{1}{c}\raisebox{-6pt}{\marginbox{1.5pt 0pt}{}}$ is the Jones-Wenzl idempotent $P^{(2)}$ for the embedded Temperley-Lieb algebra.

\section{Categories for invariants of virtual graphs}
\label{sec:virtual-graph-inv-categories}
A large class of graph invariants, surface graph invariants, and ribbon graph invariants are determined by local relations, such as a form of contraction and deletion of edges.
Examples include the Tutte-Whitney polynomial, its specializations the chromatic and flow polynomials, the Krushkal polynomial\cite{Krushkal2011}, the Bollob\'{a}s-Riordan polynomial\cite{Bollobas2001}, the \emph{$S$-polynomial} of Definition~\ref{def:s-polynomial}, and the $W_{\mathfrak{so}(N)}$ and $W_{\mathfrak{sl}(N)}$ Penrose polynomials in Section~\ref{sec:penrose-polys}.
One way to interpret these invariants in a categorical context is to consider the quotient of a category like $\mathsf{VG}^R$ by those local relations.

A sense in which local relations totally determine an algebraic invariant is that any graph with no external edges can be reduced to a scalar times an empty graph.
This is equivalent to saying that $\End(\mathbf{1})$ is isomorphic to the ring $R$ of scalars for the category, where $\mathbf{1}$ is the monoidal unit.
This perspective gives a motivation for stable equivalence: if there were no way to remove ``far away'' topology, $\End([0])$ for quotients of the surface graph category could instead be like the situation for Kauffman bracket skein modules, which are potentially infinitely generated.
Even in the case where every graph on a surface reduces to a scalar times the empty graph on the same surface, $\End([0])$ would still be $R^{\N}$, with one copy of $R$ for each homeomorphism class of surface.

In particular, graphical categories with pairings for which $\End([0])\cong R$ have a Markov-like trace.
This trace is defined by connecting the top strings to the bottom strings, which removes all external edges, giving a diagram that evaluates to a scalar.
Morphisms $a$ such that $\tr(ab) = 0$ for every compatible morphism $b$ correspond to local relations that preserve the invariant.
Such an $a$ is called a \emph{negligible element}.

\subsection{Edge contraction}
\label{subsec:edge-contraction}

In \cite{Chmutov2009}, Chmutov defines edge contraction for ribbon graphs by generalizing the fact for planar graphs that edge contraction corresponds to deleting the edge from the dual graph.

\begin{figure}[htb]
  \centering
  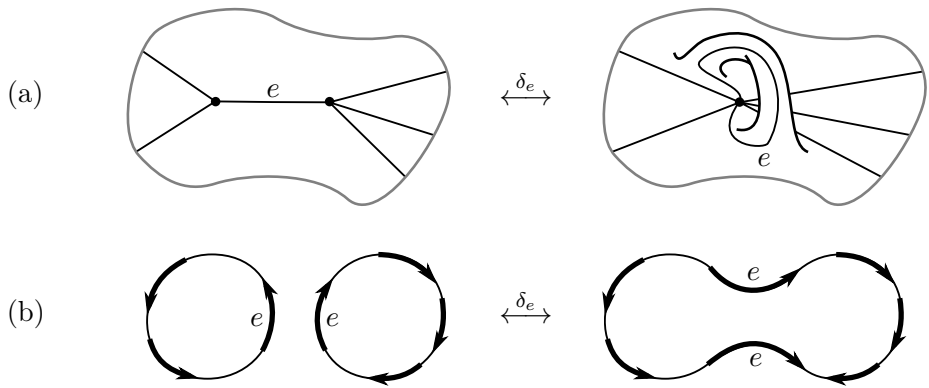
  \caption{Two virtual graphs that are locally related by the partial dual operation at the edge $e$, (a) as a cellular embedding, and (b) as an arrow presentation of a ribbon graph.}
  \label{fig:partial-dual}
\end{figure}

The \emph{partial dual} of a virtual graph $G$ at an edge $e$ is described in \cite{Chmutov2009} and \cite{Ellis-Monaghan2013}, and the operation, which we denote by $\delta_e G$, is illustrated in Figure~\ref{fig:partial-dual}.
The operation is best understood as a manipulation of an \emph{arrow presentation} for a ribbon graph, where the vertices of the ribbon graph are represented as closed disks with disjoint oriented labeled arcs on the boundaries, and the labels come in pairs to indicate how to glue in the edge disks.

The operation is an involution, and if $e,e'$ are two edges in $G$, $\delta_e$ and $\delta_{e'}$ commute.
The dual $G^*$ of $G$ is $\delta_{E(G)}G$, where $\delta_{E(G)}$ is the composition of all $\delta_e$ for $e\in E(G)$.
\begin{definition}
  \label{def:contraction}
  If $G$ is a virtual graph and $e$ is an edge in $G$, then $G/e$ is defined to be $\delta_e G-e$, the virtual graph obtained by \emph{contracting} the edge $e$.
\end{definition}

This definition avoids the problem that the quotient of a surface by an embedded loop might not itself be a manifold.
Instead, contracting a loop ``splits'' a vertex into two.
This particular definition of contraction is not well-defined for abstract graphs.

\subsection{The flow category}
Before defining the category $\Sc^Q$ for the $S$-polynomial, we motivate it with a more familiar example, the \emph{flow category}.
This category has been previously considered in, for instance, \cite{Agol2016}.
\begin{definition}
  \label{def:flow-category}
  The \emph{flow category} $\Flow^Q$ is the quotient of $\mathsf{VG}^{\mathbb{C}(Q)}$ by the following local relations:
  \begin{enumerate}
  \item Contraction-deletion: For $e$ an internal edge, if $e$ is not a loop, $[G]=[G/e]-[G-e]$, and otherwise $[G]=(Q-1)[G-e]$.
  \item If $v\in V(G)$ is a degree-$0$ internal vertex, $[G]=[G-v]$.
  \item If $G$ has a degree-$1$ internal vertex, $[G]=0$.
  \item If $G$ and $G'$ are related by move VI* in Figure~\ref{fig:vert-pliable-moves}, $[G]=[G']$.
  \end{enumerate}
  The \emph{flow polynomial} $F_G(Q)$ is the image of $G$ in $\Flow^Q([0],[0])\cong\mathbb{C}(Q)$, with $F_\emptyset(Q)=1$.
\end{definition}
It is not hard to show by induction on the number of edges that this flow polynomial is equivalent to the state sum given in Definition \ref{def:flow-polynomial}.
It is more natural to describe the flow category in terms of abstract graphs instead of virtual graphs, as in \cite{Agol2016}, but for sake of economical formalism we invoke move VI*.

\subsection{The \texorpdfstring{$S$-polynomial}{S-polynomial} category}
If $G$ is a planar graph and $e\in E(G)$ is a loop, then $F_{G/e}(Q)=F_{G-e}(Q)$.
This is because $e$ must intersect $G$ at exactly one point, so $G/e$ is a disjoint union of two graphs $G_1$ and $G_2$.
The flow polynomial is known to be multiplicative under both disjoint union and wedge sum, and it is evident that $G/e\cong G_1\amalg G_2$ and $G-e\cong G_1\vee G_2$.

Thus, for planar graphs the contraction-deletion relations for the flow category can be unified as $[G]=Q^{\beta(G,e)}[G/e]-[G-e]$, where $\beta(G,e)$ is half the difference between the Euler characteristics of $G/e$ and $G$, which is $1$ if $e$ is a loop and $0$ otherwise (see Figure~\ref{fig:spoly-contr-del}).
One way to define the $S$-polynomial is to declare that we take this rule seriously for nonplanar virtual graphs as well.

\begin{figure}[htb]
  \centering
    \begin{align*}
      \raisebox{-21pt}{\marginbox{4pt 0}{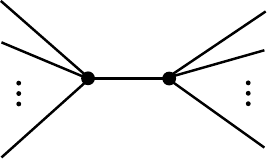}}
      &=       \raisebox{-21pt}{\marginbox{4pt 0}{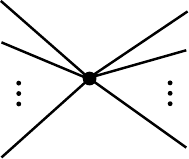}}
      -      \raisebox{-21pt}{\marginbox{4pt 0}{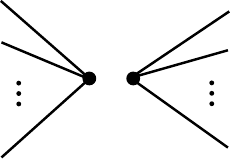}}\\
      \raisebox{-21pt}{\marginbox{4pt 0}{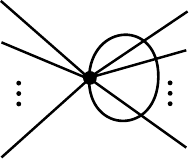}}
      &=Q\raisebox{-21pt}{\marginbox{4pt 0}{\input{fig/signed-vtx-dsd.pdf_tex}}}
        -\raisebox{-21pt}{\marginbox{4pt 0}{\input{fig/signed-vtx-d.pdf_tex}}}
    \end{align*}
  \caption{The contraction-deletion relations for $S^Q$.}
  \label{fig:spoly-contr-del}
\end{figure}

We will show that the $S$-polynomial is the invariant determined by the following axioms:
\begin{enumerate}
\item If $G$ is the empty graph, $S_G(Q)=1$.
\item If $G$ is a single-vertex virtual graph with no edges, $S_G(Q)=1$.
\item If $G$ has a degree-$1$ vertex, $S_G(Q)=0$.
\item If $G_1,G_2$ are virtual graphs, $S_{G_1\amalg G_2}(Q)=S_{G_1}(Q)S_{G_2}(Q)$.
\item If $e$ is an edge in a virtual graph $G$, $S_G(Q)=Q^{\beta(G,e)}S_{G/e}(Q)-S_{G-e}(Q)$.
\end{enumerate}
Uniqueness of such a polynomial follows from the fact that the right-hand sides of each equation involve virtual graphs of less complexity, measured by the sum of the numbers of vertices and edges.
While we could demonstrate existence by showing that the state sum in Definition~\ref{def:s-polynomial} satisfies these axioms, we will proceed by a more enlightening method: we construct a functor $\hat\Phi:\VG^{\C(Q)}\to \Br^{Q^{1/2}}$ that both satisfies the axioms and computes the state sum from $\VG^{\C(Q)}([0],[0])$.

Saying that the functor computes the polynomial means two things.
The most obvious is that $\hat \Phi$ maps any virtual graph with no external edges to the $S$-polynomial of that graph.
However, this will also be true for graphs with external edges, in the sense that if $\Sc^Q$ is the quotient of $\VG^{\C(Q)}$ by the local relations defining the $S$-polynomial, then $\hat \Phi$ factors through the quotient to give a trace-preserving functor $\Phi : \Sc^Q \to \Br^{Q^{1/2}}$.

Consider for a moment $S^0_G(Q)\in\mathbb{C}[Q^{\pm 1/2}]$, related to the $S$-polynomial by the renormalization $S_G(Q)=Q^{(\abs{E(G)}-\abs{V(G)})/2}S^0_G(Q)$.
The axioms from above when renormalized appear as follows:
\begin{enumerate}
\item If $G$ is the empty graph, $S^0_G(Q)=1$.
\item If $G$ is a single-vertex virtual graph with no edges, $S^0_G(Q)=Q^{1/2}$.
\item If $G$ has a degree-$1$ vertex, $S^0_G(Q)=0$.
\item If $G_1,G_2$ are virtual graphs, $S^0_{G_1\amalg G_2}(Q)=S^0_{G_1}(Q)S^0_{G_2}(Q)$.
\item If $e$ is an edge in a virtual graph $G$, $S^0_G(Q)=S^0_{G/e}(Q)-Q^{-1/2}S^0_{G - e}(Q)$.
\end{enumerate}

The last of these axioms suggests a relationship to the second Jones-Wenzl idempotent $P^{(2)}$, which plays a role in the following definition for $\hat\Phi$.

\begin{definition}
  \label{def:phi-functor}
  The functor $\hat \Phi:\VG^{\C(Q)}\to\Br^{Q^{1/2}}$ is defined on objects by sending $[n]$ to $[2n]$, and it is defined on morphisms in the following piece-by-piece fashion:
  \begin{itemize}
  \item Edges are replaced by $Q^{1/2}P^{(2)}=Q^{1/2}\left(\raisebox{-6pt}{\marginbox{1.5pt 0pt}{}}-Q^{-1/2}\raisebox{-6pt}{\marginbox{1.5pt 0pt}{}}\right)$.
  \item Internal vertices are replaced according to
    \begin{equation*}
      \raisebox{-14pt}{\marginbox{1.5pt 0}{
\begingroup%
  \makeatletter%
  \providecommand\color[2][]{%
    \errmessage{(Inkscape) Color is used for the text in Inkscape, but the package 'color.sty' is not loaded}%
    \renewcommand\color[2][]{}%
  }%
  \providecommand\transparent[1]{%
    \errmessage{(Inkscape) Transparency is used (non-zero) for the text in Inkscape, but the package 'transparent.sty' is not loaded}%
    \renewcommand\transparent[1]{}%
  }%
  \providecommand\rotatebox[2]{#2}%
  \newcommand*\fsize{\dimexpr\f@size pt\relax}%
  \newcommand*\lineheight[1]{\fontsize{\fsize}{#1\fsize}\selectfont}%
  \ifx\svgwidth\undefined%
    \setlength{\unitlength}{40.33613898bp}%
    \ifx\svgscale\undefined%
      \relax%
    \else%
      \setlength{\unitlength}{\unitlength * \real{\svgscale}}%
    \fi%
  \else%
    \setlength{\unitlength}{\svgwidth}%
  \fi%
  \global\let\svgwidth\undefined%
  \global\let\svgscale\undefined%
  \makeatother%
  \begin{picture}(1,0.93640957)%
    \lineheight{1}%
    \setlength\tabcolsep{0pt}%
    \put(0,0){\includegraphics[width=\unitlength,page=1]{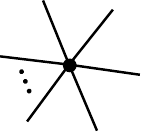}}%
  \end{picture}%
\endgroup%
}}
      \longmapsto
      Q^{-1/2}
      \raisebox{-14pt}{\marginbox{1.5pt 0}{
\begingroup%
  \makeatletter%
  \providecommand\color[2][]{%
    \errmessage{(Inkscape) Color is used for the text in Inkscape, but the package 'color.sty' is not loaded}%
    \renewcommand\color[2][]{}%
  }%
  \providecommand\transparent[1]{%
    \errmessage{(Inkscape) Transparency is used (non-zero) for the text in Inkscape, but the package 'transparent.sty' is not loaded}%
    \renewcommand\transparent[1]{}%
  }%
  \providecommand\rotatebox[2]{#2}%
  \newcommand*\fsize{\dimexpr\f@size pt\relax}%
  \newcommand*\lineheight[1]{\fontsize{\fsize}{#1\fsize}\selectfont}%
  \ifx\svgwidth\undefined%
    \setlength{\unitlength}{39.20439227bp}%
    \ifx\svgscale\undefined%
      \relax%
    \else%
      \setlength{\unitlength}{\unitlength * \real{\svgscale}}%
    \fi%
  \else%
    \setlength{\unitlength}{\svgwidth}%
  \fi%
  \global\let\svgwidth\undefined%
  \global\let\svgscale\undefined%
  \makeatother%
  \begin{picture}(1,0.95861196)%
    \lineheight{1}%
    \setlength\tabcolsep{0pt}%
    \put(0,0){\includegraphics[width=\unitlength,page=1]{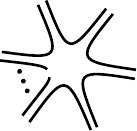}}%
  \end{picture}%
\endgroup%
}}.
    \end{equation*}
  \item Boundary vertices at the bottom of a diagram are replaced by $Q^{-1/2}\raisebox{-6pt}{\marginbox{1.5pt 0pt}{}}$, while those at the top by $\raisebox{-6pt}{\marginbox{1.5pt 0pt}{}}$.
  \item Virtual crossings are replaced by $\raisebox{-8pt}{\marginbox{1.5pt 0pt}{
\begingroup%
  \makeatletter%
  \providecommand\color[2][]{%
    \errmessage{(Inkscape) Color is used for the text in Inkscape, but the package 'color.sty' is not loaded}%
    \renewcommand\color[2][]{}%
  }%
  \providecommand\transparent[1]{%
    \errmessage{(Inkscape) Transparency is used (non-zero) for the text in Inkscape, but the package 'transparent.sty' is not loaded}%
    \renewcommand\transparent[1]{}%
  }%
  \providecommand\rotatebox[2]{#2}%
  \newcommand*\fsize{\dimexpr\f@size pt\relax}%
  \newcommand*\lineheight[1]{\fontsize{\fsize}{#1\fsize}\selectfont}%
  \ifx\svgwidth\undefined%
    \setlength{\unitlength}{25.3133083bp}%
    \ifx\svgscale\undefined%
      \relax%
    \else%
      \setlength{\unitlength}{\unitlength * \real{\svgscale}}%
    \fi%
  \else%
    \setlength{\unitlength}{\svgwidth}%
  \fi%
  \global\let\svgwidth\undefined%
  \global\let\svgscale\undefined%
  \makeatother%
  \begin{picture}(1,0.9343837)%
    \lineheight{1}%
    \setlength\tabcolsep{0pt}%
    \put(0,0){\includegraphics[width=\unitlength,page=1]{br-x-double.pdf}}%
  \end{picture}%
\endgroup%
}}$.
  \end{itemize}
\end{definition}
These replacements along with the requirements that $\hat \Phi$ be monoidal and linear completely determine $\hat \Phi$.
The normalization is the cause of the awkwardness in the difference between top and bottom boundary vertices --- a functor for $S^0$ would not have these additional factors of $Q^{\pm 1/2}$.

\begin{remark}
  There is a whole family of functors $\hat\Phi_k$ from $\VG^{\C(Q)}$ to $\Br^{Q^{1/2}}$ sending $[n]$ to $[kn]$ and edges to the Jones-Wenzl projector $P^{(k)}$, with $k$ an even natural number.
\end{remark}

Next we will show that this functor factors through another category, $\Sc^Q$, to demonstrate the relationship to the $S$-polynomial.
A consequence of the functorial construction will be that $\Sc^Q([0],[0])\cong\mathbb{C}(Q)$ and that the image of a graph in this endomorphism ring is its $S$-polynomial.

\begin{definition}
  \label{def:s-category}
  The \emph{$S$-polynomial category} $\Sc^Q$ is the quotient of $\VG^{\C(Q)}$ by the following local relations:
  \begin{enumerate}
  \item Contraction-deletion: For $e$ an internal edge, $[G]=Q^{\beta(G,e)}[G/e]-[G-e]$, where $\beta(G,e)$ indicates if $e$ is a loop.
  \item If $v\in V(G)$ is a degree-$0$ internal vertex, $[G]=[G-v]$.
  \item If $G$ has a degree-$1$ internal vertex, $[G]=0$.
  \end{enumerate}
\end{definition}

\begin{theorem}
  $\hat\Phi$ factors through $\Sc^Q$ to give a trace-preserving functor $\Phi:\Sc^Q\to\Br^{Q^{1/2}}$
\end{theorem}
\begin{proof}
  For the following, let $G\in\VG^{\C(Q)}([m],[n])$ be a diagram.

  If $G$ has a degree-$0$ internal vertex, then the diagram for $\hat\Phi(G)$ contains $Q^{-1/2}$ times a loop, which evaluates to $1$, and so $\hat\Phi(G)=\hat\Phi(G-v)$.

  If $G$ has a degree-$1$ internal vertex, then the diagram for $\hat\Phi(G)$ contains the composition $P^{(2)}\circ\raisebox{-3pt}{\marginbox{1.5pt 0pt}{}}$, where $P^{(2)}$ is from the edge incident to the vertex and $\raisebox{-3pt}{\marginbox{1.5pt 0pt}{}}$ is from the vertex itself.
  Then $\hat\Phi(G)=0$ since $P^{(2)}\circ\raisebox{-3pt}{\marginbox{1.5pt 0pt}{}}=0$.

  If $e\in E(G)$, one can show $\hat\Phi(G)=Q^{\beta(G,e)}\hat\Phi(G/e)-\hat\Phi(G-e)$, with $\beta(G,e)$ defined as it has been, by considering the two cases of $e$ being a loop or non-loop, and by expanding $e$ in the image as $Q^{1/2}\raisebox{-6pt}{\marginbox{1.5pt 0pt}{}}$ and $-\raisebox{-6pt}{\marginbox{1.5pt 0pt}{}}$.
  The first of these expansions corresponds to $Q^{\beta(G,e)}\hat\Phi(G/e)$, and the second to $-\hat\Phi(G-e)$.

  Thus, $\hat\Phi$ factors through the morphism sets of $\Sc^Q$.
  It furthermore factors through the composition law and the trace of $\Sc^Q$ because only one of the two vertices being glued contributes a factor of $Q^{-1/2}$ in $\Br^{Q^{1/2}}$.
\end{proof}

It is now established that the axioms give a well-defined polynomial invariant of virtual graphs.
The correspondence to the state sum formulation will be proved in Theorem~\ref{thm:phi-functor-gives-state-sum}.

\begin{remark}
  The definition of $\Phi$ is the extension of the $\mathsf{TL}^{Q^{1/2}}$ construction for the flow polynomial of cubic planar graphs in \cite{Fendley2009} to the $S$-polynomial of arbitrary virtual graphs.
\end{remark}

\begin{example}
  \label{example:polynomials}
  Let $G_1$ and $G_2$ respectively be the planar and toroidal theta graphs from Figure~\ref{fig:inequivalent-theta-graphs}.
  It is a quick application of the axioms to calculate
  \begin{align*}
    S_{G_1}(Q)=(Q-1)(Q-2) \text{ and }
    S_{G_2}(Q)=-2(Q-1).
  \end{align*}
  Thus, the polynomial can distinguish virtual graphs with the same underlying graph.
  The two graphs in Figure~\ref{fig:virtual-graph-example} have $S$-polynomials $(Q-1)^2$ and $Q-1$, respectively.
  Another example is the complete bipartite graph $K_{3,3}$ as a virtual graph by connecting via straight lines in $\R^2$ all the points $(1,n)$ to all the points $(2,m)$, with $1\leq m,n\leq 3$. $F_{K_{3,3}}(Q)=(Q-1)(Q-2)(Q^2-6Q+10)$, whereas $S_{K_{3,3}}(Q)=(Q-1)(Q-4)(Q+5)$.
  (The $S$-polynomials for $K_{3,3}$ with all possible rotation systems are that, $5(Q-1)(Q-4)$, and $-(Q-1)(Q-4)(Q-5)$.)
\end{example}

\begin{example}
  Alternatively, we could have computed the $S$-polynomial of the graphs from Figure~\ref{fig:inequivalent-theta-graphs} by the functor $\Phi$.
  In Figure~\ref{fig:state-sum-example}, $\raisebox{-6pt}{\marginbox{1.5pt 0pt}{
\begingroup%
  \makeatletter%
  \providecommand\color[2][]{%
    \errmessage{(Inkscape) Color is used for the text in Inkscape, but the package 'color.sty' is not loaded}%
    \renewcommand\color[2][]{}%
  }%
  \providecommand\transparent[1]{%
    \errmessage{(Inkscape) Transparency is used (non-zero) for the text in Inkscape, but the package 'transparent.sty' is not loaded}%
    \renewcommand\transparent[1]{}%
  }%
  \providecommand\rotatebox[2]{#2}%
  \newcommand*\fsize{\dimexpr\f@size pt\relax}%
  \newcommand*\lineheight[1]{\fontsize{\fsize}{#1\fsize}\selectfont}%
  \ifx\svgwidth\undefined%
    \setlength{\unitlength}{19.06299186bp}%
    \ifx\svgscale\undefined%
      \relax%
    \else%
      \setlength{\unitlength}{\unitlength * \real{\svgscale}}%
    \fi%
  \else%
    \setlength{\unitlength}{\svgwidth}%
  \fi%
  \global\let\svgwidth\undefined%
  \global\let\svgscale\undefined%
  \makeatother%
  \begin{picture}(1,0.99130448)%
    \lineheight{1}%
    \setlength\tabcolsep{0pt}%
    \put(0,0){\includegraphics[width=\unitlength,page=1]{br2-j.pdf}}%
  \end{picture}%
\endgroup%
}}$ represents $\raisebox{-6pt}{\marginbox{1.5pt 0pt}{}}-Q^{-1/2}\raisebox{-6pt}{\marginbox{1.5pt 0pt}{}}$, which in the expansion corresponds to the inclusion or exclusion of an edge.
\end{example}

\begin{figure}[tb]
  \centering
  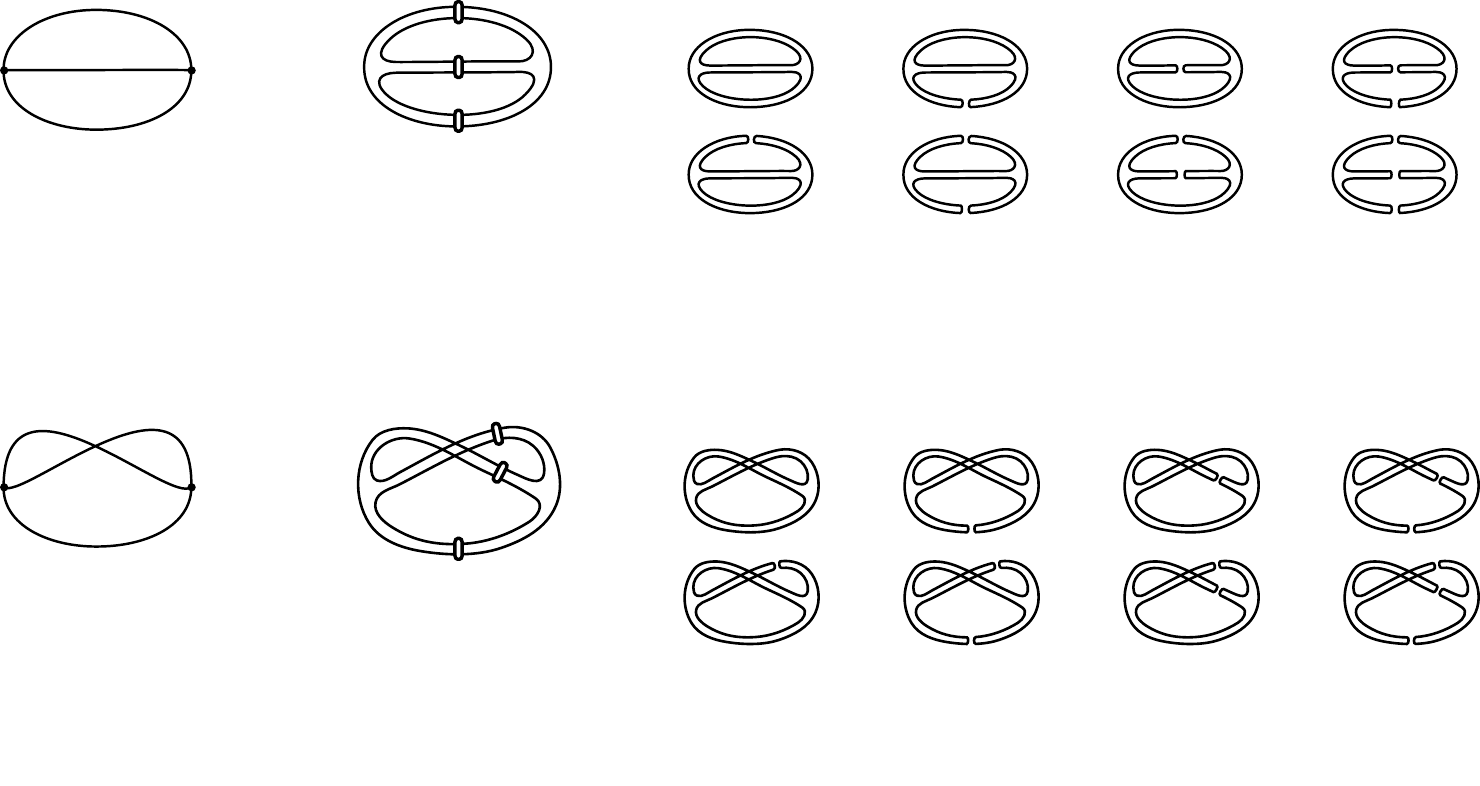
  \caption{$S$-polynomials of theta graphs by the state sum from the $\hat\Phi$ functor.}
  \label{fig:state-sum-example}
\end{figure}

\begin{theorem}
  \label{thm:phi-functor-gives-state-sum}
  Let $G$ be a virtual graph.
  Then $\hat\Phi(G)$ is $S_G(Q)$ as defined by the state sum in Definition~\ref{def:s-polynomial}.
\end{theorem}
\begin{proof}
  By viewing the two terms in $P^{(2)}$ as inclusion and exclusion of the edge, we see that the image of a virtual graph $G$ is
  \begin{equation*}
    Q^{(\abs{E(G)}-\abs{V(G)})/2}\sum_{T\subset E(G)} (-1)^{\abs{T}}Q^{(b_0(\partial(G-T))-\abs{T})/2},
  \end{equation*}
  where by $\partial(G-T)$ we mean the boundary of a ribbon graph representative for $G-T$.
  Let $G-T\incl\Sigma$ be a cellular embedding of the virtual graph $G-T$.
  The number of disks in $\Sigma-\nu(G-T)$ is $b_0(\partial(G-T))$, so by Euler characteristics,
  \begin{equation*}
    \abs{V(G)}-(\abs{E(G)}-\abs{T})+b_0(\partial(G-T)) = 2b_0(G-T)-2g(G-T),
  \end{equation*}
  because $\abs{E(G-T)}=\abs{E(G)}-\abs{T}$.
  Since in addition $\abs{V(G)}-(\abs{E(G)}-\abs{T})=b_0(G-T)-b_1(G-T)$,
  \begin{equation*}
    b_0(\partial(G-T))-\abs{T} = 2b_1(G-T) -2g(G-T) + \abs{V(G)} - \abs{E(G)}.
  \end{equation*}
  Hence the state sum is equivalently
  \begin{align*}
    \sum_{T\subset E(G)}(-1)^{\abs{T}} Q^{b_1(G-T)-g(G-T)},
  \end{align*}
  which matches Definition~\ref{def:s-polynomial}.
\end{proof}

\begin{lemma}
  $\Sc^Q([m],[n])$ has a basis in one-to-one correspondence with fixed-point-free permutations of the $m+n$ boundary half-edges.
  In particular, in the diagram for $\sigma\in S_{m+n}$, each cycle in the cycle decomposition of $\sigma$ is an interior vertex, and the cycle is the rotation system for the vertex.
\end{lemma}
\begin{proof}
  By contraction-deletion, we may assume that a particular element has no interior edges, and we may remove isolated interior vertices, hence the described set spans $\Sc^Q([m],[n])$ since a fixed point would correspond to a degree-$1$ interior vertex.
  For independence, consider the image under $\Phi$.
  In the expansion of an element corresponding to a permutation, the term corresponding to sending every edge to $\raisebox{-6pt}{\marginbox{1.5pt 0pt}{}}$ can be used to recover the permutation.
  This term can be identified by the fact that the strings are between only even boundary vertices or between only odd boundary vertices, where boundary vertices in $\Br^{Q^{1/2}}([2m],[2n])$ are \emph{even} or \emph{odd} depending on the parity of its numeric label.
  These terms are linearly independent in $\Br^{Q^{1/2}}([2m],[2n])$, hence the described set is independent.
\end{proof}

For a virtual graph $G$, call an edge $e$ a \emph{coloop} if the dual edge $e^*\in E(G^*)$ is a loop.
Bridge edges are coloops, and, if $G$ is planar, coloops are bridge edges.
From the perspective of ribbon graphs, an edge is a coloop if and only if the boundary components parallel to the edge are the same.
Thus, $b_0(\partial(G-e))=b_0(\partial G)+1$ if $e$ is a coloop, and otherwise $b_0(\partial(G-e))=b_0(\partial G)-1$.
In either case, $b_0(\partial(G/e))=b_0(\partial G)$, and an edge in $G/e$ is a coloop if and only if it was one in $G$.

\begin{theorem}
  \label{thm:s-poly-degree}
  For $G$ a virtual graph, $\deg S_G(Q)\leq b_1(G)-g(G)$.
  If $G$ has no coloops, then this is an equality and $S_G(Q)$ is monic.
\end{theorem}
\begin{proof}
  We will show that $I_G(t)=t^{b_1(G)-g(G)}S_G(t^{-1})$ is a polynomial.
  Recall that $b_1(G)-g(G)=\frac{1}{2}(\abs{E(G)}-\abs{V(G)}+b_0(\partial G))$.
  By analyzing the changes in the numbers of edges, vertices, and boundary components, one can get from the properties of the $S$-polynomial that
  \begin{enumerate}
  \item $I_{\mathrm{pt}}(t)=1$,
  \item $I_{G_1\amalg G_2}(t)=I_{G_1}(t)I_{G_2}(t)$,
  \item $I_{G}(t)=I_{G/e}(t)-t I_{G-e}(t)$ if $e\in E(G)$ is not a coloop, and
  \item $I_{G}(t)=I_{G/e}(t)-I_{G-e}(t)$ if $e\in E(G)$ is a coloop.
  \end{enumerate}
  These four properties are enough to compute $I$ for any graph, hence by induction it is a polynomial.
  Therefore $\deg S_G(Q)\leq b_1(G)-g(G)$.

  The value of $I_G(0)$ is the coefficient of the $Q^{b_1(G)-g(G)}$ term in $S_G(Q)$.
  If $G$ has no coloops, then since $G/e$ creates no new coloops, the contraction-deletion rules simplify to the single contraction rule $I_G(0)=I_{G/e}(0)$.
  Repeated application reduces the graph to a collection of discrete vertices, which evaluates to $1$.
  Hence in this case $\deg S_G(Q)=b_1(G)-g(G)$ with $S_G(Q)$ monic.
\end{proof}

\begin{definition}
  For two virtual graphs $G_1$ and $G_2$, consider arrow presentations, each with a distinguished boundary arc on a vertex disk.
  Then $G_1\vee G_2$ is the result of gluing the corresponding ribbon graphs along the two arcs.
  The resulting virtual graph is the identification of a vertex from $G_1$ with a vertex from $G_2$, with the rotation system at the vertex being some non-interleaved concatenation of both rotation systems.
\end{definition}

\begin{proposition}
  \label{prop:s-poly-bridge-zero}
  Let $G$ be a virtual graph.
  If $G$ has a bridge edge, then $S_G(Q)=0$.
  Furthermore, if $G_1$ and $G_2$ are two virtual graphs, $S_{G_1\amalg G_2}(Q)=S_{G_1\vee G_2}(Q)$.
\end{proposition}
\begin{proof}
  If $G$ has a bridge edge, then it is a composition of elements in $\VG^Q([1],[0])$ and $\VG^Q([0],[1])$.
  Since $\Sc^Q([1],[0]$ and $\Sc^Q([0],[1])$ are both zero-dimensional, it follows from the functor that $S_G(Q)=0$.
  Now let $G_1$ and $G_2$ be two virtual graphs, and let $G=(G_1\amalg G_2)\cup\{e\}$ for some new edge $e$, which is incident to a vertex in $G_1$ and a vertex in $G_2$.
  Then $G-e=G_1\amalg G_2$, $G/e=G_1\vee G_2$, and $G$ has a bridge edge.
  Hence, $0=S_G(Q)=S_{G_1\vee G_2}(Q)-S_{G_1\amalg G_2}(Q)$ by contraction-deletion.
\end{proof}

\begin{definition}
  Two distinct coloops $e,f\in E(G)$ are \emph{interlaced} if $f$ is not a coloop in $G-e$, or equivalently if $e^*$ and $f^*$ are incident in $G^*$ and their half edges come in interleaved order around the incident vertex.
\end{definition}

\begin{proposition}
  \label{thm:s-poly-interlace-strict}
  For $G$ a virtual graph, if $e\in E(G)$ is a coloop that does not interlace any other coloop, then $\deg S_G(Q)<b_1(G)-g(G)$, where $\deg 0=-\infty$.
\end{proposition}
\begin{proof}
  Let $I_G(t)$ be the polynomial as in Theorem~\ref{thm:s-poly-degree}.
  We will show under the hypothesis that $e\in E(G)$ is a coloop that does not interlace any other coloop, then $I_G(0)=0$.
  For each $f\in E(G)$ that is not a coloop, we have $I_G(0)=I_{G/f}(0)$ without changing which edges are coloops in the contraction or their interlacement.
  Hence, without loss of generality the edges of $G$ are all coloops, which is to say $G^*$ is a disjoint union of bouquets.

  Since $e$ interlaces with no coloops, $G-e$ and $G/e$ are, in some order, $G_1\amalg G_2$ and $G_1\vee G_2$ for some $G_1$ and $G_2$, by considering the dual graph.
  Hence $I_G(0)=I_{G/e}(0)-I_{G-e}(0)=0$.
\end{proof}

\subsection{Connect sums and the \texorpdfstring{$S$}{S}-polynomial}
\begin{definition}
  \label{def:edge-connect-sum}
  Let $G_1$ and $G_2$ be virtual graphs, each with a distinguished oriented edge.
  Decompose the graphs as $G_1=G_1'\circ \raisebox{-3pt}{\marginbox{1.5pt 0pt}{}}$ and $G_2=\raisebox{-3pt}{\marginbox{1.5pt 0pt}{}}\circ G_2'$, where $\raisebox{-3pt}{\marginbox{1.5pt 0pt}{}}$ and $\raisebox{-3pt}{\marginbox{1.5pt 0pt}{}}$ are the distinguished edges, co-oriented.
  The \emph{edge connect sum} of $G_1$ and $G_2$ at their respective edges is the virtual graph $G_1'\circ G_2'$, written $G_1\csum_2G_2$.
\end{definition}

\begin{proposition}
  \label{prop:s-connect-sum}
  If $G_1$ and $G_2$ are virtual graphs, $(Q-1)S_{G_1\csum_2 G_2}(Q)= S_{G_1}(Q)S_{G_2}(Q)$, no matter the choice or orientation of the distinguished edges.
\end{proposition}
\begin{proof}
  $\Sc^Q([1],[1])$ is one-dimensional and its trace is non-degenerate.
  Hence, if $G$ is a virtual graph that is $\tr G'$ for $G'\in \VG^Q([1],[1])$, the image of $G'$ in $Sc^Q([1],[1])$ is $\frac{S_G(Q)}{Q-1}\raisebox{-3pt}{\marginbox{1.5pt 0pt}{}}$.
  A similar result applies to $\Sc^Q([2],[0])$ and $\Sc^Q([0],[2])$ by composition with $\raisebox{-3pt}{\marginbox{1.5pt 0pt}{}}$, $\raisebox{-3pt}{\marginbox{1.5pt 0pt}{}}$, and $\raisebox{-3pt}{\marginbox{1.5pt 0pt}{}}$.

  Decompose $G_1$ and $G_2$ as $G_1'$ and $G_2'$ as in Definition~\ref{def:edge-connect-sum}.
  Their images in $\Sc^Q([2],[0])$ and $\Sc^Q([0],[2])$ are $\frac{S_{G_1}(Q)}{Q-1}\raisebox{-3pt}{\marginbox{1.5pt 0pt}{}}$ and $\frac{S_{G_2}(Q)}{Q-1}\raisebox{-3pt}{\marginbox{1.5pt 0pt}{}}$, respectively.
  The conclusion follows from $\raisebox{-3pt}{\marginbox{1.5pt 0pt}{}}\circ\raisebox{-3pt}{\marginbox{1.5pt 0pt}{}}=Q-1$.
\end{proof}

\begin{definition}
  Let $G_1$ and $G_2$ be virtual graphs each with a distinguished degree-$3$ vertex and incident half edge.
  Decompose $G_1$ as $G'_1\circ\raisebox{-3pt}{\marginbox{1.5pt 0pt}{
\begingroup%
  \makeatletter%
  \providecommand\color[2][]{%
    \errmessage{(Inkscape) Color is used for the text in Inkscape, but the package 'color.sty' is not loaded}%
    \renewcommand\color[2][]{}%
  }%
  \providecommand\transparent[1]{%
    \errmessage{(Inkscape) Transparency is used (non-zero) for the text in Inkscape, but the package 'transparent.sty' is not loaded}%
    \renewcommand\transparent[1]{}%
  }%
  \providecommand\rotatebox[2]{#2}%
  \newcommand*\fsize{\dimexpr\f@size pt\relax}%
  \newcommand*\lineheight[1]{\fontsize{\fsize}{#1\fsize}\selectfont}%
  \ifx\svgwidth\undefined%
    \setlength{\unitlength}{23.96091938bp}%
    \ifx\svgscale\undefined%
      \relax%
    \else%
      \setlength{\unitlength}{\unitlength * \real{\svgscale}}%
    \fi%
  \else%
    \setlength{\unitlength}{\svgwidth}%
  \fi%
  \global\let\svgwidth\undefined%
  \global\let\svgscale\undefined%
  \makeatother%
  \begin{picture}(1,0.57212585)%
    \lineheight{1}%
    \setlength\tabcolsep{0pt}%
    \put(0,0){\includegraphics[width=\unitlength,page=1]{sc-trivalent-dual.pdf}}%
  \end{picture}%
\endgroup%
}}$ and $G_2$ as $\raisebox{-3pt}{\marginbox{1.5pt 0pt}{
\begingroup%
  \makeatletter%
  \providecommand\color[2][]{%
    \errmessage{(Inkscape) Color is used for the text in Inkscape, but the package 'color.sty' is not loaded}%
    \renewcommand\color[2][]{}%
  }%
  \providecommand\transparent[1]{%
    \errmessage{(Inkscape) Transparency is used (non-zero) for the text in Inkscape, but the package 'transparent.sty' is not loaded}%
    \renewcommand\transparent[1]{}%
  }%
  \providecommand\rotatebox[2]{#2}%
  \newcommand*\fsize{\dimexpr\f@size pt\relax}%
  \newcommand*\lineheight[1]{\fontsize{\fsize}{#1\fsize}\selectfont}%
  \ifx\svgwidth\undefined%
    \setlength{\unitlength}{23.96091938bp}%
    \ifx\svgscale\undefined%
      \relax%
    \else%
      \setlength{\unitlength}{\unitlength * \real{\svgscale}}%
    \fi%
  \else%
    \setlength{\unitlength}{\svgwidth}%
  \fi%
  \global\let\svgwidth\undefined%
  \global\let\svgscale\undefined%
  \makeatother%
  \begin{picture}(1,0.57212585)%
    \lineheight{1}%
    \setlength\tabcolsep{0pt}%
    \put(0,0){\includegraphics[width=\unitlength,page=1]{sc-trivalent.pdf}}%
  \end{picture}%
\endgroup%
}}\circ G'_2$ with the distinguished half edges left-most in $\raisebox{-3pt}{\marginbox{1.5pt 0pt}{}}$ and $\raisebox{-3pt}{\marginbox{1.5pt 0pt}{}}$.
  The composition $G'_1\circ G'_2$ is the \emph{(trivalent) vertex connect sum} $G_1\csum_3G_2$ for the distinguished vertices.
\end{definition}

For a virtual graph $G$ with vertex $v$, let $\sigma_v G$ be $G$ except that $v$ is given the opposite rotation.
The \emph{twisted (trivalent) vertex connect sum} is $G_1\csum_3\sigma_v G_2$, for $v\in V(G_2)$.
For both vertex connect sums, the order of $G_1$ and $G_2$ does not matter.

\begin{proposition}
  \label{prop:s-vertex-connect-sum}
  Let $G_1$ and $G_2$ be virtual graphs with distinguished degree-$3$ vertices $v_1$ and $v_2$, respectively.
  Then, as depicted in Figure~\ref{fig:s-connect-sums},
  \begin{equation*}
    (Q-1)(Q-2)S_{G_1\csum_3\sigma_{v_2}G_2}(Q)-2(Q-1)S_{G_1\csum_3 G_2}(Q)
    =S_{G_1}(Q)S_{\sigma_{v_2}G_2}(Q) + S_{\sigma_{v_1}G_1}(Q)S_{G_2}(Q).
  \end{equation*}
\end{proposition}
\begin{proof}
  The space $\Sc^Q([3],[0])$ has the basis $\{\raisebox{-3pt}{\marginbox{1.5pt 0pt}{}},\raisebox{-3pt}{\marginbox{1.5pt 0pt}{
\begingroup%
  \makeatletter%
  \providecommand\color[2][]{%
    \errmessage{(Inkscape) Color is used for the text in Inkscape, but the package 'color.sty' is not loaded}%
    \renewcommand\color[2][]{}%
  }%
  \providecommand\transparent[1]{%
    \errmessage{(Inkscape) Transparency is used (non-zero) for the text in Inkscape, but the package 'transparent.sty' is not loaded}%
    \renewcommand\transparent[1]{}%
  }%
  \providecommand\rotatebox[2]{#2}%
  \newcommand*\fsize{\dimexpr\f@size pt\relax}%
  \newcommand*\lineheight[1]{\fontsize{\fsize}{#1\fsize}\selectfont}%
  \ifx\svgwidth\undefined%
    \setlength{\unitlength}{23.96091938bp}%
    \ifx\svgscale\undefined%
      \relax%
    \else%
      \setlength{\unitlength}{\unitlength * \real{\svgscale}}%
    \fi%
  \else%
    \setlength{\unitlength}{\svgwidth}%
  \fi%
  \global\let\svgwidth\undefined%
  \global\let\svgscale\undefined%
  \makeatother%
  \begin{picture}(1,0.57212585)%
    \lineheight{1}%
    \setlength\tabcolsep{0pt}%
    \put(0,0){\includegraphics[width=\unitlength,page=1]{sc-trivalent-flip.pdf}}%
  \end{picture}%
\endgroup%
}}\}$ and $\Sc^Q([0],[3])$ the basis $\{\raisebox{-3pt}{\marginbox{1.5pt 0pt}{}},\raisebox{-3pt}{\marginbox{1.5pt 0pt}{
\begingroup%
  \makeatletter%
  \providecommand\color[2][]{%
    \errmessage{(Inkscape) Color is used for the text in Inkscape, but the package 'color.sty' is not loaded}%
    \renewcommand\color[2][]{}%
  }%
  \providecommand\transparent[1]{%
    \errmessage{(Inkscape) Transparency is used (non-zero) for the text in Inkscape, but the package 'transparent.sty' is not loaded}%
    \renewcommand\transparent[1]{}%
  }%
  \providecommand\rotatebox[2]{#2}%
  \newcommand*\fsize{\dimexpr\f@size pt\relax}%
  \newcommand*\lineheight[1]{\fontsize{\fsize}{#1\fsize}\selectfont}%
  \ifx\svgwidth\undefined%
    \setlength{\unitlength}{23.96091938bp}%
    \ifx\svgscale\undefined%
      \relax%
    \else%
      \setlength{\unitlength}{\unitlength * \real{\svgscale}}%
    \fi%
  \else%
    \setlength{\unitlength}{\svgwidth}%
  \fi%
  \global\let\svgwidth\undefined%
  \global\let\svgscale\undefined%
  \makeatother%
  \begin{picture}(1,0.57212585)%
    \lineheight{1}%
    \setlength\tabcolsep{0pt}%
    \put(0,0){\includegraphics[width=\unitlength,page=1]{sc-trivalent-flip-dual.pdf}}%
  \end{picture}%
\endgroup%
}}\}$.
  The argument is similar to the one in Lemma~\ref{prop:s-connect-sum},
  but each graph is represented as a composition with a trivalent vertex.
  By writing $G'_1$ and $G'_2$ in terms of the respective bases, we can expand both sides of the required equation.
  The coefficients are the $S$-polynomials of the two theta graphs.
\end{proof}

\begin{figure}[htb]
  \centering
  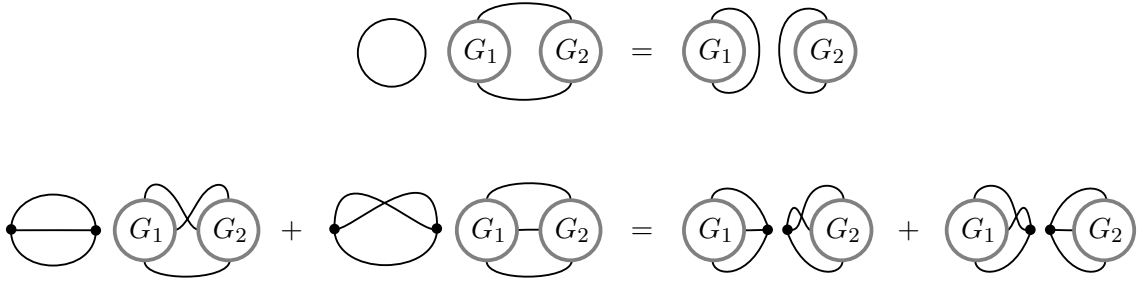
  \caption{Graphical representations of the edge and vertex connect sum relations.}
  \label{fig:s-connect-sums}
\end{figure}

\subsection{Local relations for the \texorpdfstring{$S$}{S}-polynomial}
\label{subsec:local-relations-s-poly}
The $S$-polynomial has additional local relations at certain values of $Q$.
The idea is that there is a pairing between $\Sc^Q([n],[0])$ and $\Sc^Q([0],[n])$ by composition, and, when this pairing is degenerate, elements in the radical of the pairing (that is, the kernel of the pairing as a map $\Sc^Q([n],[0])\to\Sc^Q([0],[n])^*$) give local linear relations.
This is a slight generalization of the use of negligible elements.

The space $\Sc^Q([3],[0])$ has the basis $\{\raisebox{-3pt}{\marginbox{1.5pt 0pt}{}},\raisebox{-3pt}{\marginbox{1.5pt 0pt}{}}\}$, and the Gramian matrix of the pairing with respect to this basis is
\begin{equation*}
  \begin{pmatrix}
    (Q-1)(Q-2)&-2(Q-1)\\
    -2(Q-1)&(Q-1)(Q-2)
  \end{pmatrix},
\end{equation*}
which is singular when $Q=0,1,4$.
For $Q=0$, the radical is spanned by $\raisebox{-3pt}{\marginbox{1.5pt 0pt}{}}-\raisebox{-3pt}{\marginbox{1.5pt 0pt}{}}$, implying the local relation $S_{\raisebox{-3pt}{\marginbox{1.5pt 0pt}{}}\circ G}(0)=S_{\raisebox{-3pt}{\marginbox{1.5pt 0pt}{}}\circ G}(0)$, where $G\in\VG^Q([0],[3])$.
Similarly, for $Q=4$, the radical is spanned by $\raisebox{-3pt}{\marginbox{1.5pt 0pt}{}}+\raisebox{-3pt}{\marginbox{1.5pt 0pt}{}}$, implying $S_{\raisebox{-3pt}{\marginbox{1.5pt 0pt}{}}\circ G}(4)=-S_{\raisebox{-3pt}{\marginbox{1.5pt 0pt}{}}\circ G}(4)$.
The $Q=1$ case is uninteresting, and it is a direct consequence of the state sum that $S_G(1)=0$ for all $G$.

Recall that for $v\in V(G)$, $\sigma_v G$ is $G$ with the rotation system at $v$ reversed.
When $W\subset V(G)$, let $\sigma_W$ denote the composition of all $\sigma_v$ for $v\in W$.

\begin{lemma}
  \label{lemma:q4-vtx-flip}
  At $Q=4$, the $S$-polynomial has additional local relations as depicted in Figure~\ref{fig:Q4-relation}.
  In particular, for $G$ a virtual graph and $v\in V(G)$, $S_{\sigma_v G}(4)=(-1)^{\deg(v)}S_G(4)$.
\end{lemma}
\begin{figure}[htb]
  \centering
\begingroup%
  \makeatletter%
  \providecommand\color[2][]{%
    \errmessage{(Inkscape) Color is used for the text in Inkscape, but the package 'color.sty' is not loaded}%
    \renewcommand\color[2][]{}%
  }%
  \providecommand\transparent[1]{%
    \errmessage{(Inkscape) Transparency is used (non-zero) for the text in Inkscape, but the package 'transparent.sty' is not loaded}%
    \renewcommand\transparent[1]{}%
  }%
  \providecommand\rotatebox[2]{#2}%
  \newcommand*\fsize{\dimexpr\f@size pt\relax}%
  \newcommand*\lineheight[1]{\fontsize{\fsize}{#1\fsize}\selectfont}%
  \ifx\svgwidth\undefined%
    \setlength{\unitlength}{157.38290802bp}%
    \ifx\svgscale\undefined%
      \relax%
    \else%
      \setlength{\unitlength}{\unitlength * \real{\svgscale}}%
    \fi%
  \else%
    \setlength{\unitlength}{\svgwidth}%
  \fi%
  \global\let\svgwidth\undefined%
  \global\let\svgscale\undefined%
  \makeatother%
  \begin{picture}(1,0.29299551)%
    \lineheight{1}%
    \setlength\tabcolsep{0pt}%
    \put(0,0){\includegraphics[width=\unitlength,page=1]{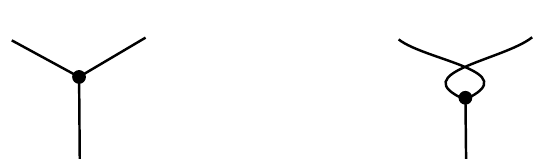}}%
    \put(0.31815596,0.11647913){\color[rgb]{0,0,0}\makebox(0,0)[lt]{\lineheight{1.25}\smash{\begin{tabular}[t]{l}${}=(-1)^{\deg(v)}$\end{tabular}}}}%
    \put(0,0){\includegraphics[width=\unitlength,page=2]{Q4-relation.pdf}}%
  \end{picture}%
\endgroup%

  \caption{At $Q=4$, the $S$-polynomial has additional local relations for vertex flips (rotation reversals), from negligible elements of $\Sc^4$.}
  \label{fig:Q4-relation}
\end{figure}
\begin{proof}
  We induct on the degree $k$ of flipped vertices, where $k=3$ was handled by analyzing the Gramian, and $k=1,2$ are from the fact flipping such vertices does not alter the virtual graph.
  \begin{align*}
    \raisebox{-16pt}{\marginbox{1.5pt 0}{
\begingroup%
  \makeatletter%
  \providecommand\color[2][]{%
    \errmessage{(Inkscape) Color is used for the text in Inkscape, but the package 'color.sty' is not loaded}%
    \renewcommand\color[2][]{}%
  }%
  \providecommand\transparent[1]{%
    \errmessage{(Inkscape) Transparency is used (non-zero) for the text in Inkscape, but the package 'transparent.sty' is not loaded}%
    \renewcommand\transparent[1]{}%
  }%
  \providecommand\rotatebox[2]{#2}%
  \newcommand*\fsize{\dimexpr\f@size pt\relax}%
  \newcommand*\lineheight[1]{\fontsize{\fsize}{#1\fsize}\selectfont}%
  \ifx\svgwidth\undefined%
    \setlength{\unitlength}{49.05671643bp}%
    \ifx\svgscale\undefined%
      \relax%
    \else%
      \setlength{\unitlength}{\unitlength * \real{\svgscale}}%
    \fi%
  \else%
    \setlength{\unitlength}{\svgwidth}%
  \fi%
  \global\let\svgwidth\undefined%
  \global\let\svgscale\undefined%
  \makeatother%
  \begin{picture}(1,0.77551888)%
    \lineheight{1}%
    \setlength\tabcolsep{0pt}%
    \put(0,0){\includegraphics[width=\unitlength,page=1]{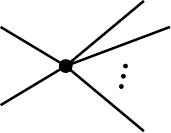}}%
  \end{picture}%
\endgroup%
}}
    &=\raisebox{-16pt}{\marginbox{1.5pt 0}{
\begingroup%
  \makeatletter%
  \providecommand\color[2][]{%
    \errmessage{(Inkscape) Color is used for the text in Inkscape, but the package 'color.sty' is not loaded}%
    \renewcommand\color[2][]{}%
  }%
  \providecommand\transparent[1]{%
    \errmessage{(Inkscape) Transparency is used (non-zero) for the text in Inkscape, but the package 'transparent.sty' is not loaded}%
    \renewcommand\transparent[1]{}%
  }%
  \providecommand\rotatebox[2]{#2}%
  \newcommand*\fsize{\dimexpr\f@size pt\relax}%
  \newcommand*\lineheight[1]{\fontsize{\fsize}{#1\fsize}\selectfont}%
  \ifx\svgwidth\undefined%
    \setlength{\unitlength}{64.05671453bp}%
    \ifx\svgscale\undefined%
      \relax%
    \else%
      \setlength{\unitlength}{\unitlength * \real{\svgscale}}%
    \fi%
  \else%
    \setlength{\unitlength}{\svgwidth}%
  \fi%
  \global\let\svgwidth\undefined%
  \global\let\svgscale\undefined%
  \makeatother%
  \begin{picture}(1,0.59391755)%
    \lineheight{1}%
    \setlength\tabcolsep{0pt}%
    \put(0,0){\includegraphics[width=\unitlength,page=1]{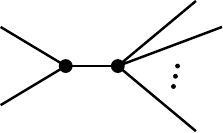}}%
  \end{picture}%
\endgroup%
}}
      +\raisebox{-16pt}{\marginbox{1.5pt 0}{
\begingroup%
  \makeatletter%
  \providecommand\color[2][]{%
    \errmessage{(Inkscape) Color is used for the text in Inkscape, but the package 'color.sty' is not loaded}%
    \renewcommand\color[2][]{}%
  }%
  \providecommand\transparent[1]{%
    \errmessage{(Inkscape) Transparency is used (non-zero) for the text in Inkscape, but the package 'transparent.sty' is not loaded}%
    \renewcommand\transparent[1]{}%
  }%
  \providecommand\rotatebox[2]{#2}%
  \newcommand*\fsize{\dimexpr\f@size pt\relax}%
  \newcommand*\lineheight[1]{\fontsize{\fsize}{#1\fsize}\selectfont}%
  \ifx\svgwidth\undefined%
    \setlength{\unitlength}{45.25601997bp}%
    \ifx\svgscale\undefined%
      \relax%
    \else%
      \setlength{\unitlength}{\unitlength * \real{\svgscale}}%
    \fi%
  \else%
    \setlength{\unitlength}{\svgwidth}%
  \fi%
  \global\let\svgwidth\undefined%
  \global\let\svgscale\undefined%
  \makeatother%
  \begin{picture}(1,0.84064854)%
    \lineheight{1}%
    \setlength\tabcolsep{0pt}%
    \put(0,0){\includegraphics[width=\unitlength,page=1]{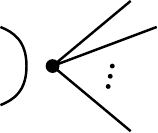}}%
  \end{picture}%
\endgroup%
}}\\
    &=-(-1)^{k-1}\raisebox{-16pt}{\marginbox{1.5pt 0}{
\begingroup%
  \makeatletter%
  \providecommand\color[2][]{%
    \errmessage{(Inkscape) Color is used for the text in Inkscape, but the package 'color.sty' is not loaded}%
    \renewcommand\color[2][]{}%
  }%
  \providecommand\transparent[1]{%
    \errmessage{(Inkscape) Transparency is used (non-zero) for the text in Inkscape, but the package 'transparent.sty' is not loaded}%
    \renewcommand\transparent[1]{}%
  }%
  \providecommand\rotatebox[2]{#2}%
  \newcommand*\fsize{\dimexpr\f@size pt\relax}%
  \newcommand*\lineheight[1]{\fontsize{\fsize}{#1\fsize}\selectfont}%
  \ifx\svgwidth\undefined%
    \setlength{\unitlength}{64.03460349bp}%
    \ifx\svgscale\undefined%
      \relax%
    \else%
      \setlength{\unitlength}{\unitlength * \real{\svgscale}}%
    \fi%
  \else%
    \setlength{\unitlength}{\svgwidth}%
  \fi%
  \global\let\svgwidth\undefined%
  \global\let\svgscale\undefined%
  \makeatother%
  \begin{picture}(1,0.59402653)%
    \lineheight{1}%
    \setlength\tabcolsep{0pt}%
    \put(0,0){\includegraphics[width=\unitlength,page=1]{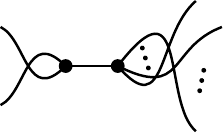}}%
  \end{picture}%
\endgroup%
}}
      +(-1)^{k-2}\raisebox{-16pt}{\marginbox{1.5pt 0}{
\begingroup%
  \makeatletter%
  \providecommand\color[2][]{%
    \errmessage{(Inkscape) Color is used for the text in Inkscape, but the package 'color.sty' is not loaded}%
    \renewcommand\color[2][]{}%
  }%
  \providecommand\transparent[1]{%
    \errmessage{(Inkscape) Transparency is used (non-zero) for the text in Inkscape, but the package 'transparent.sty' is not loaded}%
    \renewcommand\transparent[1]{}%
  }%
  \providecommand\rotatebox[2]{#2}%
  \newcommand*\fsize{\dimexpr\f@size pt\relax}%
  \newcommand*\lineheight[1]{\fontsize{\fsize}{#1\fsize}\selectfont}%
  \ifx\svgwidth\undefined%
    \setlength{\unitlength}{56.53460449bp}%
    \ifx\svgscale\undefined%
      \relax%
    \else%
      \setlength{\unitlength}{\unitlength * \real{\svgscale}}%
    \fi%
  \else%
    \setlength{\unitlength}{\svgwidth}%
  \fi%
  \global\let\svgwidth\undefined%
  \global\let\svgscale\undefined%
  \makeatother%
  \begin{picture}(1,0.67283133)%
    \lineheight{1}%
    \setlength\tabcolsep{0pt}%
    \put(0,0){\includegraphics[width=\unitlength,page=1]{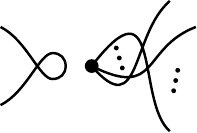}}%
  \end{picture}%
\endgroup%
}}\\
    &=(-1)^{k}\left(\raisebox{-16pt}{\marginbox{1.5pt 0}{}}
      +\raisebox{-16pt}{\marginbox{1.5pt 0}{}}\right)
      =(-1)^{k}\raisebox{-16pt}{\marginbox{1.5pt 0}{
\begingroup%
  \makeatletter%
  \providecommand\color[2][]{%
    \errmessage{(Inkscape) Color is used for the text in Inkscape, but the package 'color.sty' is not loaded}%
    \renewcommand\color[2][]{}%
  }%
  \providecommand\transparent[1]{%
    \errmessage{(Inkscape) Transparency is used (non-zero) for the text in Inkscape, but the package 'transparent.sty' is not loaded}%
    \renewcommand\transparent[1]{}%
  }%
  \providecommand\rotatebox[2]{#2}%
  \newcommand*\fsize{\dimexpr\f@size pt\relax}%
  \newcommand*\lineheight[1]{\fontsize{\fsize}{#1\fsize}\selectfont}%
  \ifx\svgwidth\undefined%
    \setlength{\unitlength}{49.03459976bp}%
    \ifx\svgscale\undefined%
      \relax%
    \else%
      \setlength{\unitlength}{\unitlength * \real{\svgscale}}%
    \fi%
  \else%
    \setlength{\unitlength}{\svgwidth}%
  \fi%
  \global\let\svgwidth\undefined%
  \global\let\svgscale\undefined%
  \makeatother%
  \begin{picture}(1,0.77574311)%
    \lineheight{1}%
    \setlength\tabcolsep{0pt}%
    \put(0,0){\includegraphics[width=\unitlength,page=1]{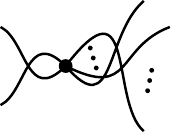}}%
  \end{picture}%
\endgroup%
}}
  \end{align*}
\end{proof}

\begin{proposition}
  \label{prop:s0-f0}
  If $G$ is a virtual graph, $S_G(0)=F_G(0)$.
\end{proposition}
\begin{proof}
  The $S$-polynomial at $Q=0$ has the additional local relation $\raisebox{-3pt}{\marginbox{1.5pt 0pt}{}}=\raisebox{-3pt}{\marginbox{1.5pt 0pt}{}}$.
  By a similar use of the contraction-deletion relation as in Lemma~\ref{lemma:q4-vtx-flip}, the polynomial is invariant under move VI*, and so at this evaluation the $S$-polynomial has the same axioms as the flow polynomial.
\end{proof}

\begin{lemma}
  \label{lemma:f0-bridge}
  Let $G$ be an abstract graph.
  The following are equivalent: (1) $G$ has a bridge edge, (2) $F_G(Q)=0$, and (3) $F_G(0)=0$.
\end{lemma}
\begin{proof}
  If $G$ has a bridge edge, then the flow polynomial of $G$ is from the composition of elements in $\Flow^Q([0],[1])$ and $\Flow^Q([1],[0])$, but both are zero-dimensional hence $F_G(Q)=0$, and thus $F_G(0)=0$.

  Consider the renormalization $f_G(x)=(-1)^{b_1(G)}F_G(1-x)$.
  This satisfies $f_{\mathrm{pt}}(x)=1$, $f_{G_1\amalg G_2}(x)=f_{G_1}(x)f_{G_2}(x)$, $f_{G}(x)=x f_{G-e}(x)$ for $e\in E(G)$ a loop edge, and $f_{G}(x)=f_{G/e}(x)+f_{G-e}(x)$ for $e\in E(G)$ a non-loop non-bridge edge, where $f_G(x)=0$ in the case $G$ has a bridge edge.
  By induction, if $G$ has no bridge edges, $f_G(x)$ is a nonzero degree-$b_1(G)$ polynomial with non-negative coefficients.
  Thus, if $G$ has no bridge edges, $f_G(1)>0$ and therefore $F_G(0)\neq 0$.
\end{proof}
\begin{corollary}
  Let $G$ be a virtual graph.
  The following are equivalent: (1) $G$ has a bridge edge, (2) $S_G(Q)=0$, and (3) $S_G(0)=0$.
\end{corollary}
\begin{proof}
  This follows from Proposition~\ref{prop:s-poly-bridge-zero} and Lemma~\ref{lemma:f0-bridge}.
\end{proof}

\begin{proposition}
  \label{prop:s4-f4}
  Let $G$ be a virtual graph such that $\sigma_W G$ is planar for some $W\subset V(G)$.
  Then with $n=\sum_{v\in W}\deg(v)$, $(-1)^n S_4(G)=F_4(G)$.
  If $G$ is a cubic, then $S_4(G)=(-1)^{\abs{W}}F_4(G)$.
\end{proposition}
\begin{proof}
  Let $W\subset V(G)$ be such that $\sigma_W G$ is planar.
  By Lemma~\ref{lemma:q4-vtx-flip},
  \[S_{\sigma_W G}(Q)
    =\left(\prod_{v\in W}(-1)^{\deg(v)}\right)S_G(Q)
    =(-1)^{\sum_{v\in W}\deg(v)}S_G(Q).\]
  The result then follows from the equivalence of the two polynomials for planar graphs.
\end{proof}
\begin{remark}
  The $S$-polynomials of virtual graphs for $K_{3,3}$ from Example~\ref{example:polynomials} have a root at $4$, yet the flow polynomial of $K_{3,3}$ is nonzero at $4$, hence $K_{3,3}$ is nonplanar.
  The flow polynomial of the Petersen graph is zero at $4$, as are the $S$-polynomials of all rotation systems of the Petersen graph.
  However, the condition $\pm S_G(4)=F_G(4)$ does not characterize planarity even when $S_G(4)\neq 0$, demonstrated by the cubic virtual graph in Figure~\ref{fig:withk33minor} whose underlying graph is nonplanar.
\end{remark}
\begin{figure}[htb]
  \centering
\begingroup%
  \makeatletter%
  \providecommand\color[2][]{%
    \errmessage{(Inkscape) Color is used for the text in Inkscape, but the package 'color.sty' is not loaded}%
    \renewcommand\color[2][]{}%
  }%
  \providecommand\transparent[1]{%
    \errmessage{(Inkscape) Transparency is used (non-zero) for the text in Inkscape, but the package 'transparent.sty' is not loaded}%
    \renewcommand\transparent[1]{}%
  }%
  \providecommand\rotatebox[2]{#2}%
  \newcommand*\fsize{\dimexpr\f@size pt\relax}%
  \newcommand*\lineheight[1]{\fontsize{\fsize}{#1\fsize}\selectfont}%
  \ifx\svgwidth\undefined%
    \setlength{\unitlength}{275.47878251bp}%
    \ifx\svgscale\undefined%
      \relax%
    \else%
      \setlength{\unitlength}{\unitlength * \real{\svgscale}}%
    \fi%
  \else%
    \setlength{\unitlength}{\svgwidth}%
  \fi%
  \global\let\svgwidth\undefined%
  \global\let\svgscale\undefined%
  \makeatother%
  \begin{picture}(1,0.24672859)%
    \lineheight{1}%
    \setlength\tabcolsep{0pt}%
    \put(0,0){\includegraphics[width=\unitlength,page=1]{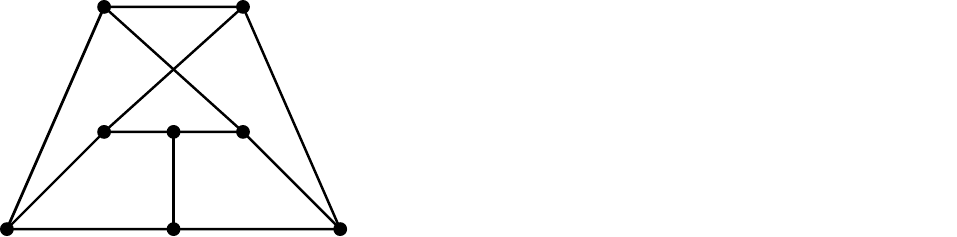}}%
    \put(0.4050557,0.1553087){\color[rgb]{0,0,0}\makebox(0,0)[lt]{\lineheight{0}\smash{\begin{tabular}[t]{l}$F_G(Q)=(Q-1)(Q-2)(Q-3)(Q^2-6 Q+11)$\end{tabular}}}}%
    \put(0.4050557,0.09141993){\color[rgb]{0,0,0}\makebox(0,0)[lt]{\lineheight{0}\smash{\begin{tabular}[t]{l}$S_G(Q)=-(Q-1)(Q-3)(Q^2-8 Q+22)$\end{tabular}}}}%
  \end{picture}%
\endgroup%

  \caption{A virtual graph with a $K_{3,3}$ minor where $F_G(4)=-S_G(4)=18\neq 0$.}
  \label{fig:withk33minor}
\end{figure}

One class of additional local linear relations comes from analyzing symmetrized elements.
Let $\operatorname{Sym}:\Sc^Q([n],[0])\to \Sc^Q([n],[0])$ be defined as
composition with the symmetrizer $\frac{1}{n!}\sum_{\sigma\in S_n}\sigma_*$, let $p_0(n_1,\dots,n_k)\in \Sc^Q([\sum_{i}n_i],[0])$ indicate any basis element with exactly $k$ internal vertices of degrees $n_1,\dots,n_k$, and let $p(n_1,\dots,n_k)=\operatorname{Sym}(p_0(n_1,\dots,n_k))$, which is independent of the choice of $p_0$.
Since the symmetrizer is an idempotent, if at a particular evaluation the pairing on $\operatorname{Sym} \Sc^Q([n],[0])$ is degenerate, then the pairing on $\Sc^Q([n],[0])$ is degenerate as well.
Table~\ref{tab:s-sym-negligible-elts} lists bases for negligible elements of $\operatorname{Sym}\Sc^{(n-1)^2}([n],[0])$, giving additional local relations for $\Sc^{(n-1)^2}([n],[0])$.
These are the only additional relations arising in this way for $n=2,\dots,8$.

\begin{table}[htb]
  \centering
  \caption{Some local relations for $\Sc$ at particular evaluations.}
  \label{tab:s-sym-negligible-elts}
  \begin{tabular}{rl}
    \toprule
    $Q$ & Basis for negligible elements of $\operatorname{Sym} \Sc^Q([Q^{1/2}+1],[0])$\\
    \midrule
    $1$ & $p(2)$ \\
    \addlinespace[2pt]
    $4$ & $p(3)$ \\
    \addlinespace[2pt]
    $9$ & $p(4) - \frac{3}{2}p(2,2)$ \\
    \addlinespace[2pt]
    $16$ & $p(5) - \frac{10}{3}p(3,2)$ \\
    \addlinespace[2pt]
    $25$ & $p(6) - \frac{15}{4}p(4,2)- \frac{5}{3}p(3,3) + \frac{25}{8}p(2,2,2)$ \\
    \addlinespace[2pt]
    $36$ & $p(7) - \frac{21}{5}p(5,2) - \frac{7}{2}p(4,3) + \frac{21}{2}p(3,2,2)$ \\
    \addlinespace[2pt]
    $49$ & $p(8) - \frac{14}{3}p(6,2) - \frac{56}{15}p(5,3) - \frac{7}{4}p(4,4) + \frac{49}{4}p(4,2,2) + \frac{98}{9}p(3,3,2) - \frac{343}{48}p(2,2,2,2)$ \\
    \bottomrule
  \end{tabular}
\end{table}

We computed the determinants of the Gramian matrices for $\Sc^Q([n],[0])$ with $n=2,\dots,6$ (listed in Table~\ref{tab:s-gramian-dets}) and offer the following conjecture:
\begin{conjecture}
  $S_Q(G)$ has additional local linear relations at $Q=n^2$ with $n$ an integer.
  In particular, the pairing is degenerate for $\Sc^Q([n],[0])$ exactly at $Q=1^2,2^2,\cdots,(n-1)^2$ and additionally at $Q=0$ when $n\geq 3$.
\end{conjecture}

\begin{table}[htb]
  \centering
  \caption{Determinants of Gramian matrices for $\Sc^Q([n],[0])$.}
  \label{tab:s-gramian-dets}
  \begin{tabular}{rl}
    \toprule
    $n$ & Determinant \\
    \midrule
    $2$ & $Q-1$ \\
    \addlinespace[2pt]
    $3$ & $(Q-4)(Q-1)^2Q$ \\
    \addlinespace[2pt]
    $4$ & $(Q-9)(Q-4)^6(Q-1)^9Q^8$\\
    \addlinespace[2pt]
    $5$ & $(Q-16)(Q-9)^{12}(Q-4)^{38}(Q-1)^{44}Q^{61}$\\ 
    \addlinespace[2pt]
    $6$ & $(Q-25)(Q-16)^{20}(Q-9)^{120}(Q-4)^{250}(Q-1)^{290}Q^{484}$\\
    \bottomrule
  \end{tabular}
\end{table}

\begin{remark}
There is also a representation-theoretic motivation for this conjecture.
In \cite{Fendley2009}, Fendley and Krushkal find an infinite family of local relations for the flow polynomial of planar graphs by pulling back the trace radical of the Temperley-Lieb algebra to the flow algebra of planar graphs.
This approach works because the trace radical is well-understood: it exists only when the loop parameter is a root of unity, and it is the tensor ideal generated by a Jones-Wenzl projector.
The projectors have a straightforward combinatorial description, and they give local relations for planar graphs with $n$ boundary edges at the values
\[
Q = 4 \cos^2 \left( \frac{\pi k}{2 n+1}\right),
\]
where $k < 2n$.
For $k = 1$ these are the Beraha numbers, which are conjectured to be accumulation points of the zeros of chromatic polynomials of planar triangulations.
Because the loop value $c = Q^{1/2}$ of the Temperley-Lieb category is interpreted as the \emph{quantum integer} $[2]_q = q + q^{-1}$, $Q = [2]_q^2$ with $q=e^{\pi ik/N}$ takes the above form.

As previously discussed, passing to nonplanar graphs requires the use of the Brauer category instead of the Temperley-Lieb category.
In this case, the special loop values are no longer $q + q^{-1}$ for $q$ a root of unity, but integers, as shown by Wenzl~\cite{Wenzl1988}, who computed that the trace for the Brauer algebras is degenerate only when the loop value is integral.
Because $Q$ corresponds to the square of the loop value under the functor $\Phi$, it appears that local relations for the $S$-polynomial should occur at squares of integers.

There are similar, but more complicated, formulas for projectors of the Brauer algebras~\cite{Tuba2005, Lehrer2012, King2016}.
We did not attempt to pull these back to the $S$-polynomial category, but instead we computed some local relations directly, as in Subsection~\ref{subsec:local-relations-s-poly}.
One complication is that $(P^{(2)})^{\otimes n}\circ \mathsf{Br}^{Q^{1/2}}([m],[n])\circ (P^{(2)})^{\otimes m}$ does not characterize the image of the functor from $\mathsf{S}^Q$, unlike the situation for the planar graph flow category and $\mathsf{TL}^{Q^{1/2}}$.

Finally, since there seem to be some connections between the flow polynomial and the $S$-polynomial, we mention that Jacobsen and Salas~\cite{Jacobsen2013} make a similar conjecture based on computational evidence for the flow and chromatic polynomials for certain families of nonplanar graphs, that the analogue for the Beraha numbers for nonplanar graphs are simply the nonnegative integers.
\end{remark}

\subsection{Virtual chromatic polynomial}
There is a virtual graph version of the chromatic polynomial, defined by analogy to the axioms for the $S$-polynomial from the flow polynomial.
For $G$ a virtual graph, the Laurent polynomial $\lambda_G(t)$ is determined by the following:
\begin{itemize}
\item If $G$ is a collection of $n$ isolated vertices, $\lambda_G(t)=t^n$.
\item If $e\in E(G)$ is a non-loop edge, $\lambda_G(t)=\lambda_{G-e}(t)-\lambda_{G/e}(t)$.
\item If $e\in E(G)$ is a loop edge, $\lambda_G(t)=\lambda_{G-e}(t)-t^{-1}\lambda_{G/e}(t)$.
\end{itemize}

\begin{proposition}
  Letting $G$ be a virtual graph, then $\lambda_G(Q)=Q^{b_0(G)-g(G)}S_{G^*}(Q)$.
  Equivalently, $S_G(Q)=Q^{g(G)-b_0(G)}\lambda_{G^*}(Q)$.
  Therefore $\lambda_G$ is well-defined, and for planar graphs $\lambda_G$ is the chromatic polynomial.
\end{proposition}
\begin{proof}
  Let $\lambda^0_G(Q)=S^0_{G^*}(Q)$.
  For $e\in E(G)$ and $F=E(G)-e$,
  \begin{align*}
    \lambda^0_G(Q)=S^0_{\delta_{E(G)}G}(Q)
                  &=S^0_{\delta_FG-e}(Q)-Q^{-1/2}S^0_{\delta_e\delta_FG-e}(Q)\\
                  &=S^0_{\delta_F(G-e)}(Q)-Q^{1/2}S^0_{\delta_F(\delta_eG-e)}(Q)\\
                  &=\lambda^0_{G-e}(Q)-Q^{-1/2}\lambda^0_{G/e}(Q).
  \end{align*}
  We will show $\lambda_G(Q)=Q^{\abs{V(G)}/2}\lambda^0_G(Q)$.

  If $G$ is a collection of isolated vertices, $Q^{\abs{V(G)}/2}\lambda^0_G(Q)=S^0_G(Q)=Q^{\abs{V(G)}/2}S^0_G(Q)=Q^{\abs{V(G)}}$.
  If $e\in E(G)$ is a loop edge or a non-loop edge, then one can check that $Q^{\abs{V(G)}/2}\lambda^0_G(Q)$ has the same deletion-contraction relation as $\lambda_G(Q)$.

  Then, $\lambda_G(Q)=Q^{\abs{V(G)}/2}S^0_{G^*}(Q)=Q^{(\abs{V(G)}-\abs{E(G)}+\abs{V(G^*)}/2}S_{G^*}(Q)$, and the result follows from $\abs{V(G)}-\abs{E(G)}+\abs{V(G^*)}=2b_0(G)-2g(G)$.
\end{proof}

\section{Penrose polynomials}
\label{sec:penrose-polys}

In \cite{Penrose1971}, Penrose introduced tensor diagrams and calculated invariants of virtual graphs by interpreting each vertex as a tensor with edges representing contraction via some nondegenerate (anti)symmetric form, and he demonstrated the correspondence between the $\mathfrak{so}(3)$, $\mathfrak{sl}(2)$, and ``$\mathfrak{so}(-2)$'' invariants, where the ``dimension'' is $\sum_i\langle e^i,e_i\rangle$ when $\{e_i\}_i$ and $\{e^i\}_i$ are bases such that $\langle e_i,e^j\rangle=\delta_{ij}$, a convention that allows the form to remain implicit in tensor diagrams.
A general case is summarized in \cite{Bar-Natan1997}, which associates to a metric Lie algebra $\mathfrak{g}$ a scalar invariant $W_{\mathfrak{g}}(G)$ of a virtual graph $G$, all of whose vertices are degree $2$ or $3$, by replacing each degree-$2$ vertex with $\langle-,-\rangle:\mathfrak{g}^{\otimes 2}\to\C$, each degree-$3$ vertex with the invariant $3$-form $\langle -,[-,-]\rangle:\mathfrak{g}^{\otimes 3}\to\C$, and then contracting along each edge with the Casimir element.
In other words, this is an invariant from coloring the graph by the adjoint representation of $\mathfrak{g}$, where the orientation of the edges is irrelevant since the metric gives $\mathfrak{g}\cong\mathfrak{g}^*$.
The invariants $W_{\mathfrak{sl}(N)}(G)$, $W_{\mathfrak{so}(N)}(G)$, and $W_{\mathfrak{sp}(2N)}(G)$ are polynomials in $N$, and these are called the \emph{Penrose polynomials}.
By rescaling the $2$- and $3$-forms, the Penrose polynomials vary by a normalization factor of $a^{\abs{V(G)}}b^{\abs{E(G)}}$, with $a$ and $b$ functions of $N$ independent of the graph.

The $W_{\mathfrak{so}(N)}$ polynomial was extended in \cite{Aigner1997} for planar graphs with vertices of arbitrary degree, and again in \cite{Ellis-Monaghan2013} for surface graphs.
An algebra of connected cubic virtual graphs modulo the IHX relation is considered in \cite{Duzhin1998}, where, with multiplication being edge connect sum, the map to the Penrose polynomials is a homomorphism.

We give an extension of $W_{\mathfrak{sl}(N)}$ to signed virtual graphs of arbitrary degree in Section~\ref{sec:sln-rels}.

\subsection{Via the Brauer category}

In this section, we reiterate \cite{Penrose1971} and \cite{Bar-Natan1997} for $W_{\mathfrak{so}(N)}$ and $W_{\mathfrak{sl}(N)}$ in terms of the Brauer category and virtual graphs.
Let $k$ be a field, and consider the vector space $k^N$ with the standard inner product, which gives an isomorphism $k^N\cong (k^N)^*$ and a correspondence between $\mathfrak{gl}(N)=\End_k(k^N)$ and $k^N\otimes k^N$.
The trace operator $\tr:\mathfrak{gl}(N)\to k$ under this correspondence is $A\mapsto \raisebox{-3pt}{\marginbox{1.5pt 0pt}{}}\circ A$, and matrix multiplication is $A\otimes B\mapsto (\raisebox{-3pt}{\marginbox{1.5pt 0pt}{}}\otimes\raisebox{-3pt}{\marginbox{1.5pt 0pt}{}}\otimes\raisebox{-3pt}{\marginbox{1.5pt 0pt}{}})\circ(A\otimes B)$.
The Killing form on $\mathfrak{gl}(N)$ is a scale multiple of the trace form $A\otimes B\mapsto \tr(AB)$ and thus is a scale multiple of
\begin{equation*}
  A\otimes B\mapsto \raisebox{-5pt}{\marginbox{1.5pt 0pt}{
\begingroup%
  \makeatletter%
  \providecommand\color[2][]{%
    \errmessage{(Inkscape) Color is used for the text in Inkscape, but the package 'color.sty' is not loaded}%
    \renewcommand\color[2][]{}%
  }%
  \providecommand\transparent[1]{%
    \errmessage{(Inkscape) Transparency is used (non-zero) for the text in Inkscape, but the package 'transparent.sty' is not loaded}%
    \renewcommand\transparent[1]{}%
  }%
  \providecommand\rotatebox[2]{#2}%
  \newcommand*\fsize{\dimexpr\f@size pt\relax}%
  \newcommand*\lineheight[1]{\fontsize{\fsize}{#1\fsize}\selectfont}%
  \ifx\svgwidth\undefined%
    \setlength{\unitlength}{29.10137747bp}%
    \ifx\svgscale\undefined%
      \relax%
    \else%
      \setlength{\unitlength}{\unitlength * \real{\svgscale}}%
    \fi%
  \else%
    \setlength{\unitlength}{\svgwidth}%
  \fi%
  \global\let\svgwidth\undefined%
  \global\let\svgscale\undefined%
  \makeatother%
  \begin{picture}(1,0.6142505)%
    \lineheight{1}%
    \setlength\tabcolsep{0pt}%
    \put(0,0){\includegraphics[width=\unitlength,page=1]{br-gl-killing.pdf}}%
  \end{picture}%
\endgroup%
}}\circ (A\otimes B).
\end{equation*}
The vertical reflection $\raisebox{-5pt}{\marginbox{1.5pt 0pt}{
\begingroup%
  \makeatletter%
  \providecommand\color[2][]{%
    \errmessage{(Inkscape) Color is used for the text in Inkscape, but the package 'color.sty' is not loaded}%
    \renewcommand\color[2][]{}%
  }%
  \providecommand\transparent[1]{%
    \errmessage{(Inkscape) Transparency is used (non-zero) for the text in Inkscape, but the package 'transparent.sty' is not loaded}%
    \renewcommand\transparent[1]{}%
  }%
  \providecommand\rotatebox[2]{#2}%
  \newcommand*\fsize{\dimexpr\f@size pt\relax}%
  \newcommand*\lineheight[1]{\fontsize{\fsize}{#1\fsize}\selectfont}%
  \ifx\svgwidth\undefined%
    \setlength{\unitlength}{29.10137747bp}%
    \ifx\svgscale\undefined%
      \relax%
    \else%
      \setlength{\unitlength}{\unitlength * \real{\svgscale}}%
    \fi%
  \else%
    \setlength{\unitlength}{\svgwidth}%
  \fi%
  \global\let\svgwidth\undefined%
  \global\let\svgscale\undefined%
  \makeatother%
  \begin{picture}(1,0.6142505)%
    \lineheight{1}%
    \setlength\tabcolsep{0pt}%
    \put(0,0){\includegraphics[width=\unitlength,page=1]{br-gl-casimir.pdf}}%
  \end{picture}%
\endgroup%
}}$ is the Casimir, and $\raisebox{-5pt}{\marginbox{1.5pt 0pt}{}}\circ\raisebox{-5pt}{\marginbox{1.5pt 0pt}{}}=N^2=\dim(\mathfrak{gl}(N))$.
The anti-involution $A\mapsto A^T$ is composition with $\raisebox{-6pt}{\marginbox{1.5pt 0pt}{}}$, and since $BA=(A^TB^T)^T$ we have the following representation for the Lie bracket:
\begin{equation*}
  [A,B]=\left(\raisebox{-8pt}{\marginbox{1.5pt 0pt}{
\begingroup%
  \makeatletter%
  \providecommand\color[2][]{%
    \errmessage{(Inkscape) Color is used for the text in Inkscape, but the package 'color.sty' is not loaded}%
    \renewcommand\color[2][]{}%
  }%
  \providecommand\transparent[1]{%
    \errmessage{(Inkscape) Transparency is used (non-zero) for the text in Inkscape, but the package 'transparent.sty' is not loaded}%
    \renewcommand\transparent[1]{}%
  }%
  \providecommand\rotatebox[2]{#2}%
  \newcommand*\fsize{\dimexpr\f@size pt\relax}%
  \newcommand*\lineheight[1]{\fontsize{\fsize}{#1\fsize}\selectfont}%
  \ifx\svgwidth\undefined%
    \setlength{\unitlength}{25.3133083bp}%
    \ifx\svgscale\undefined%
      \relax%
    \else%
      \setlength{\unitlength}{\unitlength * \real{\svgscale}}%
    \fi%
  \else%
    \setlength{\unitlength}{\svgwidth}%
  \fi%
  \global\let\svgwidth\undefined%
  \global\let\svgscale\undefined%
  \makeatother%
  \begin{picture}(1,0.9343837)%
    \lineheight{1}%
    \setlength\tabcolsep{0pt}%
    \put(0,0){\includegraphics[width=\unitlength,page=1]{br-lie-one.pdf}}%
  \end{picture}%
\endgroup%
}}-\raisebox{-8pt}{\marginbox{1.5pt 0pt}{
\begingroup%
  \makeatletter%
  \providecommand\color[2][]{%
    \errmessage{(Inkscape) Color is used for the text in Inkscape, but the package 'color.sty' is not loaded}%
    \renewcommand\color[2][]{}%
  }%
  \providecommand\transparent[1]{%
    \errmessage{(Inkscape) Transparency is used (non-zero) for the text in Inkscape, but the package 'transparent.sty' is not loaded}%
    \renewcommand\transparent[1]{}%
  }%
  \providecommand\rotatebox[2]{#2}%
  \newcommand*\fsize{\dimexpr\f@size pt\relax}%
  \newcommand*\lineheight[1]{\fontsize{\fsize}{#1\fsize}\selectfont}%
  \ifx\svgwidth\undefined%
    \setlength{\unitlength}{25.3133083bp}%
    \ifx\svgscale\undefined%
      \relax%
    \else%
      \setlength{\unitlength}{\unitlength * \real{\svgscale}}%
    \fi%
  \else%
    \setlength{\unitlength}{\svgwidth}%
  \fi%
  \global\let\svgwidth\undefined%
  \global\let\svgscale\undefined%
  \makeatother%
  \begin{picture}(1,0.9343837)%
    \lineheight{1}%
    \setlength\tabcolsep{0pt}%
    \put(0,0){\includegraphics[width=\unitlength,page=1]{br-lie-two.pdf}}%
  \end{picture}%
\endgroup%
}}\right)\circ(A\otimes B).
\end{equation*}

The algebra $Br^N_2$ contains projectors onto $\mathfrak{so}(N)$ and $\mathfrak{sl}(N)$.
The element $\frac{1}{2}\raisebox{-6pt}{\marginbox{1.5pt 0pt}{}}-\frac{1}{2}\raisebox{-6pt}{\marginbox{1.5pt 0pt}{}}$ is the antisymmetrizer, thus projects onto $\mathfrak{so}(N)$.
The element $\raisebox{-6pt}{\marginbox{1.5pt 0pt}{}}:=\raisebox{-6pt}{\marginbox{1.5pt 0pt}{}}-\frac{1}{N}\raisebox{-6pt}{\marginbox{1.5pt 0pt}{}}$ eliminates the trace of a matrix, since composition with it is equivalent to $A\mapsto A-\frac{\tr A}{N}\id_N$, thus projects onto $\mathfrak{sl}(N)$.

The Penrose polynomials, then, can be calculated by replacing each (trivalent) vertex with $\raisebox{-8pt}{\marginbox{1.5pt 0pt}{}}-\raisebox{-8pt}{\marginbox{1.5pt 0pt}{}}$ and each edge with the respective projector, since the Killing form for $\mathfrak{gl}(N)$ is two parallel strings.
Virtual crossings are replaced with $\raisebox{-8pt}{\marginbox{1.5pt 0pt}{}}$.
We are being cavalier with ``top'' and ``bottom'' with the diagrams because self-duality affords us this liberty.

There are some shortcuts one may take to calculate these Penrose polynomials.
For $W_{\mathfrak{so}(N)}$, since $\raisebox{-6pt}{\marginbox{1.5pt 0pt}{}}$ negates the antisymmetrizer, $\raisebox{-6pt}{\marginbox{1.5pt 0pt}{}}\circ\raisebox{-8pt}{\marginbox{1.5pt 0pt}{}}\circ(\raisebox{-6pt}{\marginbox{1.5pt 0pt}{}}\otimes\raisebox{-6pt}{\marginbox{1.5pt 0pt}{}})=-\raisebox{-8pt}{\marginbox{1.5pt 0pt}{}}$.
Hence, up to normalization, we may replace trivalent vertices with $\raisebox{-8pt}{\marginbox{1.5pt 0pt}{}}$ and edges with $\raisebox{-6pt}{\marginbox{1.5pt 0pt}{}}-\raisebox{-6pt}{\marginbox{1.5pt 0pt}{}}$.
We can extend this to arbitrary virtual graphs (as for surface graphs in \cite{Ellis-Monaghan2013}) using the following replacement:
\begin{equation*}
  \raisebox{-14pt}{\marginbox{1.5pt 0}{}}
  \longmapsto
  \raisebox{-14pt}{\marginbox{1.5pt 0}{}}.
\end{equation*}
Note that isolated vertices contribute a factor of $N$.
Loops evaluate to $N(N-1)$.

For $W_{\mathfrak{sl}(N)}$, we may replace edges that are incident to trivalent vertices with $\raisebox{-6pt}{\marginbox{1.5pt 0pt}{}}$ rather than $\raisebox{-6pt}{\marginbox{1.5pt 0pt}{}}-\frac{1}{N}\raisebox{-6pt}{\marginbox{1.5pt 0pt}{}}$ because $\tr[-,-]=0$.
Loops evaluate to $N^2-1$.

\subsection{Relations for \texorpdfstring{$W_{\mathfrak{so}(N)}$}{Wso(N)}}

In this section, we describe the ``contraction-deletion'' relations for $W_{\mathfrak{so}(N)}$ present in \cite{Ellis-Monaghan2013}, but in terms of virtual graphs.
In that paper, Ellis-Monaghan and Moffatt use the \emph{twisted dual} operation for graphs on unoriented surfaces, yet we have only been considering oriented surfaces.
Nevertheless, we will derive a relation directly from the Brauer algebra replacement rules, but we will also show how to augment our virtual graph notation to accommodate the description of relations with the twisted dual.

Let $\raisebox{-6pt}{\marginbox{1.5pt 0pt}{
\begingroup%
  \makeatletter%
  \providecommand\color[2][]{%
    \errmessage{(Inkscape) Color is used for the text in Inkscape, but the package 'color.sty' is not loaded}%
    \renewcommand\color[2][]{}%
  }%
  \providecommand\transparent[1]{%
    \errmessage{(Inkscape) Transparency is used (non-zero) for the text in Inkscape, but the package 'transparent.sty' is not loaded}%
    \renewcommand\transparent[1]{}%
  }%
  \providecommand\rotatebox[2]{#2}%
  \newcommand*\fsize{\dimexpr\f@size pt\relax}%
  \newcommand*\lineheight[1]{\fontsize{\fsize}{#1\fsize}\selectfont}%
  \ifx\svgwidth\undefined%
    \setlength{\unitlength}{19.06299186bp}%
    \ifx\svgscale\undefined%
      \relax%
    \else%
      \setlength{\unitlength}{\unitlength * \real{\svgscale}}%
    \fi%
  \else%
    \setlength{\unitlength}{\svgwidth}%
  \fi%
  \global\let\svgwidth\undefined%
  \global\let\svgscale\undefined%
  \makeatother%
  \begin{picture}(1,0.99130448)%
    \lineheight{1}%
    \setlength\tabcolsep{0pt}%
    \put(0,0){\includegraphics[width=\unitlength,page=1]{br2-alt.pdf}}%
  \end{picture}%
\endgroup%
}}:=\raisebox{-6pt}{\marginbox{1.5pt 0pt}{}}-\raisebox{-6pt}{\marginbox{1.5pt 0pt}{}}$.
In general, a solid bar across $n$ strings is $\sum_{\sigma\in S_n}(-1)^{\sigma}\sigma_*$ in Penrose notation.
\begin{proposition}
  \label{prop:son-contr-del}
  For $W_{\mathfrak{so}(N)}$, the following relation holds for a non-loop edge:
  \begin{equation*}
    \raisebox{-15pt}{\marginbox{1.5pt 0}{
\begingroup%
  \makeatletter%
  \providecommand\color[2][]{%
    \errmessage{(Inkscape) Color is used for the text in Inkscape, but the package 'color.sty' is not loaded}%
    \renewcommand\color[2][]{}%
  }%
  \providecommand\transparent[1]{%
    \errmessage{(Inkscape) Transparency is used (non-zero) for the text in Inkscape, but the package 'transparent.sty' is not loaded}%
    \renewcommand\transparent[1]{}%
  }%
  \providecommand\rotatebox[2]{#2}%
  \newcommand*\fsize{\dimexpr\f@size pt\relax}%
  \newcommand*\lineheight[1]{\fontsize{\fsize}{#1\fsize}\selectfont}%
  \ifx\svgwidth\undefined%
    \setlength{\unitlength}{80.16330617bp}%
    \ifx\svgscale\undefined%
      \relax%
    \else%
      \setlength{\unitlength}{\unitlength * \real{\svgscale}}%
    \fi%
  \else%
    \setlength{\unitlength}{\svgwidth}%
  \fi%
  \global\let\svgwidth\undefined%
  \global\let\svgscale\undefined%
  \makeatother%
  \begin{picture}(1,0.43901773)%
    \lineheight{1}%
    \setlength\tabcolsep{0pt}%
    \put(0,0){\includegraphics[width=\unitlength,page=1]{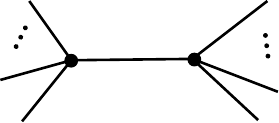}}%
  \end{picture}%
\endgroup%
}}
    =\raisebox{-15pt}{\marginbox{1.5pt 0}{
\begingroup%
  \makeatletter%
  \providecommand\color[2][]{%
    \errmessage{(Inkscape) Color is used for the text in Inkscape, but the package 'color.sty' is not loaded}%
    \renewcommand\color[2][]{}%
  }%
  \providecommand\transparent[1]{%
    \errmessage{(Inkscape) Transparency is used (non-zero) for the text in Inkscape, but the package 'transparent.sty' is not loaded}%
    \renewcommand\transparent[1]{}%
  }%
  \providecommand\rotatebox[2]{#2}%
  \newcommand*\fsize{\dimexpr\f@size pt\relax}%
  \newcommand*\lineheight[1]{\fontsize{\fsize}{#1\fsize}\selectfont}%
  \ifx\svgwidth\undefined%
    \setlength{\unitlength}{44.9400898bp}%
    \ifx\svgscale\undefined%
      \relax%
    \else%
      \setlength{\unitlength}{\unitlength * \real{\svgscale}}%
    \fi%
  \else%
    \setlength{\unitlength}{\svgwidth}%
  \fi%
  \global\let\svgwidth\undefined%
  \global\let\svgscale\undefined%
  \makeatother%
  \begin{picture}(1,0.78868601)%
    \lineheight{1}%
    \setlength\tabcolsep{0pt}%
    \put(0,0){\includegraphics[width=\unitlength,page=1]{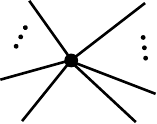}}%
  \end{picture}%
\endgroup%
}}
    -(-1)^m\raisebox{-15pt}{\marginbox{1.5pt 0}{
\begingroup%
  \makeatletter%
  \providecommand\color[2][]{%
    \errmessage{(Inkscape) Color is used for the text in Inkscape, but the package 'color.sty' is not loaded}%
    \renewcommand\color[2][]{}%
  }%
  \providecommand\transparent[1]{%
    \errmessage{(Inkscape) Transparency is used (non-zero) for the text in Inkscape, but the package 'transparent.sty' is not loaded}%
    \renewcommand\transparent[1]{}%
  }%
  \providecommand\rotatebox[2]{#2}%
  \newcommand*\fsize{\dimexpr\f@size pt\relax}%
  \newcommand*\lineheight[1]{\fontsize{\fsize}{#1\fsize}\selectfont}%
  \ifx\svgwidth\undefined%
    \setlength{\unitlength}{51.39968703bp}%
    \ifx\svgscale\undefined%
      \relax%
    \else%
      \setlength{\unitlength}{\unitlength * \real{\svgscale}}%
    \fi%
  \else%
    \setlength{\unitlength}{\svgwidth}%
  \fi%
  \global\let\svgwidth\undefined%
  \global\let\svgscale\undefined%
  \makeatother%
  \begin{picture}(1,0.71842933)%
    \lineheight{1}%
    \setlength\tabcolsep{0pt}%
    \put(0,0){\includegraphics[width=\unitlength,page=1]{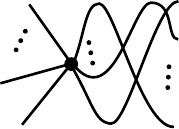}}%
  \end{picture}%
\endgroup%
}}
  \end{equation*}
  where $m$ is the number of half edges incident to the right vertex, not including the edge to be contracted.
  For a loop edge,
  \begin{align*}
    \raisebox{-15pt}{\marginbox{1.5pt 0}{
\begingroup%
  \makeatletter%
  \providecommand\color[2][]{%
    \errmessage{(Inkscape) Color is used for the text in Inkscape, but the package 'color.sty' is not loaded}%
    \renewcommand\color[2][]{}%
  }%
  \providecommand\transparent[1]{%
    \errmessage{(Inkscape) Transparency is used (non-zero) for the text in Inkscape, but the package 'transparent.sty' is not loaded}%
    \renewcommand\transparent[1]{}%
  }%
  \providecommand\rotatebox[2]{#2}%
  \newcommand*\fsize{\dimexpr\f@size pt\relax}%
  \newcommand*\lineheight[1]{\fontsize{\fsize}{#1\fsize}\selectfont}%
  \ifx\svgwidth\undefined%
    \setlength{\unitlength}{44.94009014bp}%
    \ifx\svgscale\undefined%
      \relax%
    \else%
      \setlength{\unitlength}{\unitlength * \real{\svgscale}}%
    \fi%
  \else%
    \setlength{\unitlength}{\svgwidth}%
  \fi%
  \global\let\svgwidth\undefined%
  \global\let\svgscale\undefined%
  \makeatother%
  \begin{picture}(1,0.78868588)%
    \lineheight{1}%
    \setlength\tabcolsep{0pt}%
    \put(0,0){\includegraphics[width=\unitlength,page=1]{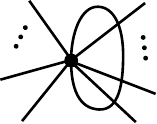}}%
  \end{picture}%
\endgroup%
}}
    =\raisebox{-15pt}{\marginbox{1.5pt 0}{
\begingroup%
  \makeatletter%
  \providecommand\color[2][]{%
    \errmessage{(Inkscape) Color is used for the text in Inkscape, but the package 'color.sty' is not loaded}%
    \renewcommand\color[2][]{}%
  }%
  \providecommand\transparent[1]{%
    \errmessage{(Inkscape) Transparency is used (non-zero) for the text in Inkscape, but the package 'transparent.sty' is not loaded}%
    \renewcommand\transparent[1]{}%
  }%
  \providecommand\rotatebox[2]{#2}%
  \newcommand*\fsize{\dimexpr\f@size pt\relax}%
  \newcommand*\lineheight[1]{\fontsize{\fsize}{#1\fsize}\selectfont}%
  \ifx\svgwidth\undefined%
    \setlength{\unitlength}{59.16329002bp}%
    \ifx\svgscale\undefined%
      \relax%
    \else%
      \setlength{\unitlength}{\unitlength * \real{\svgscale}}%
    \fi%
  \else%
    \setlength{\unitlength}{\svgwidth}%
  \fi%
  \global\let\svgwidth\undefined%
  \global\let\svgscale\undefined%
  \makeatother%
  \begin{picture}(1,0.59484756)%
    \lineheight{1}%
    \setlength\tabcolsep{0pt}%
    \put(0,0){\includegraphics[width=\unitlength,page=1]{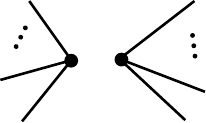}}%
  \end{picture}%
\endgroup%
}}
    -(-1)^m\raisebox{-15pt}{\marginbox{1.5pt 0}{}}
  \end{align*}
  where $m$ is the number of half edges between the two half edges of the loop on the right side.
\end{proposition}
\begin{proof}
  Subdividing edges as necessary,
  \begin{align*}
    \raisebox{-15pt}{\marginbox{1.5pt 0}{
\begingroup%
  \makeatletter%
  \providecommand\color[2][]{%
    \errmessage{(Inkscape) Color is used for the text in Inkscape, but the package 'color.sty' is not loaded}%
    \renewcommand\color[2][]{}%
  }%
  \providecommand\transparent[1]{%
    \errmessage{(Inkscape) Transparency is used (non-zero) for the text in Inkscape, but the package 'transparent.sty' is not loaded}%
    \renewcommand\transparent[1]{}%
  }%
  \providecommand\rotatebox[2]{#2}%
  \newcommand*\fsize{\dimexpr\f@size pt\relax}%
  \newcommand*\lineheight[1]{\fontsize{\fsize}{#1\fsize}\selectfont}%
  \ifx\svgwidth\undefined%
    \setlength{\unitlength}{80.44661768bp}%
    \ifx\svgscale\undefined%
      \relax%
    \else%
      \setlength{\unitlength}{\unitlength * \real{\svgscale}}%
    \fi%
  \else%
    \setlength{\unitlength}{\svgwidth}%
  \fi%
  \global\let\svgwidth\undefined%
  \global\let\svgscale\undefined%
  \makeatother%
  \begin{picture}(1,0.45732441)%
    \lineheight{1}%
    \setlength\tabcolsep{0pt}%
    \put(0,0){\includegraphics[width=\unitlength,page=1]{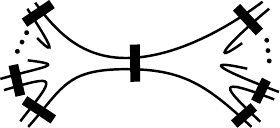}}%
  \end{picture}%
\endgroup%
}}
    &=\raisebox{-15pt}{\marginbox{1.5pt 0}{
\begingroup%
  \makeatletter%
  \providecommand\color[2][]{%
    \errmessage{(Inkscape) Color is used for the text in Inkscape, but the package 'color.sty' is not loaded}%
    \renewcommand\color[2][]{}%
  }%
  \providecommand\transparent[1]{%
    \errmessage{(Inkscape) Transparency is used (non-zero) for the text in Inkscape, but the package 'transparent.sty' is not loaded}%
    \renewcommand\transparent[1]{}%
  }%
  \providecommand\rotatebox[2]{#2}%
  \newcommand*\fsize{\dimexpr\f@size pt\relax}%
  \newcommand*\lineheight[1]{\fontsize{\fsize}{#1\fsize}\selectfont}%
  \ifx\svgwidth\undefined%
    \setlength{\unitlength}{80.44661765bp}%
    \ifx\svgscale\undefined%
      \relax%
    \else%
      \setlength{\unitlength}{\unitlength * \real{\svgscale}}%
    \fi%
  \else%
    \setlength{\unitlength}{\svgwidth}%
  \fi%
  \global\let\svgwidth\undefined%
  \global\let\svgscale\undefined%
  \makeatother%
  \begin{picture}(1,0.45732441)%
    \lineheight{1}%
    \setlength\tabcolsep{0pt}%
    \put(0,0){\includegraphics[width=\unitlength,page=1]{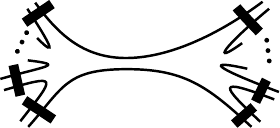}}%
  \end{picture}%
\endgroup%
}}
      -\raisebox{-15pt}{\marginbox{1.5pt 0}{
\begingroup%
  \makeatletter%
  \providecommand\color[2][]{%
    \errmessage{(Inkscape) Color is used for the text in Inkscape, but the package 'color.sty' is not loaded}%
    \renewcommand\color[2][]{}%
  }%
  \providecommand\transparent[1]{%
    \errmessage{(Inkscape) Transparency is used (non-zero) for the text in Inkscape, but the package 'transparent.sty' is not loaded}%
    \renewcommand\transparent[1]{}%
  }%
  \providecommand\rotatebox[2]{#2}%
  \newcommand*\fsize{\dimexpr\f@size pt\relax}%
  \newcommand*\lineheight[1]{\fontsize{\fsize}{#1\fsize}\selectfont}%
  \ifx\svgwidth\undefined%
    \setlength{\unitlength}{80.44661765bp}%
    \ifx\svgscale\undefined%
      \relax%
    \else%
      \setlength{\unitlength}{\unitlength * \real{\svgscale}}%
    \fi%
  \else%
    \setlength{\unitlength}{\svgwidth}%
  \fi%
  \global\let\svgwidth\undefined%
  \global\let\svgscale\undefined%
  \makeatother%
  \begin{picture}(1,0.45732441)%
    \lineheight{1}%
    \setlength\tabcolsep{0pt}%
    \put(0,0){\includegraphics[width=\unitlength,page=1]{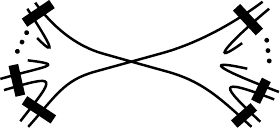}}%
  \end{picture}%
\endgroup%
}}\\
    &=\raisebox{-15pt}{\marginbox{1.5pt 0}{}}
      -(-1)^m\raisebox{-15pt}{\marginbox{1.5pt 0}{
\begingroup%
  \makeatletter%
  \providecommand\color[2][]{%
    \errmessage{(Inkscape) Color is used for the text in Inkscape, but the package 'color.sty' is not loaded}%
    \renewcommand\color[2][]{}%
  }%
  \providecommand\transparent[1]{%
    \errmessage{(Inkscape) Transparency is used (non-zero) for the text in Inkscape, but the package 'transparent.sty' is not loaded}%
    \renewcommand\transparent[1]{}%
  }%
  \providecommand\rotatebox[2]{#2}%
  \newcommand*\fsize{\dimexpr\f@size pt\relax}%
  \newcommand*\lineheight[1]{\fontsize{\fsize}{#1\fsize}\selectfont}%
  \ifx\svgwidth\undefined%
    \setlength{\unitlength}{80.14988689bp}%
    \ifx\svgscale\undefined%
      \relax%
    \else%
      \setlength{\unitlength}{\unitlength * \real{\svgscale}}%
    \fi%
  \else%
    \setlength{\unitlength}{\svgwidth}%
  \fi%
  \global\let\svgwidth\undefined%
  \global\let\svgscale\undefined%
  \makeatother%
  \begin{picture}(1,0.55110977)%
    \lineheight{1}%
    \setlength\tabcolsep{0pt}%
    \put(0,0){\includegraphics[width=\unitlength,page=1]{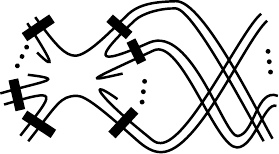}}%
  \end{picture}%
\endgroup%
}}
  \end{align*}
  where the second equality is from flipping the right-hand cluster over by introducing $m+1$ $\raisebox{-6pt}{\marginbox{1.5pt 0pt}{}}$'s, then using the fact that $\raisebox{-6pt}{\marginbox{1.5pt 0pt}{}}\circ\raisebox{-6pt}{\marginbox{1.5pt 0pt}{}}=-\raisebox{-6pt}{\marginbox{1.5pt 0pt}{}}$ for each of the $m$ $\raisebox{-6pt}{\marginbox{1.5pt 0pt}{}}$'s.
  The proof for loop edges is similar, but be aware that when $m=0$ contraction introduces an isolated vertex.
\end{proof}
\begin{corollary}
  $W_{\mathfrak{so}(N)}(G)-W_{\mathfrak{so}(N)}(G/e)
    =W_{\mathfrak{so}(N)}(\delta_e G)-W_{\mathfrak{so}(N)}(\delta_e G/e)$. That is,
  \begin{align*}
    \raisebox{-15pt}{\marginbox{1.5pt 0}{}}
    -\raisebox{-15pt}{\marginbox{1.5pt 0}{}}
    &= \raisebox{-15pt}{\marginbox{1.5pt 0}{}}
      -\raisebox{-15pt}{\marginbox{1.5pt 0}{}}.
  \end{align*}
\end{corollary}

As temporary notation, for an edge $e\in E(G)$ with a distinguished half edge, let $G\curlywedge e$ denote the second term on the right-hand side for each relation, where the distinguished half edge is the side where the rotation is reversed, which for a loop we mean the half edges counterclockwise from it.
The following relations give a recursive procedure to calculate $W_{\mathfrak{so}(N)}(G)$.
\begin{enumerate}
\item If $v\in V(G)$ is an isolated vertex, $W_{\mathfrak{so}(N)}(G)=NW_{\mathfrak{so}(N)}(G-v)$.
\item If $v\in V(G)$ is a vertex of degree $1$, $W_{\mathfrak{so}(N)}(G)=0$.
\item For any edge $e\in E(G)$, $W_{\mathfrak{so}(N)}(G)=W_{\mathfrak{so}(N)}(G/e)-(-1)^m W_{\mathfrak{so}(N)}(G\curlywedge e)$, where $m+1$ is the number incident half edges from the distinguished end of $e$.
\item $W_{\mathfrak{so}(N)}(G_1\amalg G_2)=W_{\mathfrak{so}(N)}(G_1)W_{\mathfrak{so}(N)}(G_2)$
\end{enumerate}
Other rules one might apply include:
\begin{itemize}
\item Subdividing an edge multiplies $W_{\mathfrak{so}(N)}$ by $2$.
  (This is due to the normalization.)
\item If $e$ is a loop whose half edges are adjacent in the rotation, then
  \begin{equation*}
    W_{\mathfrak{so}(N)}(G)=(N-1)W_{\mathfrak{so}(N)}(G-e). 
  \end{equation*}
\item Flipping a degree-$m$ vertex over (the virtual version of move V) multiplies by $(-1)^m$.
\item If $v$ is a degree-$3$ vertex with a loop, $W_{\mathfrak{so}(N)}(G)=0$.
\item The IHX relation from $\ad$-invariance of the Lie bracket:
  $\raisebox{-8pt}{\marginbox{1.5pt 0pt}{
\begingroup%
  \makeatletter%
  \providecommand\color[2][]{%
    \errmessage{(Inkscape) Color is used for the text in Inkscape, but the package 'color.sty' is not loaded}%
    \renewcommand\color[2][]{}%
  }%
  \providecommand\transparent[1]{%
    \errmessage{(Inkscape) Transparency is used (non-zero) for the text in Inkscape, but the package 'transparent.sty' is not loaded}%
    \renewcommand\transparent[1]{}%
  }%
  \providecommand\rotatebox[2]{#2}%
  \newcommand*\fsize{\dimexpr\f@size pt\relax}%
  \newcommand*\lineheight[1]{\fontsize{\fsize}{#1\fsize}\selectfont}%
  \ifx\svgwidth\undefined%
    \setlength{\unitlength}{22.32155614bp}%
    \ifx\svgscale\undefined%
      \relax%
    \else%
      \setlength{\unitlength}{\unitlength * \real{\svgscale}}%
    \fi%
  \else%
    \setlength{\unitlength}{\svgwidth}%
  \fi%
  \global\let\svgwidth\undefined%
  \global\let\svgscale\undefined%
  \makeatother%
  \begin{picture}(1,1.32721177)%
    \lineheight{1}%
    \setlength\tabcolsep{0pt}%
    \put(0,0){\includegraphics[width=\unitlength,page=1]{ad3.pdf}}%
  \end{picture}%
\endgroup%
}}=\raisebox{-8pt}{\marginbox{1.5pt 0pt}{
\begingroup%
  \makeatletter%
  \providecommand\color[2][]{%
    \errmessage{(Inkscape) Color is used for the text in Inkscape, but the package 'color.sty' is not loaded}%
    \renewcommand\color[2][]{}%
  }%
  \providecommand\transparent[1]{%
    \errmessage{(Inkscape) Transparency is used (non-zero) for the text in Inkscape, but the package 'transparent.sty' is not loaded}%
    \renewcommand\transparent[1]{}%
  }%
  \providecommand\rotatebox[2]{#2}%
  \newcommand*\fsize{\dimexpr\f@size pt\relax}%
  \newcommand*\lineheight[1]{\fontsize{\fsize}{#1\fsize}\selectfont}%
  \ifx\svgwidth\undefined%
    \setlength{\unitlength}{22.32655767bp}%
    \ifx\svgscale\undefined%
      \relax%
    \else%
      \setlength{\unitlength}{\unitlength * \real{\svgscale}}%
    \fi%
  \else%
    \setlength{\unitlength}{\svgwidth}%
  \fi%
  \global\let\svgwidth\undefined%
  \global\let\svgscale\undefined%
  \makeatother%
  \begin{picture}(1,1.32691446)%
    \lineheight{1}%
    \setlength\tabcolsep{0pt}%
    \put(0,0){\includegraphics[width=\unitlength,page=1]{ad1.pdf}}%
  \end{picture}%
\endgroup%
}}+\raisebox{-8pt}{\marginbox{1.5pt 0pt}{
\begingroup%
  \makeatletter%
  \providecommand\color[2][]{%
    \errmessage{(Inkscape) Color is used for the text in Inkscape, but the package 'color.sty' is not loaded}%
    \renewcommand\color[2][]{}%
  }%
  \providecommand\transparent[1]{%
    \errmessage{(Inkscape) Transparency is used (non-zero) for the text in Inkscape, but the package 'transparent.sty' is not loaded}%
    \renewcommand\transparent[1]{}%
  }%
  \providecommand\rotatebox[2]{#2}%
  \newcommand*\fsize{\dimexpr\f@size pt\relax}%
  \newcommand*\lineheight[1]{\fontsize{\fsize}{#1\fsize}\selectfont}%
  \ifx\svgwidth\undefined%
    \setlength{\unitlength}{22.51837607bp}%
    \ifx\svgscale\undefined%
      \relax%
    \else%
      \setlength{\unitlength}{\unitlength * \real{\svgscale}}%
    \fi%
  \else%
    \setlength{\unitlength}{\svgwidth}%
  \fi%
  \global\let\svgwidth\undefined%
  \global\let\svgscale\undefined%
  \makeatother%
  \begin{picture}(1,1.31561152)%
    \lineheight{1}%
    \setlength\tabcolsep{0pt}%
    \put(0,0){\includegraphics[width=\unitlength,page=1]{ad2.pdf}}%
  \end{picture}%
\endgroup%
}}$.
\end{itemize}

\begin{proposition}
  If $G_1$ and $G_2$ are virtual graphs,
  \[ 2N(N-1)W_{\mathfrak{so}(N)}(G_1\csum_2 G_2) = W_{\mathfrak{so}(N)}(G_1\amalg G_2), \]
  where the connect sum is performed at degree-$2$ vertices, and
  \[ N(N-1)(N-2)W_{\mathfrak{so}(N)}(G_1\csum_3 G_2) = W_{\mathfrak{so}(N)}(G_1\amalg G_2), \]
  where $N(N-1)(N-2)$ is the $W_{\mathfrak{so}(N)}$ polynomial of the theta graph.
\end{proposition}
\begin{proof}
  The approach is similar to Propositions \ref{prop:s-connect-sum} and \ref{prop:s-vertex-connect-sum}.
  Proposition~\ref{prop:son-contr-del} reduces a graph with two or three boundary edges to a multiple of one with a single interior vertex.
\end{proof}

\begin{figure}[htb]
  \centering
\begingroup%
  \makeatletter%
  \providecommand\color[2][]{%
    \errmessage{(Inkscape) Color is used for the text in Inkscape, but the package 'color.sty' is not loaded}%
    \renewcommand\color[2][]{}%
  }%
  \providecommand\transparent[1]{%
    \errmessage{(Inkscape) Transparency is used (non-zero) for the text in Inkscape, but the package 'transparent.sty' is not loaded}%
    \renewcommand\transparent[1]{}%
  }%
  \providecommand\rotatebox[2]{#2}%
  \newcommand*\fsize{\dimexpr\f@size pt\relax}%
  \newcommand*\lineheight[1]{\fontsize{\fsize}{#1\fsize}\selectfont}%
  \ifx\svgwidth\undefined%
    \setlength{\unitlength}{191.01180413bp}%
    \ifx\svgscale\undefined%
      \relax%
    \else%
      \setlength{\unitlength}{\unitlength * \real{\svgscale}}%
    \fi%
  \else%
    \setlength{\unitlength}{\svgwidth}%
  \fi%
  \global\let\svgwidth\undefined%
  \global\let\svgscale\undefined%
  \makeatother%
  \begin{picture}(1,0.34514879)%
    \lineheight{1}%
    \setlength\tabcolsep{0pt}%
    \put(0.00006135,0.31072594){\color[rgb]{0,0,0}\makebox(0,0)[lt]{\lineheight{0}\smash{\begin{tabular}[t]{l}T1)\end{tabular}}}}%
    \put(0.00006135,0.07419731){\color[rgb]{0,0,0}\makebox(0,0)[lt]{\lineheight{0}\smash{\begin{tabular}[t]{l}T2)\end{tabular}}}}%
    \put(0,0){\includegraphics[width=\unitlength,page=1]{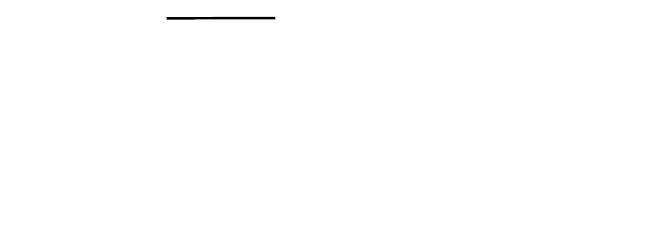}}%
    \put(0.57257419,0.31332811){\color[rgb]{0,0,0}\makebox(0,0)[lt]{\lineheight{0}\smash{\begin{tabular}[t]{l}$\sim$\end{tabular}}}}%
    \put(0,0){\includegraphics[width=\unitlength,page=2]{twist-moves.pdf}}%
    \put(0.57233971,0.08054088){\color[rgb]{0,0,0}\makebox(0,0)[lt]{\lineheight{1.25}\smash{\begin{tabular}[t]{l}$\sim$\end{tabular}}}}%
    \put(0,0){\includegraphics[width=\unitlength,page=3]{twist-moves.pdf}}%
  \end{picture}%
\endgroup%

  \caption{Moves for half twists.}
  \label{fig:twist-moves}
\end{figure}

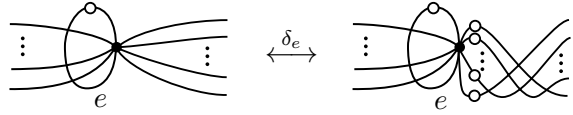
\begin{figure}[htb]
  \centering
\begingroup%
  \makeatletter%
  \providecommand\color[2][]{%
    \errmessage{(Inkscape) Color is used for the text in Inkscape, but the package 'color.sty' is not loaded}%
    \renewcommand\color[2][]{}%
  }%
  \providecommand\transparent[1]{%
    \errmessage{(Inkscape) Transparency is used (non-zero) for the text in Inkscape, but the package 'transparent.sty' is not loaded}%
    \renewcommand\transparent[1]{}%
  }%
  \providecommand\rotatebox[2]{#2}%
  \newcommand*\fsize{\dimexpr\f@size pt\relax}%
  \newcommand*\lineheight[1]{\fontsize{\fsize}{#1\fsize}\selectfont}%
  \ifx\svgwidth\undefined%
    \setlength{\unitlength}{211.43190934bp}%
    \ifx\svgscale\undefined%
      \relax%
    \else%
      \setlength{\unitlength}{\unitlength * \real{\svgscale}}%
    \fi%
  \else%
    \setlength{\unitlength}{\svgwidth}%
  \fi%
  \global\let\svgwidth\undefined%
  \global\let\svgscale\undefined%
  \makeatother%
  \begin{picture}(1,0.19348393)%
    \lineheight{1}%
    \setlength\tabcolsep{0pt}%
    \put(0,0){\includegraphics[width=\unitlength,page=1]{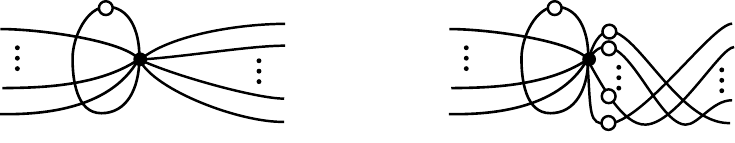}}%
    \put(0.45567489,0.08639406){\color[rgb]{0,0,0}\makebox(0,0)[lt]{\lineheight{1.25}\smash{\begin{tabular}[t]{l}$\overset{\delta_e}{\longleftrightarrow}$\end{tabular}}}}%
    \put(0.1501054,0.00531421){\color[rgb]{0,0,0}\makebox(0,0)[lt]{\lineheight{1.25}\smash{\begin{tabular}[t]{l}$e$\end{tabular}}}}%
    \put(0.75693702,0.00278052){\color[rgb]{0,0,0}\makebox(0,0)[lt]{\lineheight{1.25}\smash{\begin{tabular}[t]{l}$e$\end{tabular}}}}%
  \end{picture}%
\endgroup%

  \caption{The partial duality operation for a loop $e$ with a single twist along it.}
  \label{fig:twist-loop-dual}
\end{figure}

Now we will consider virtual graphs with signed edges, where an \emph{edge signing} is a function $E(G)\to\{\pm 1\}$.
We will take edges to be positive by default, while the negatively signed edges are marked with an open circle.
A negatively signed edge can be interpreted as a \emph{half twist} if we were to consider graphs in unoriented surfaces, and so edge signs are subject to the relations in Figure~\ref{fig:twist-moves}.
In the evaluation of $W_{\mathfrak{so}(N)}$, half twists are replaced with the anti-involution $\raisebox{-6pt}{\marginbox{1.5pt 0pt}{}}$, and the effect of twisting is that $W_{\mathfrak{so}(N)}(\tau_e G)=-W_{\mathfrak{so}(N)}$.
For an edge $e$ in $G$, let $\tau_e G$ denote $G$ with an additional twist along $e$.
When there are no twists along $e$, which can always be arranged when $e$ is not a loop, then $\delta_e G$ is defined as it was for virtual graphs without twists in Figure~\ref{fig:partial-dual}.
When $e$ is a loop with a twist, then by similarly considering a dual graph for a cellular embedding in an unoriented surface, or by considering the arrow presentation, partial duality is given by the more complicated transformation in Figure~\ref{fig:twist-loop-dual}.
Hence, Proposition~\ref{prop:son-contr-del} can be restated in a form closer to that in \cite{Ellis-Monaghan2013} by saying that for any edge $e$ in $G$,
\begin{equation*}
  W_{\mathfrak{so}(N)}(G)=W_{\mathfrak{so}(N)}(\delta_e G-e)-W_{\mathfrak{so}(N)}(\delta_e\tau_e G-e),
\end{equation*}
or, using $G/e=\delta_e G-e$, that $W_{\mathfrak{so}(N)}(G)=W_{\mathfrak{so}(N)}(G/e)-W_{\mathfrak{so}(N)}(\tau_e G/e)$.

By extending $W_{\mathfrak{sl}(N)}$ to accommodate half twists, using the same correspondence that a half twist is replaced with the anti-involution $\raisebox{-6pt}{\marginbox{1.5pt 0pt}{}}$, we obtain the following relationship:
\begin{lemma}
  \label{lem:so-as-sl}
  With a suitable normalization for the Penrose polynomials,
  \begin{align*}
    W_{\mathfrak{so}(N)}(G)
    &=\sum_{S\subset E(G)} (-1)^{\abs{S}} W_{\mathfrak{sl}(N)}(\tau_SG),
  \end{align*}
  where $\tau_S$ is the composition of all $\tau_s$ for $s\in S$.  
\end{lemma}
\begin{proof}
  This follows from the analysis of the primitive central idempotents for $Br^N_2$.  In particular,
  \begin{align*}
    \left(\frac{1}{2}\raisebox{-6pt}{\marginbox{1.5pt 0pt}{}}-\frac{1}{2}\raisebox{-6pt}{\marginbox{1.5pt 0pt}{}}\right)
    \circ\left(\raisebox{-6pt}{\marginbox{1.5pt 0pt}{}}-\frac{1}{N}\raisebox{-6pt}{\marginbox{1.5pt 0pt}{}}\right)
    &=\frac{1}{2}\raisebox{-6pt}{\marginbox{1.5pt 0pt}{}}-\frac{1}{2}\raisebox{-6pt}{\marginbox{1.5pt 0pt}{}}.
  \end{align*}
\end{proof}

\subsection{Relations for \texorpdfstring{$W_{\mathfrak{sl}(N)}$}{Wsl(N)}}
\label{sec:sln-rels}

In this section, we extend the $W_{\mathfrak{sl}(N)}$ polynomial to non-cubic signed virtual graphs, derive some contraction-deletion-like rules, and then carry over Bar-Natan's results about planarity and the $\mathfrak{sl}(2)$ specialization.

For $v$ a vertex in a virtual graph $G$, let $\sigma_v G$ be the graph obtained from reversing the rotation system at $v$, and for $W\subset V(G)$, let $\sigma_W$ be the composition of all $\sigma_v$ for $v\in W$.
\begin{lemma}
  \label{lem:sl-as-s}
  With a suitable normalization, we can give $W_{\mathfrak{sl}(N)}$ in terms of the $S$-polynomial for a cubic virtual graph $G$ as
  \begin{equation*}
    W_{\mathfrak{sl}(N)}(G)=
    \sum_{W\subset V(G)}(-1)^{\abs{W}}S_{\sigma_W(G)}(N^2).
  \end{equation*}
\end{lemma}
\begin{proof}
  This is a state sum from expanding the vertices as $\raisebox{-8pt}{\marginbox{1.5pt 0pt}{}}$ and $-\raisebox{-8pt}{\marginbox{1.5pt 0pt}{}}$.
\end{proof}

We now extend the $W_{\mathfrak{sl}(N)}$ polynomial to non-cubic virtual graphs using the $S$-polynomial.
To get contraction-deletion-like relations, we deal with signed virtual graphs.
\begin{definition}
  A \emph{signed} virtual graph $G$ is a virtual graph along with a function $s:V(G)\to\{\pm 1\}$ assigning a sign to each vertex.
  When considering a signed virtual spatial graph, classical crossings are not given a sign.
  We will usually consider $s(v)=(-1)^{\deg(v)}$.
\end{definition}

Murakami in \cite{Murakami1993} defines an invariant $Z_s$ of signed spatial graphs in terms of the HOMFLY polynomial.
Recall that the HOMFLY polynomial is an invariant of oriented links in $S^3$ and is a Laurent polynomial in $\mathbb{Z}[(q-q^{-1})^{\pm 1},a^{\pm 1}]$ given by the skein relation
\begin{equation*}
  a^{-1}\raisebox{-6pt}{\marginbox{1.5pt 0pt}{
\begingroup%
  \makeatletter%
  \providecommand\color[2][]{%
    \errmessage{(Inkscape) Color is used for the text in Inkscape, but the package 'color.sty' is not loaded}%
    \renewcommand\color[2][]{}%
  }%
  \providecommand\transparent[1]{%
    \errmessage{(Inkscape) Transparency is used (non-zero) for the text in Inkscape, but the package 'transparent.sty' is not loaded}%
    \renewcommand\transparent[1]{}%
  }%
  \providecommand\rotatebox[2]{#2}%
  \newcommand*\fsize{\dimexpr\f@size pt\relax}%
  \newcommand*\lineheight[1]{\fontsize{\fsize}{#1\fsize}\selectfont}%
  \ifx\svgwidth\undefined%
    \setlength{\unitlength}{10.82395958bp}%
    \ifx\svgscale\undefined%
      \relax%
    \else%
      \setlength{\unitlength}{\unitlength * \real{\svgscale}}%
    \fi%
  \else%
    \setlength{\unitlength}{\svgwidth}%
  \fi%
  \global\let\svgwidth\undefined%
  \global\let\svgscale\undefined%
  \makeatother%
  \begin{picture}(1,1.80573815)%
    \lineheight{1}%
    \setlength\tabcolsep{0pt}%
    \put(0,0){\includegraphics[width=\unitlength,page=1]{cr-right.pdf}}%
  \end{picture}%
\endgroup%
}}-a\raisebox{-6pt}{\marginbox{1.5pt 0pt}{
\begingroup%
  \makeatletter%
  \providecommand\color[2][]{%
    \errmessage{(Inkscape) Color is used for the text in Inkscape, but the package 'color.sty' is not loaded}%
    \renewcommand\color[2][]{}%
  }%
  \providecommand\transparent[1]{%
    \errmessage{(Inkscape) Transparency is used (non-zero) for the text in Inkscape, but the package 'transparent.sty' is not loaded}%
    \renewcommand\transparent[1]{}%
  }%
  \providecommand\rotatebox[2]{#2}%
  \newcommand*\fsize{\dimexpr\f@size pt\relax}%
  \newcommand*\lineheight[1]{\fontsize{\fsize}{#1\fsize}\selectfont}%
  \ifx\svgwidth\undefined%
    \setlength{\unitlength}{10.82395958bp}%
    \ifx\svgscale\undefined%
      \relax%
    \else%
      \setlength{\unitlength}{\unitlength * \real{\svgscale}}%
    \fi%
  \else%
    \setlength{\unitlength}{\svgwidth}%
  \fi%
  \global\let\svgwidth\undefined%
  \global\let\svgscale\undefined%
  \makeatother%
  \begin{picture}(1,1.80573815)%
    \lineheight{1}%
    \setlength\tabcolsep{0pt}%
    \put(0,0){\includegraphics[width=\unitlength,page=1]{cr-left.pdf}}%
  \end{picture}%
\endgroup%
}}=(q-q^{-1})\raisebox{-6pt}{\marginbox{1.5pt 0pt}{
\begingroup%
  \makeatletter%
  \providecommand\color[2][]{%
    \errmessage{(Inkscape) Color is used for the text in Inkscape, but the package 'color.sty' is not loaded}%
    \renewcommand\color[2][]{}%
  }%
  \providecommand\transparent[1]{%
    \errmessage{(Inkscape) Transparency is used (non-zero) for the text in Inkscape, but the package 'transparent.sty' is not loaded}%
    \renewcommand\transparent[1]{}%
  }%
  \providecommand\rotatebox[2]{#2}%
  \newcommand*\fsize{\dimexpr\f@size pt\relax}%
  \newcommand*\lineheight[1]{\fontsize{\fsize}{#1\fsize}\selectfont}%
  \ifx\svgwidth\undefined%
    \setlength{\unitlength}{10.8110676bp}%
    \ifx\svgscale\undefined%
      \relax%
    \else%
      \setlength{\unitlength}{\unitlength * \real{\svgscale}}%
    \fi%
  \else%
    \setlength{\unitlength}{\svgwidth}%
  \fi%
  \global\let\svgwidth\undefined%
  \global\let\svgscale\undefined%
  \makeatother%
  \begin{picture}(1,1.80802073)%
    \lineheight{1}%
    \setlength\tabcolsep{0pt}%
    \put(0,0){\includegraphics[width=\unitlength,page=1]{cr-zero.pdf}}%
  \end{picture}%
\endgroup%
}}
\end{equation*}
with the normalization that the HOMFLY polynomial of the unknot is $1$.
The HOMFLY polynomial of the disjoint union of an unknot and a link $L$ is that of $L$ times a factor of $\frac{a^{-1}-a}{q-q^{-1}}$.
The Murakami polynomial is given by replacing each vertex $v$ by
\begin{equation*}
  \raisebox{-14pt}{\marginbox{1.5pt 0}{}}
  \longmapsto
  \raisebox{-14pt}{\marginbox{1.5pt 0}{
\begingroup%
  \makeatletter%
  \providecommand\color[2][]{%
    \errmessage{(Inkscape) Color is used for the text in Inkscape, but the package 'color.sty' is not loaded}%
    \renewcommand\color[2][]{}%
  }%
  \providecommand\transparent[1]{%
    \errmessage{(Inkscape) Transparency is used (non-zero) for the text in Inkscape, but the package 'transparent.sty' is not loaded}%
    \renewcommand\transparent[1]{}%
  }%
  \providecommand\rotatebox[2]{#2}%
  \newcommand*\fsize{\dimexpr\f@size pt\relax}%
  \newcommand*\lineheight[1]{\fontsize{\fsize}{#1\fsize}\selectfont}%
  \ifx\svgwidth\undefined%
    \setlength{\unitlength}{41.38394479bp}%
    \ifx\svgscale\undefined%
      \relax%
    \else%
      \setlength{\unitlength}{\unitlength * \real{\svgscale}}%
    \fi%
  \else%
    \setlength{\unitlength}{\svgwidth}%
  \fi%
  \global\let\svgwidth\undefined%
  \global\let\svgscale\undefined%
  \makeatother%
  \begin{picture}(1,0.95805197)%
    \lineheight{1}%
    \setlength\tabcolsep{0pt}%
    \put(0,0){\includegraphics[width=\unitlength,page=1]{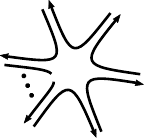}}%
  \end{picture}%
\endgroup%
}}
  +a^{-\deg(v)}s(v)
  \raisebox{-14pt}{\marginbox{1.5pt 0}{
\begingroup%
  \makeatletter%
  \providecommand\color[2][]{%
    \errmessage{(Inkscape) Color is used for the text in Inkscape, but the package 'color.sty' is not loaded}%
    \renewcommand\color[2][]{}%
  }%
  \providecommand\transparent[1]{%
    \errmessage{(Inkscape) Transparency is used (non-zero) for the text in Inkscape, but the package 'transparent.sty' is not loaded}%
    \renewcommand\transparent[1]{}%
  }%
  \providecommand\rotatebox[2]{#2}%
  \newcommand*\fsize{\dimexpr\f@size pt\relax}%
  \newcommand*\lineheight[1]{\fontsize{\fsize}{#1\fsize}\selectfont}%
  \ifx\svgwidth\undefined%
    \setlength{\unitlength}{41.3401239bp}%
    \ifx\svgscale\undefined%
      \relax%
    \else%
      \setlength{\unitlength}{\unitlength * \real{\svgscale}}%
    \fi%
  \else%
    \setlength{\unitlength}{\svgwidth}%
  \fi%
  \global\let\svgwidth\undefined%
  \global\let\svgscale\undefined%
  \makeatother%
  \begin{picture}(1,0.96109576)%
    \lineheight{1}%
    \setlength\tabcolsep{0pt}%
    \put(0,0){\includegraphics[width=\unitlength,page=1]{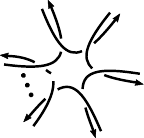}}%
  \end{picture}%
\endgroup%
}},
\end{equation*}
each edge by the idempotent
\begin{equation*}
  \raisebox{-6pt}{\marginbox{1.5pt 0pt}{
\begingroup%
  \makeatletter%
  \providecommand\color[2][]{%
    \errmessage{(Inkscape) Color is used for the text in Inkscape, but the package 'color.sty' is not loaded}%
    \renewcommand\color[2][]{}%
  }%
  \providecommand\transparent[1]{%
    \errmessage{(Inkscape) Transparency is used (non-zero) for the text in Inkscape, but the package 'transparent.sty' is not loaded}%
    \renewcommand\transparent[1]{}%
  }%
  \providecommand\rotatebox[2]{#2}%
  \newcommand*\fsize{\dimexpr\f@size pt\relax}%
  \newcommand*\lineheight[1]{\fontsize{\fsize}{#1\fsize}\selectfont}%
  \ifx\svgwidth\undefined%
    \setlength{\unitlength}{11.51953029bp}%
    \ifx\svgscale\undefined%
      \relax%
    \else%
      \setlength{\unitlength}{\unitlength * \real{\svgscale}}%
    \fi%
  \else%
    \setlength{\unitlength}{\svgwidth}%
  \fi%
  \global\let\svgwidth\undefined%
  \global\let\svgscale\undefined%
  \makeatother%
  \begin{picture}(1,1.77482522)%
    \lineheight{1}%
    \setlength\tabcolsep{0pt}%
    \put(0,0){\includegraphics[width=\unitlength,page=1]{cr-id.pdf}}%
  \end{picture}%
\endgroup%
}}-\frac{q-q^{-1}}{a^{-1}-a}\raisebox{-6pt}{\marginbox{1.5pt 0pt}{
\begingroup%
  \makeatletter%
  \providecommand\color[2][]{%
    \errmessage{(Inkscape) Color is used for the text in Inkscape, but the package 'color.sty' is not loaded}%
    \renewcommand\color[2][]{}%
  }%
  \providecommand\transparent[1]{%
    \errmessage{(Inkscape) Transparency is used (non-zero) for the text in Inkscape, but the package 'transparent.sty' is not loaded}%
    \renewcommand\transparent[1]{}%
  }%
  \providecommand\rotatebox[2]{#2}%
  \newcommand*\fsize{\dimexpr\f@size pt\relax}%
  \newcommand*\lineheight[1]{\fontsize{\fsize}{#1\fsize}\selectfont}%
  \ifx\svgwidth\undefined%
    \setlength{\unitlength}{11.51816753bp}%
    \ifx\svgscale\undefined%
      \relax%
    \else%
      \setlength{\unitlength}{\unitlength * \real{\svgscale}}%
    \fi%
  \else%
    \setlength{\unitlength}{\svgwidth}%
  \fi%
  \global\let\svgwidth\undefined%
  \global\let\svgscale\undefined%
  \makeatother%
  \begin{picture}(1,1.77504669)%
    \lineheight{1}%
    \setlength\tabcolsep{0pt}%
    \put(0,0){\includegraphics[width=\unitlength,page=1]{cr-cap.pdf}}%
  \end{picture}%
\endgroup%
}},
\end{equation*}
and each classical crossing by $\raisebox{-6pt}{\marginbox{1.5pt 0pt}{
\begingroup%
  \makeatletter%
  \providecommand\color[2][]{%
    \errmessage{(Inkscape) Color is used for the text in Inkscape, but the package 'color.sty' is not loaded}%
    \renewcommand\color[2][]{}%
  }%
  \providecommand\transparent[1]{%
    \errmessage{(Inkscape) Transparency is used (non-zero) for the text in Inkscape, but the package 'transparent.sty' is not loaded}%
    \renewcommand\transparent[1]{}%
  }%
  \providecommand\rotatebox[2]{#2}%
  \newcommand*\fsize{\dimexpr\f@size pt\relax}%
  \newcommand*\lineheight[1]{\fontsize{\fsize}{#1\fsize}\selectfont}%
  \ifx\svgwidth\undefined%
    \setlength{\unitlength}{17.11504129bp}%
    \ifx\svgscale\undefined%
      \relax%
    \else%
      \setlength{\unitlength}{\unitlength * \real{\svgscale}}%
    \fi%
  \else%
    \setlength{\unitlength}{\svgwidth}%
  \fi%
  \global\let\svgwidth\undefined%
  \global\let\svgscale\undefined%
  \makeatother%
  \begin{picture}(1,1.20492949)%
    \lineheight{1}%
    \setlength\tabcolsep{0pt}%
    \put(0,0){\includegraphics[width=\unitlength,page=1]{cr-dright.pdf}}%
  \end{picture}%
\endgroup%
}}$.

When $a=q^{-N}$ for $N\in\mathbb{N}$, the HOMFLY loop factor is the quantum integer $[N]_q$.
If $q\to 1$, then the skein relation reduces to $\raisebox{-6pt}{\marginbox{1.5pt 0pt}{}}=\raisebox{-6pt}{\marginbox{1.5pt 0pt}{}}$ with loops giving a factor of $[N]_1=N$, and so this evaluation of the HOMFLY polynomial can be thought of as taking place in $\Br^{N}$, with the edge idempotent reducing to the Jones-Wenzl projector $P^{(2)}$.
In fact, the Murakami polynomial for a signed cubic spatial graph, all of whose vertex signs are $-1$, gives at this evaluation a normalization of $W_{\mathfrak{sl}(N)}$ for the underlying virtual graph (the ribbon graph).
Conversely, by assigning arbitrary classical crossings to virtual crossings in a diagram of a virtual graph, the Murakami polynomial gives a way to extend $W_{\mathfrak{sl}(N)}$ to non-cubic virtual graphs.
We will deal exclusively with this evaluation as we work with virtual graphs.

With the normalization as in Lemma~\ref{lem:sl-as-s}, this evaluation of the Murakami polynomial in $\Br^N$ can be given functorially from the category of signed virtual graphs as
\begin{align*}
    \raisebox{-14pt}{\marginbox{1.5pt 0}{}}
  &\longmapsto
  N^{-1}\left(\raisebox{-14pt}{\marginbox{1.5pt 0}{
\begingroup%
  \makeatletter%
  \providecommand\color[2][]{%
    \errmessage{(Inkscape) Color is used for the text in Inkscape, but the package 'color.sty' is not loaded}%
    \renewcommand\color[2][]{}%
  }%
  \providecommand\transparent[1]{%
    \errmessage{(Inkscape) Transparency is used (non-zero) for the text in Inkscape, but the package 'transparent.sty' is not loaded}%
    \renewcommand\transparent[1]{}%
  }%
  \providecommand\rotatebox[2]{#2}%
  \newcommand*\fsize{\dimexpr\f@size pt\relax}%
  \newcommand*\lineheight[1]{\fontsize{\fsize}{#1\fsize}\selectfont}%
  \ifx\svgwidth\undefined%
    \setlength{\unitlength}{39.20439227bp}%
    \ifx\svgscale\undefined%
      \relax%
    \else%
      \setlength{\unitlength}{\unitlength * \real{\svgscale}}%
    \fi%
  \else%
    \setlength{\unitlength}{\svgwidth}%
  \fi%
  \global\let\svgwidth\undefined%
  \global\let\svgscale\undefined%
  \makeatother%
  \begin{picture}(1,0.95861196)%
    \lineheight{1}%
    \setlength\tabcolsep{0pt}%
    \put(0,0){\includegraphics[width=\unitlength,page=1]{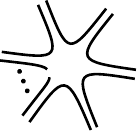}}%
  \end{picture}%
\endgroup%
}}
  +s(v)
    \raisebox{-14pt}{\marginbox{1.5pt 0}{
\begingroup%
  \makeatletter%
  \providecommand\color[2][]{%
    \errmessage{(Inkscape) Color is used for the text in Inkscape, but the package 'color.sty' is not loaded}%
    \renewcommand\color[2][]{}%
  }%
  \providecommand\transparent[1]{%
    \errmessage{(Inkscape) Transparency is used (non-zero) for the text in Inkscape, but the package 'transparent.sty' is not loaded}%
    \renewcommand\transparent[1]{}%
  }%
  \providecommand\rotatebox[2]{#2}%
  \newcommand*\fsize{\dimexpr\f@size pt\relax}%
  \newcommand*\lineheight[1]{\fontsize{\fsize}{#1\fsize}\selectfont}%
  \ifx\svgwidth\undefined%
    \setlength{\unitlength}{39.21044821bp}%
    \ifx\svgscale\undefined%
      \relax%
    \else%
      \setlength{\unitlength}{\unitlength * \real{\svgscale}}%
    \fi%
  \else%
    \setlength{\unitlength}{\svgwidth}%
  \fi%
  \global\let\svgwidth\undefined%
  \global\let\svgscale\undefined%
  \makeatother%
  \begin{picture}(1,0.95814899)%
    \lineheight{1}%
    \setlength\tabcolsep{0pt}%
    \put(0,0){\includegraphics[width=\unitlength,page=1]{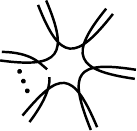}}%
  \end{picture}%
\endgroup%
}}\right)\\
\end{align*}
with edges being sent to $N\raisebox{-6pt}{\marginbox{1.5pt 0pt}{}}-\raisebox{-6pt}{\marginbox{1.5pt 0pt}{}}$.
This functor factors through $\Sc^{N^2}$.
In this category, define the following symbol for the element induced by the order reversing permutation:
\begin{equation*}
  \raisebox{-17pt}{\marginbox{4pt 0}{
\begingroup%
  \makeatletter%
  \providecommand\color[2][]{%
    \errmessage{(Inkscape) Color is used for the text in Inkscape, but the package 'color.sty' is not loaded}%
    \renewcommand\color[2][]{}%
  }%
  \providecommand\transparent[1]{%
    \errmessage{(Inkscape) Transparency is used (non-zero) for the text in Inkscape, but the package 'transparent.sty' is not loaded}%
    \renewcommand\transparent[1]{}%
  }%
  \providecommand\rotatebox[2]{#2}%
  \newcommand*\fsize{\dimexpr\f@size pt\relax}%
  \newcommand*\lineheight[1]{\fontsize{\fsize}{#1\fsize}\selectfont}%
  \ifx\svgwidth\undefined%
    \setlength{\unitlength}{38.57390958bp}%
    \ifx\svgscale\undefined%
      \relax%
    \else%
      \setlength{\unitlength}{\unitlength * \real{\svgscale}}%
    \fi%
  \else%
    \setlength{\unitlength}{\svgwidth}%
  \fi%
  \global\let\svgwidth\undefined%
  \global\let\svgscale\undefined%
  \makeatother%
  \begin{picture}(1,1.04955658)%
    \lineheight{1}%
    \setlength\tabcolsep{0pt}%
    \put(0,0){\includegraphics[width=\unitlength,page=1]{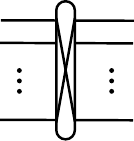}}%
  \end{picture}%
\endgroup%
}}
  =
  \raisebox{-12pt}{\marginbox{4pt 0}{
\begingroup%
  \makeatletter%
  \providecommand\color[2][]{%
    \errmessage{(Inkscape) Color is used for the text in Inkscape, but the package 'color.sty' is not loaded}%
    \renewcommand\color[2][]{}%
  }%
  \providecommand\transparent[1]{%
    \errmessage{(Inkscape) Transparency is used (non-zero) for the text in Inkscape, but the package 'transparent.sty' is not loaded}%
    \renewcommand\transparent[1]{}%
  }%
  \providecommand\rotatebox[2]{#2}%
  \newcommand*\fsize{\dimexpr\f@size pt\relax}%
  \newcommand*\lineheight[1]{\fontsize{\fsize}{#1\fsize}\selectfont}%
  \ifx\svgwidth\undefined%
    \setlength{\unitlength}{38.46091029bp}%
    \ifx\svgscale\undefined%
      \relax%
    \else%
      \setlength{\unitlength}{\unitlength * \real{\svgscale}}%
    \fi%
  \else%
    \setlength{\unitlength}{\svgwidth}%
  \fi%
  \global\let\svgwidth\undefined%
  \global\let\svgscale\undefined%
  \makeatother%
  \begin{picture}(1,0.76331007)%
    \lineheight{1}%
    \setlength\tabcolsep{0pt}%
    \put(0,0){\includegraphics[width=\unitlength,page=1]{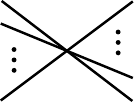}}%
  \end{picture}%
\endgroup%
}}.
\end{equation*}
Define for $a=\pm 1$ the following signed vertex elements of all degrees:
\begin{equation*}
  \raisebox{-21pt}{\marginbox{4pt 0}{
\begingroup%
  \makeatletter%
  \providecommand\color[2][]{%
    \errmessage{(Inkscape) Color is used for the text in Inkscape, but the package 'color.sty' is not loaded}%
    \renewcommand\color[2][]{}%
  }%
  \providecommand\transparent[1]{%
    \errmessage{(Inkscape) Transparency is used (non-zero) for the text in Inkscape, but the package 'transparent.sty' is not loaded}%
    \renewcommand\transparent[1]{}%
  }%
  \providecommand\rotatebox[2]{#2}%
  \newcommand*\fsize{\dimexpr\f@size pt\relax}%
  \newcommand*\lineheight[1]{\fontsize{\fsize}{#1\fsize}\selectfont}%
  \ifx\svgwidth\undefined%
    \setlength{\unitlength}{34.77544598bp}%
    \ifx\svgscale\undefined%
      \relax%
    \else%
      \setlength{\unitlength}{\unitlength * \real{\svgscale}}%
    \fi%
  \else%
    \setlength{\unitlength}{\svgwidth}%
  \fi%
  \global\let\svgwidth\undefined%
  \global\let\svgscale\undefined%
  \makeatother%
  \begin{picture}(1,1.14416332)%
    \lineheight{1}%
    \setlength\tabcolsep{0pt}%
    \put(0,0){\includegraphics[width=\unitlength,page=1]{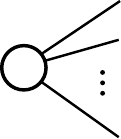}}%
    \put(0.11948407,0.48919745){\color[rgb]{0,0,0}\makebox(0,0)[lt]{\lineheight{0}\smash{\begin{tabular}[t]{l}$a$\end{tabular}}}}%
  \end{picture}%
\endgroup%
}}
  =
  \raisebox{-21pt}{\marginbox{4pt 0}{
\begingroup%
  \makeatletter%
  \providecommand\color[2][]{%
    \errmessage{(Inkscape) Color is used for the text in Inkscape, but the package 'color.sty' is not loaded}%
    \renewcommand\color[2][]{}%
  }%
  \providecommand\transparent[1]{%
    \errmessage{(Inkscape) Transparency is used (non-zero) for the text in Inkscape, but the package 'transparent.sty' is not loaded}%
    \renewcommand\transparent[1]{}%
  }%
  \providecommand\rotatebox[2]{#2}%
  \newcommand*\fsize{\dimexpr\f@size pt\relax}%
  \newcommand*\lineheight[1]{\fontsize{\fsize}{#1\fsize}\selectfont}%
  \ifx\svgwidth\undefined%
    \setlength{\unitlength}{30.40327399bp}%
    \ifx\svgscale\undefined%
      \relax%
    \else%
      \setlength{\unitlength}{\unitlength * \real{\svgscale}}%
    \fi%
  \else%
    \setlength{\unitlength}{\svgwidth}%
  \fi%
  \global\let\svgwidth\undefined%
  \global\let\svgscale\undefined%
  \makeatother%
  \begin{picture}(1,1.30870082)%
    \lineheight{1}%
    \setlength\tabcolsep{0pt}%
    \put(0,0){\includegraphics[width=\unitlength,page=1]{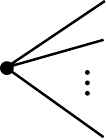}}%
  \end{picture}%
\endgroup%
}}
  +a
  \raisebox{-21pt}{\marginbox{4pt 0}{
\begingroup%
  \makeatletter%
  \providecommand\color[2][]{%
    \errmessage{(Inkscape) Color is used for the text in Inkscape, but the package 'color.sty' is not loaded}%
    \renewcommand\color[2][]{}%
  }%
  \providecommand\transparent[1]{%
    \errmessage{(Inkscape) Transparency is used (non-zero) for the text in Inkscape, but the package 'transparent.sty' is not loaded}%
    \renewcommand\transparent[1]{}%
  }%
  \providecommand\rotatebox[2]{#2}%
  \newcommand*\fsize{\dimexpr\f@size pt\relax}%
  \newcommand*\lineheight[1]{\fontsize{\fsize}{#1\fsize}\selectfont}%
  \ifx\svgwidth\undefined%
    \setlength{\unitlength}{30.40327754bp}%
    \ifx\svgscale\undefined%
      \relax%
    \else%
      \setlength{\unitlength}{\unitlength * \real{\svgscale}}%
    \fi%
  \else%
    \setlength{\unitlength}{\svgwidth}%
  \fi%
  \global\let\svgwidth\undefined%
  \global\let\svgscale\undefined%
  \makeatother%
  \begin{picture}(1,1.34021451)%
    \lineheight{1}%
    \setlength\tabcolsep{0pt}%
    \put(0,0){\includegraphics[width=\unitlength,page=1]{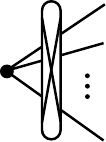}}%
  \end{picture}%
\endgroup%
}}.
\end{equation*}
Then by replacing each signed vertex $v$ with the $a=s(v)$ element, one gets the same evaluation of the Murakami polynomial, which for negatively signed cubic virtual graphs agrees with Lemma~\ref{lem:sl-as-s}.

\begin{definition}
  \label{lem:sl-as-s-extended}
  Let $G$ be a signed virtual graph.
  Define $W_{\mathfrak{sl}(N)}(G)$ using the above $S$-polynomial expansion.
  By expanding the signed vertex elements, we can write
  \begin{equation*}
    W_{\mathfrak{sl}(N)}(G)=\sum_{W\subset V(G)}\left(\prod_{v\in W}s(v)\right)S_{\sigma_W(G)}(N^2),
  \end{equation*}
  which agrees with the usual $W_{\mathfrak{sl}(N)}$ when $G$ is a cubic graph with negatively signed edges.
\end{definition}

\begin{remark}
  The expansion for a normalization of $W_{\mathfrak{sl}(N)}$ can also be expressed as follows.
  For each $n\in\mathbb{N}$ and $s\in\{\pm 1\}$
  there are invariant forms $\alpha^s_n:\mathfrak{sl}(N)^{\otimes n}\to \mathbb{C}$ on the adjoint representation  defined by $\alpha^s_n(A_1,\dots,A_n)=\tr(A_1\cdots A_n+s A_n\cdots A_1)$.
  In particular, $\alpha^{-1}_3(A_1,A_2,A_3)=c B(A_1,[A_2,A_3])$ for some constant $c$, where $B$ is the Killing form.
  By placing such operators at each vertex of a signed virtual graph, one can contract the edges using the Casimir element for $\mathfrak{sl}(N)$, yielding an element of $\mathbb{C}$.
\end{remark}

\begin{theorem}
  \label{thm:sl-edge-contr-del}
  Let $a,b\in\{\pm 1\}$.
  The extended $W_{\mathfrak{sl}(N)}$ polynomial satisfies
  \begin{equation*}
    \raisebox{-21pt}{\marginbox{4pt 0}{
\begingroup%
  \makeatletter%
  \providecommand\color[2][]{%
    \errmessage{(Inkscape) Color is used for the text in Inkscape, but the package 'color.sty' is not loaded}%
    \renewcommand\color[2][]{}%
  }%
  \providecommand\transparent[1]{%
    \errmessage{(Inkscape) Transparency is used (non-zero) for the text in Inkscape, but the package 'transparent.sty' is not loaded}%
    \renewcommand\transparent[1]{}%
  }%
  \providecommand\rotatebox[2]{#2}%
  \newcommand*\fsize{\dimexpr\f@size pt\relax}%
  \newcommand*\lineheight[1]{\fontsize{\fsize}{#1\fsize}\selectfont}%
  \ifx\svgwidth\undefined%
    \setlength{\unitlength}{76.72208059bp}%
    \ifx\svgscale\undefined%
      \relax%
    \else%
      \setlength{\unitlength}{\unitlength * \real{\svgscale}}%
    \fi%
  \else%
    \setlength{\unitlength}{\svgwidth}%
  \fi%
  \global\let\svgwidth\undefined%
  \global\let\svgscale\undefined%
  \makeatother%
  \begin{picture}(1,0.59446466)%
    \lineheight{1}%
    \setlength\tabcolsep{0pt}%
    \put(0,0){\includegraphics[width=\unitlength,page=1]{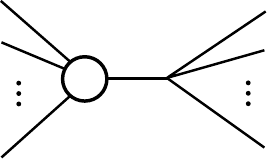}}%
    \put(0.28254586,0.2558121){\color[rgb]{0,0,0}\makebox(0,0)[lt]{\lineheight{1.25}\smash{\begin{tabular}[t]{l}$a$\end{tabular}}}}%
    \put(0,0){\includegraphics[width=\unitlength,page=2]{signed-vtx-a-b.pdf}}%
    \put(0.61142287,0.25319377){\color[rgb]{0,0,0}\makebox(0,0)[lt]{\lineheight{1.25}\smash{\begin{tabular}[t]{l}$b$\end{tabular}}}}%
  \end{picture}%
\endgroup%
}}
    =
    \raisebox{-21pt}{\marginbox{4pt 0}{
\begingroup%
  \makeatletter%
  \providecommand\color[2][]{%
    \errmessage{(Inkscape) Color is used for the text in Inkscape, but the package 'color.sty' is not loaded}%
    \renewcommand\color[2][]{}%
  }%
  \providecommand\transparent[1]{%
    \errmessage{(Inkscape) Transparency is used (non-zero) for the text in Inkscape, but the package 'transparent.sty' is not loaded}%
    \renewcommand\transparent[1]{}%
  }%
  \providecommand\rotatebox[2]{#2}%
  \newcommand*\fsize{\dimexpr\f@size pt\relax}%
  \newcommand*\lineheight[1]{\fontsize{\fsize}{#1\fsize}\selectfont}%
  \ifx\svgwidth\undefined%
    \setlength{\unitlength}{54.22212192bp}%
    \ifx\svgscale\undefined%
      \relax%
    \else%
      \setlength{\unitlength}{\unitlength * \real{\svgscale}}%
    \fi%
  \else%
    \setlength{\unitlength}{\svgwidth}%
  \fi%
  \global\let\svgwidth\undefined%
  \global\let\svgscale\undefined%
  \makeatother%
  \begin{picture}(1,0.84114303)%
    \lineheight{1}%
    \setlength\tabcolsep{0pt}%
    \put(0,0){\includegraphics[width=\unitlength,page=1]{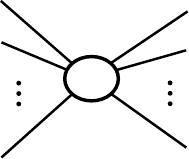}}%
    \put(0.38038788,0.37043635){\color[rgb]{0,0,0}\makebox(0,0)[lt]{\lineheight{1.25}\smash{\begin{tabular}[t]{l}$ab$\end{tabular}}}}%
  \end{picture}%
\endgroup%
}}
    +b
    \raisebox{-21pt}{\marginbox{4pt 0}{
\begingroup%
  \makeatletter%
  \providecommand\color[2][]{%
    \errmessage{(Inkscape) Color is used for the text in Inkscape, but the package 'color.sty' is not loaded}%
    \renewcommand\color[2][]{}%
  }%
  \providecommand\transparent[1]{%
    \errmessage{(Inkscape) Transparency is used (non-zero) for the text in Inkscape, but the package 'transparent.sty' is not loaded}%
    \renewcommand\transparent[1]{}%
  }%
  \providecommand\rotatebox[2]{#2}%
  \newcommand*\fsize{\dimexpr\f@size pt\relax}%
  \newcommand*\lineheight[1]{\fontsize{\fsize}{#1\fsize}\selectfont}%
  \ifx\svgwidth\undefined%
    \setlength{\unitlength}{54.22214213bp}%
    \ifx\svgscale\undefined%
      \relax%
    \else%
      \setlength{\unitlength}{\unitlength * \real{\svgscale}}%
    \fi%
  \else%
    \setlength{\unitlength}{\svgwidth}%
  \fi%
  \global\let\svgwidth\undefined%
  \global\let\svgscale\undefined%
  \makeatother%
  \begin{picture}(1,0.84114282)%
    \lineheight{1}%
    \setlength\tabcolsep{0pt}%
    \put(0,0){\includegraphics[width=\unitlength,page=1]{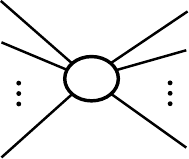}}%
    \put(0.38038775,0.37043609){\color[rgb]{0,0,0}\makebox(0,0)[lt]{\lineheight{1.25}\smash{\begin{tabular}[t]{l}$ab$\end{tabular}}}}%
    \put(0,0){\includegraphics[width=\unitlength,page=2]{signed-vtx-abx.pdf}}%
  \end{picture}%
\endgroup%
}}
    -  \raisebox{-21pt}{\marginbox{4pt 0}{
\begingroup%
  \makeatletter%
  \providecommand\color[2][]{%
    \errmessage{(Inkscape) Color is used for the text in Inkscape, but the package 'color.sty' is not loaded}%
    \renewcommand\color[2][]{}%
  }%
  \providecommand\transparent[1]{%
    \errmessage{(Inkscape) Transparency is used (non-zero) for the text in Inkscape, but the package 'transparent.sty' is not loaded}%
    \renewcommand\transparent[1]{}%
  }%
  \providecommand\rotatebox[2]{#2}%
  \newcommand*\fsize{\dimexpr\f@size pt\relax}%
  \newcommand*\lineheight[1]{\fontsize{\fsize}{#1\fsize}\selectfont}%
  \ifx\svgwidth\undefined%
    \setlength{\unitlength}{70.72207664bp}%
    \ifx\svgscale\undefined%
      \relax%
    \else%
      \setlength{\unitlength}{\unitlength * \real{\svgscale}}%
    \fi%
  \else%
    \setlength{\unitlength}{\svgwidth}%
  \fi%
  \global\let\svgwidth\undefined%
  \global\let\svgscale\undefined%
  \makeatother%
  \begin{picture}(1,0.64489848)%
    \lineheight{1}%
    \setlength\tabcolsep{0pt}%
    \put(0,0){\includegraphics[width=\unitlength,page=1]{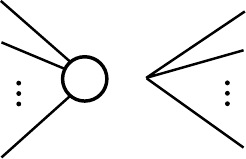}}%
    \put(0.30651671,0.27751504){\color[rgb]{0,0,0}\makebox(0,0)[lt]{\lineheight{1.25}\smash{\begin{tabular}[t]{l}$a$\end{tabular}}}}%
    \put(0,0){\includegraphics[width=\unitlength,page=2]{signed-vtx-asb.pdf}}%
    \put(0.57845634,0.27467454){\color[rgb]{0,0,0}\makebox(0,0)[lt]{\lineheight{1.25}\smash{\begin{tabular}[t]{l}$b$\end{tabular}}}}%
  \end{picture}%
\endgroup%
}}.
  \end{equation*}
\end{theorem}
\begin{proof}
  Expand both sides and use contraction-deletion in $\Sc^{N^2}$:
  \begin{align*}
    \raisebox{-21pt}{\marginbox{3pt 0}{}}
    &= \raisebox{-21pt}{\marginbox{3pt 0}{
\begingroup%
  \makeatletter%
  \providecommand\color[2][]{%
    \errmessage{(Inkscape) Color is used for the text in Inkscape, but the package 'color.sty' is not loaded}%
    \renewcommand\color[2][]{}%
  }%
  \providecommand\transparent[1]{%
    \errmessage{(Inkscape) Transparency is used (non-zero) for the text in Inkscape, but the package 'transparent.sty' is not loaded}%
    \renewcommand\transparent[1]{}%
  }%
  \providecommand\rotatebox[2]{#2}%
  \newcommand*\fsize{\dimexpr\f@size pt\relax}%
  \newcommand*\lineheight[1]{\fontsize{\fsize}{#1\fsize}\selectfont}%
  \ifx\svgwidth\undefined%
    \setlength{\unitlength}{76.72208342bp}%
    \ifx\svgscale\undefined%
      \relax%
    \else%
      \setlength{\unitlength}{\unitlength * \real{\svgscale}}%
    \fi%
  \else%
    \setlength{\unitlength}{\svgwidth}%
  \fi%
  \global\let\svgwidth\undefined%
  \global\let\svgscale\undefined%
  \makeatother%
  \begin{picture}(1,0.59446463)%
    \lineheight{1}%
    \setlength\tabcolsep{0pt}%
    \put(0,0){\includegraphics[width=\unitlength,page=1]{signed-vtx-d-d.pdf}}%
  \end{picture}%
\endgroup%
}}
      + a\raisebox{-21pt}{\marginbox{3pt 0}{
\begingroup%
  \makeatletter%
  \providecommand\color[2][]{%
    \errmessage{(Inkscape) Color is used for the text in Inkscape, but the package 'color.sty' is not loaded}%
    \renewcommand\color[2][]{}%
  }%
  \providecommand\transparent[1]{%
    \errmessage{(Inkscape) Transparency is used (non-zero) for the text in Inkscape, but the package 'transparent.sty' is not loaded}%
    \renewcommand\transparent[1]{}%
  }%
  \providecommand\rotatebox[2]{#2}%
  \newcommand*\fsize{\dimexpr\f@size pt\relax}%
  \newcommand*\lineheight[1]{\fontsize{\fsize}{#1\fsize}\selectfont}%
  \ifx\svgwidth\undefined%
    \setlength{\unitlength}{76.72208087bp}%
    \ifx\svgscale\undefined%
      \relax%
    \else%
      \setlength{\unitlength}{\unitlength * \real{\svgscale}}%
    \fi%
  \else%
    \setlength{\unitlength}{\svgwidth}%
  \fi%
  \global\let\svgwidth\undefined%
  \global\let\svgscale\undefined%
  \makeatother%
  \begin{picture}(1,0.59446465)%
    \lineheight{1}%
    \setlength\tabcolsep{0pt}%
    \put(0,0){\includegraphics[width=\unitlength,page=1]{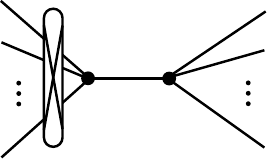}}%
  \end{picture}%
\endgroup%
}}\\
      &\qquad+ b\raisebox{-21pt}{\marginbox{3pt 0}{
\begingroup%
  \makeatletter%
  \providecommand\color[2][]{%
    \errmessage{(Inkscape) Color is used for the text in Inkscape, but the package 'color.sty' is not loaded}%
    \renewcommand\color[2][]{}%
  }%
  \providecommand\transparent[1]{%
    \errmessage{(Inkscape) Transparency is used (non-zero) for the text in Inkscape, but the package 'transparent.sty' is not loaded}%
    \renewcommand\transparent[1]{}%
  }%
  \providecommand\rotatebox[2]{#2}%
  \newcommand*\fsize{\dimexpr\f@size pt\relax}%
  \newcommand*\lineheight[1]{\fontsize{\fsize}{#1\fsize}\selectfont}%
  \ifx\svgwidth\undefined%
    \setlength{\unitlength}{76.7220806bp}%
    \ifx\svgscale\undefined%
      \relax%
    \else%
      \setlength{\unitlength}{\unitlength * \real{\svgscale}}%
    \fi%
  \else%
    \setlength{\unitlength}{\svgwidth}%
  \fi%
  \global\let\svgwidth\undefined%
  \global\let\svgscale\undefined%
  \makeatother%
  \begin{picture}(1,0.59446466)%
    \lineheight{1}%
    \setlength\tabcolsep{0pt}%
    \put(0,0){\includegraphics[width=\unitlength,page=1]{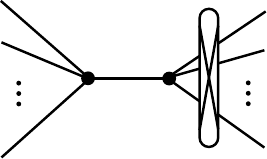}}%
  \end{picture}%
\endgroup%
}}
        + ab\raisebox{-21pt}{\marginbox{3pt 0}{
\begingroup%
  \makeatletter%
  \providecommand\color[2][]{%
    \errmessage{(Inkscape) Color is used for the text in Inkscape, but the package 'color.sty' is not loaded}%
    \renewcommand\color[2][]{}%
  }%
  \providecommand\transparent[1]{%
    \errmessage{(Inkscape) Transparency is used (non-zero) for the text in Inkscape, but the package 'transparent.sty' is not loaded}%
    \renewcommand\transparent[1]{}%
  }%
  \providecommand\rotatebox[2]{#2}%
  \newcommand*\fsize{\dimexpr\f@size pt\relax}%
  \newcommand*\lineheight[1]{\fontsize{\fsize}{#1\fsize}\selectfont}%
  \ifx\svgwidth\undefined%
    \setlength{\unitlength}{76.72208059bp}%
    \ifx\svgscale\undefined%
      \relax%
    \else%
      \setlength{\unitlength}{\unitlength * \real{\svgscale}}%
    \fi%
  \else%
    \setlength{\unitlength}{\svgwidth}%
  \fi%
  \global\let\svgwidth\undefined%
  \global\let\svgscale\undefined%
  \makeatother%
  \begin{picture}(1,0.59446466)%
    \lineheight{1}%
    \setlength\tabcolsep{0pt}%
    \put(0,0){\includegraphics[width=\unitlength,page=1]{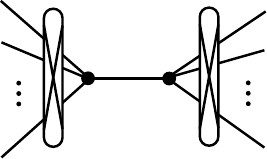}}%
  \end{picture}%
\endgroup%
}}\\
    &= \raisebox{-21pt}{\marginbox{3pt 0}{
\begingroup%
  \makeatletter%
  \providecommand\color[2][]{%
    \errmessage{(Inkscape) Color is used for the text in Inkscape, but the package 'color.sty' is not loaded}%
    \renewcommand\color[2][]{}%
  }%
  \providecommand\transparent[1]{%
    \errmessage{(Inkscape) Transparency is used (non-zero) for the text in Inkscape, but the package 'transparent.sty' is not loaded}%
    \renewcommand\transparent[1]{}%
  }%
  \providecommand\rotatebox[2]{#2}%
  \newcommand*\fsize{\dimexpr\f@size pt\relax}%
  \newcommand*\lineheight[1]{\fontsize{\fsize}{#1\fsize}\selectfont}%
  \ifx\svgwidth\undefined%
    \setlength{\unitlength}{54.22198792bp}%
    \ifx\svgscale\undefined%
      \relax%
    \else%
      \setlength{\unitlength}{\unitlength * \real{\svgscale}}%
    \fi%
  \else%
    \setlength{\unitlength}{\svgwidth}%
  \fi%
  \global\let\svgwidth\undefined%
  \global\let\svgscale\undefined%
  \makeatother%
  \begin{picture}(1,0.84114511)%
    \lineheight{1}%
    \setlength\tabcolsep{0pt}%
    \put(0,0){\includegraphics[width=\unitlength,page=1]{signed-vtx-d.pdf}}%
  \end{picture}%
\endgroup%
}}
      -\raisebox{-21pt}{\marginbox{3pt 0}{
\begingroup%
  \makeatletter%
  \providecommand\color[2][]{%
    \errmessage{(Inkscape) Color is used for the text in Inkscape, but the package 'color.sty' is not loaded}%
    \renewcommand\color[2][]{}%
  }%
  \providecommand\transparent[1]{%
    \errmessage{(Inkscape) Transparency is used (non-zero) for the text in Inkscape, but the package 'transparent.sty' is not loaded}%
    \renewcommand\transparent[1]{}%
  }%
  \providecommand\rotatebox[2]{#2}%
  \newcommand*\fsize{\dimexpr\f@size pt\relax}%
  \newcommand*\lineheight[1]{\fontsize{\fsize}{#1\fsize}\selectfont}%
  \ifx\svgwidth\undefined%
    \setlength{\unitlength}{66.22207681bp}%
    \ifx\svgscale\undefined%
      \relax%
    \else%
      \setlength{\unitlength}{\unitlength * \real{\svgscale}}%
    \fi%
  \else%
    \setlength{\unitlength}{\svgwidth}%
  \fi%
  \global\let\svgwidth\undefined%
  \global\let\svgscale\undefined%
  \makeatother%
  \begin{picture}(1,0.68872146)%
    \lineheight{1}%
    \setlength\tabcolsep{0pt}%
    \put(0,0){\includegraphics[width=\unitlength,page=1]{signed-vtx-dsd.pdf}}%
  \end{picture}%
\endgroup%
}}
      +a\raisebox{-21pt}{\marginbox{3pt 0}{
\begingroup%
  \makeatletter%
  \providecommand\color[2][]{%
    \errmessage{(Inkscape) Color is used for the text in Inkscape, but the package 'color.sty' is not loaded}%
    \renewcommand\color[2][]{}%
  }%
  \providecommand\transparent[1]{%
    \errmessage{(Inkscape) Transparency is used (non-zero) for the text in Inkscape, but the package 'transparent.sty' is not loaded}%
    \renewcommand\transparent[1]{}%
  }%
  \providecommand\rotatebox[2]{#2}%
  \newcommand*\fsize{\dimexpr\f@size pt\relax}%
  \newcommand*\lineheight[1]{\fontsize{\fsize}{#1\fsize}\selectfont}%
  \ifx\svgwidth\undefined%
    \setlength{\unitlength}{54.22197267bp}%
    \ifx\svgscale\undefined%
      \relax%
    \else%
      \setlength{\unitlength}{\unitlength * \real{\svgscale}}%
    \fi%
  \else%
    \setlength{\unitlength}{\svgwidth}%
  \fi%
  \global\let\svgwidth\undefined%
  \global\let\svgscale\undefined%
  \makeatother%
  \begin{picture}(1,0.84114544)%
    \lineheight{1}%
    \setlength\tabcolsep{0pt}%
    \put(0,0){\includegraphics[width=\unitlength,page=1]{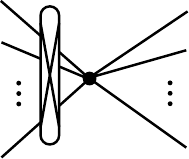}}%
  \end{picture}%
\endgroup%
}}
      -a\raisebox{-21pt}{\marginbox{3pt 0}{
\begingroup%
  \makeatletter%
  \providecommand\color[2][]{%
    \errmessage{(Inkscape) Color is used for the text in Inkscape, but the package 'color.sty' is not loaded}%
    \renewcommand\color[2][]{}%
  }%
  \providecommand\transparent[1]{%
    \errmessage{(Inkscape) Transparency is used (non-zero) for the text in Inkscape, but the package 'transparent.sty' is not loaded}%
    \renewcommand\transparent[1]{}%
  }%
  \providecommand\rotatebox[2]{#2}%
  \newcommand*\fsize{\dimexpr\f@size pt\relax}%
  \newcommand*\lineheight[1]{\fontsize{\fsize}{#1\fsize}\selectfont}%
  \ifx\svgwidth\undefined%
    \setlength{\unitlength}{66.22207397bp}%
    \ifx\svgscale\undefined%
      \relax%
    \else%
      \setlength{\unitlength}{\unitlength * \real{\svgscale}}%
    \fi%
  \else%
    \setlength{\unitlength}{\svgwidth}%
  \fi%
  \global\let\svgwidth\undefined%
  \global\let\svgscale\undefined%
  \makeatother%
  \begin{picture}(1,0.68872153)%
    \lineheight{1}%
    \setlength\tabcolsep{0pt}%
    \put(0,0){\includegraphics[width=\unitlength,page=1]{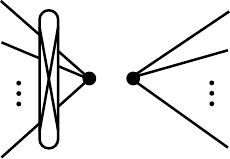}}%
  \end{picture}%
\endgroup%
}}\\
    &\qquad+b\raisebox{-21pt}{\marginbox{3pt 0}{
\begingroup%
  \makeatletter%
  \providecommand\color[2][]{%
    \errmessage{(Inkscape) Color is used for the text in Inkscape, but the package 'color.sty' is not loaded}%
    \renewcommand\color[2][]{}%
  }%
  \providecommand\transparent[1]{%
    \errmessage{(Inkscape) Transparency is used (non-zero) for the text in Inkscape, but the package 'transparent.sty' is not loaded}%
    \renewcommand\transparent[1]{}%
  }%
  \providecommand\rotatebox[2]{#2}%
  \newcommand*\fsize{\dimexpr\f@size pt\relax}%
  \newcommand*\lineheight[1]{\fontsize{\fsize}{#1\fsize}\selectfont}%
  \ifx\svgwidth\undefined%
    \setlength{\unitlength}{54.22197267bp}%
    \ifx\svgscale\undefined%
      \relax%
    \else%
      \setlength{\unitlength}{\unitlength * \real{\svgscale}}%
    \fi%
  \else%
    \setlength{\unitlength}{\svgwidth}%
  \fi%
  \global\let\svgwidth\undefined%
  \global\let\svgscale\undefined%
  \makeatother%
  \begin{picture}(1,0.84114544)%
    \lineheight{1}%
    \setlength\tabcolsep{0pt}%
    \put(0,0){\includegraphics[width=\unitlength,page=1]{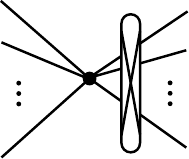}}%
  \end{picture}%
\endgroup%
}}
      -b\raisebox{-21pt}{\marginbox{3pt 0}{
\begingroup%
  \makeatletter%
  \providecommand\color[2][]{%
    \errmessage{(Inkscape) Color is used for the text in Inkscape, but the package 'color.sty' is not loaded}%
    \renewcommand\color[2][]{}%
  }%
  \providecommand\transparent[1]{%
    \errmessage{(Inkscape) Transparency is used (non-zero) for the text in Inkscape, but the package 'transparent.sty' is not loaded}%
    \renewcommand\transparent[1]{}%
  }%
  \providecommand\rotatebox[2]{#2}%
  \newcommand*\fsize{\dimexpr\f@size pt\relax}%
  \newcommand*\lineheight[1]{\fontsize{\fsize}{#1\fsize}\selectfont}%
  \ifx\svgwidth\undefined%
    \setlength{\unitlength}{66.22207681bp}%
    \ifx\svgscale\undefined%
      \relax%
    \else%
      \setlength{\unitlength}{\unitlength * \real{\svgscale}}%
    \fi%
  \else%
    \setlength{\unitlength}{\svgwidth}%
  \fi%
  \global\let\svgwidth\undefined%
  \global\let\svgscale\undefined%
  \makeatother%
  \begin{picture}(1,0.6887215)%
    \lineheight{1}%
    \setlength\tabcolsep{0pt}%
    \put(0,0){\includegraphics[width=\unitlength,page=1]{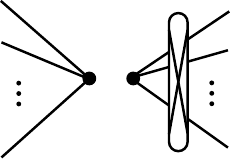}}%
  \end{picture}%
\endgroup%
}}
      +ab\raisebox{-21pt}{\marginbox{3pt 0}{
\begingroup%
  \makeatletter%
  \providecommand\color[2][]{%
    \errmessage{(Inkscape) Color is used for the text in Inkscape, but the package 'color.sty' is not loaded}%
    \renewcommand\color[2][]{}%
  }%
  \providecommand\transparent[1]{%
    \errmessage{(Inkscape) Transparency is used (non-zero) for the text in Inkscape, but the package 'transparent.sty' is not loaded}%
    \renewcommand\transparent[1]{}%
  }%
  \providecommand\rotatebox[2]{#2}%
  \newcommand*\fsize{\dimexpr\f@size pt\relax}%
  \newcommand*\lineheight[1]{\fontsize{\fsize}{#1\fsize}\selectfont}%
  \ifx\svgwidth\undefined%
    \setlength{\unitlength}{54.22197549bp}%
    \ifx\svgscale\undefined%
      \relax%
    \else%
      \setlength{\unitlength}{\unitlength * \real{\svgscale}}%
    \fi%
  \else%
    \setlength{\unitlength}{\svgwidth}%
  \fi%
  \global\let\svgwidth\undefined%
  \global\let\svgscale\undefined%
  \makeatother%
  \begin{picture}(1,0.8411454)%
    \lineheight{1}%
    \setlength\tabcolsep{0pt}%
    \put(0,0){\includegraphics[width=\unitlength,page=1]{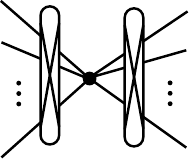}}%
  \end{picture}%
\endgroup%
}}
      -ab\raisebox{-21pt}{\marginbox{3pt 0}{
\begingroup%
  \makeatletter%
  \providecommand\color[2][]{%
    \errmessage{(Inkscape) Color is used for the text in Inkscape, but the package 'color.sty' is not loaded}%
    \renewcommand\color[2][]{}%
  }%
  \providecommand\transparent[1]{%
    \errmessage{(Inkscape) Transparency is used (non-zero) for the text in Inkscape, but the package 'transparent.sty' is not loaded}%
    \renewcommand\transparent[1]{}%
  }%
  \providecommand\rotatebox[2]{#2}%
  \newcommand*\fsize{\dimexpr\f@size pt\relax}%
  \newcommand*\lineheight[1]{\fontsize{\fsize}{#1\fsize}\selectfont}%
  \ifx\svgwidth\undefined%
    \setlength{\unitlength}{66.22207425bp}%
    \ifx\svgscale\undefined%
      \relax%
    \else%
      \setlength{\unitlength}{\unitlength * \real{\svgscale}}%
    \fi%
  \else%
    \setlength{\unitlength}{\svgwidth}%
  \fi%
  \global\let\svgwidth\undefined%
  \global\let\svgscale\undefined%
  \makeatother%
  \begin{picture}(1,0.68872148)%
    \lineheight{1}%
    \setlength\tabcolsep{0pt}%
    \put(0,0){\includegraphics[width=\unitlength,page=1]{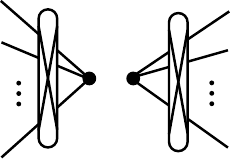}}%
  \end{picture}%
\endgroup%
}}\\
    &=\raisebox{-21pt}{\marginbox{3pt 0}{}}
      +b\raisebox{-21pt}{\marginbox{3pt 0}{}}
      -\raisebox{-21pt}{\marginbox{3pt 0}{}}.
  \end{align*}
\end{proof}

\begin{theorem}
  \label{thm:sl-loop-contr-del}
  Let $a\in\{\pm 1\}$.
  The extended $W_{\mathfrak{sl}(N)}$ polynomial satisfies
  \begin{equation*}
    \raisebox{-21pt}{\marginbox{4pt 0}{
\begingroup%
  \makeatletter%
  \providecommand\color[2][]{%
    \errmessage{(Inkscape) Color is used for the text in Inkscape, but the package 'color.sty' is not loaded}%
    \renewcommand\color[2][]{}%
  }%
  \providecommand\transparent[1]{%
    \errmessage{(Inkscape) Transparency is used (non-zero) for the text in Inkscape, but the package 'transparent.sty' is not loaded}%
    \renewcommand\transparent[1]{}%
  }%
  \providecommand\rotatebox[2]{#2}%
  \newcommand*\fsize{\dimexpr\f@size pt\relax}%
  \newcommand*\lineheight[1]{\fontsize{\fsize}{#1\fsize}\selectfont}%
  \ifx\svgwidth\undefined%
    \setlength{\unitlength}{54.22198877bp}%
    \ifx\svgscale\undefined%
      \relax%
    \else%
      \setlength{\unitlength}{\unitlength * \real{\svgscale}}%
    \fi%
  \else%
    \setlength{\unitlength}{\svgwidth}%
  \fi%
  \global\let\svgwidth\undefined%
  \global\let\svgscale\undefined%
  \makeatother%
  \begin{picture}(1,0.84114509)%
    \lineheight{1}%
    \setlength\tabcolsep{0pt}%
    \put(0,0){\includegraphics[width=\unitlength,page=1]{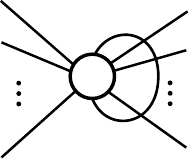}}%
    \put(0.44021824,0.37557571){\color[rgb]{0,0,0}\makebox(0,0)[lt]{\lineheight{1.25}\smash{\begin{tabular}[t]{l}$a$\end{tabular}}}}%
  \end{picture}%
\endgroup%
}}
    =
    \frac{1}{2}N^2\left(
    \raisebox{-21pt}{\marginbox{4pt 0}{
\begingroup%
  \makeatletter%
  \providecommand\color[2][]{%
    \errmessage{(Inkscape) Color is used for the text in Inkscape, but the package 'color.sty' is not loaded}%
    \renewcommand\color[2][]{}%
  }%
  \providecommand\transparent[1]{%
    \errmessage{(Inkscape) Transparency is used (non-zero) for the text in Inkscape, but the package 'transparent.sty' is not loaded}%
    \renewcommand\transparent[1]{}%
  }%
  \providecommand\rotatebox[2]{#2}%
  \newcommand*\fsize{\dimexpr\f@size pt\relax}%
  \newcommand*\lineheight[1]{\fontsize{\fsize}{#1\fsize}\selectfont}%
  \ifx\svgwidth\undefined%
    \setlength{\unitlength}{70.72207777bp}%
    \ifx\svgscale\undefined%
      \relax%
    \else%
      \setlength{\unitlength}{\unitlength * \real{\svgscale}}%
    \fi%
  \else%
    \setlength{\unitlength}{\svgwidth}%
  \fi%
  \global\let\svgwidth\undefined%
  \global\let\svgscale\undefined%
  \makeatother%
  \begin{picture}(1,0.64489843)%
    \lineheight{1}%
    \setlength\tabcolsep{0pt}%
    \put(0,0){\includegraphics[width=\unitlength,page=1]{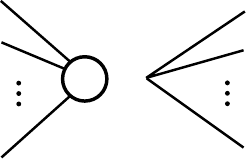}}%
    \put(0.3065166,0.27751507){\color[rgb]{0,0,0}\makebox(0,0)[lt]{\lineheight{1.25}\smash{\begin{tabular}[t]{l}$a$\end{tabular}}}}%
    \put(0,0){\includegraphics[width=\unitlength,page=2]{signed-vtx-asp.pdf}}%
    \put(0.55822008,0.27786789){\color[rgb]{0,0,0}\makebox(0,0)[lt]{\lineheight{1.25}\smash{\begin{tabular}[t]{l}$+$\end{tabular}}}}%
  \end{picture}%
\endgroup%
}}
    +
    \raisebox{-21pt}{\marginbox{4pt 0}{
\begingroup%
  \makeatletter%
  \providecommand\color[2][]{%
    \errmessage{(Inkscape) Color is used for the text in Inkscape, but the package 'color.sty' is not loaded}%
    \renewcommand\color[2][]{}%
  }%
  \providecommand\transparent[1]{%
    \errmessage{(Inkscape) Transparency is used (non-zero) for the text in Inkscape, but the package 'transparent.sty' is not loaded}%
    \renewcommand\transparent[1]{}%
  }%
  \providecommand\rotatebox[2]{#2}%
  \newcommand*\fsize{\dimexpr\f@size pt\relax}%
  \newcommand*\lineheight[1]{\fontsize{\fsize}{#1\fsize}\selectfont}%
  \ifx\svgwidth\undefined%
    \setlength{\unitlength}{70.72207666bp}%
    \ifx\svgscale\undefined%
      \relax%
    \else%
      \setlength{\unitlength}{\unitlength * \real{\svgscale}}%
    \fi%
  \else%
    \setlength{\unitlength}{\svgwidth}%
  \fi%
  \global\let\svgwidth\undefined%
  \global\let\svgscale\undefined%
  \makeatother%
  \begin{picture}(1,0.6448992)%
    \lineheight{1}%
    \setlength\tabcolsep{0pt}%
    \put(0,0){\includegraphics[width=\unitlength,page=1]{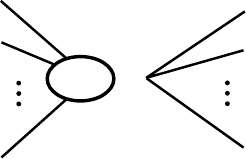}}%
    \put(0.22698006,0.28291362){\color[rgb]{0,0,0}\makebox(0,0)[lt]{\lineheight{1.25}\smash{\begin{tabular}[t]{l}$-a$\end{tabular}}}}%
    \put(0,0){\includegraphics[width=\unitlength,page=2]{signed-vtx-nasm.pdf}}%
    \put(0.5608451,0.28054622){\color[rgb]{0,0,0}\makebox(0,0)[lt]{\lineheight{1.25}\smash{\begin{tabular}[t]{l}$-$\end{tabular}}}}%
  \end{picture}%
\endgroup%
}}\right)
    -  \raisebox{-21pt}{\marginbox{4pt 0}{
\begingroup%
  \makeatletter%
  \providecommand\color[2][]{%
    \errmessage{(Inkscape) Color is used for the text in Inkscape, but the package 'color.sty' is not loaded}%
    \renewcommand\color[2][]{}%
  }%
  \providecommand\transparent[1]{%
    \errmessage{(Inkscape) Transparency is used (non-zero) for the text in Inkscape, but the package 'transparent.sty' is not loaded}%
    \renewcommand\transparent[1]{}%
  }%
  \providecommand\rotatebox[2]{#2}%
  \newcommand*\fsize{\dimexpr\f@size pt\relax}%
  \newcommand*\lineheight[1]{\fontsize{\fsize}{#1\fsize}\selectfont}%
  \ifx\svgwidth\undefined%
    \setlength{\unitlength}{54.22198963bp}%
    \ifx\svgscale\undefined%
      \relax%
    \else%
      \setlength{\unitlength}{\unitlength * \real{\svgscale}}%
    \fi%
  \else%
    \setlength{\unitlength}{\svgwidth}%
  \fi%
  \global\let\svgwidth\undefined%
  \global\let\svgscale\undefined%
  \makeatother%
  \begin{picture}(1,0.84114508)%
    \lineheight{1}%
    \setlength\tabcolsep{0pt}%
    \put(0,0){\includegraphics[width=\unitlength,page=1]{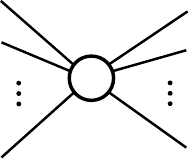}}%
    \put(0.4352783,0.36569564){\color[rgb]{0,0,0}\makebox(0,0)[lt]{\lineheight{1.25}\smash{\begin{tabular}[t]{l}$a$\end{tabular}}}}%
  \end{picture}%
\endgroup%
}}.
  \end{equation*}
\end{theorem}
\begin{proof}
  Similarly expand both sides and use contraction-deletion for the loop edge.
\end{proof}

\begin{corollary}
  Let $a,b\in\{\pm 1\}$.  The extended $W_{\mathfrak{sl}(N)}$ polynomial satisfies
  \begin{enumerate}
  \item $\displaystyle     \raisebox{-21pt}{\marginbox{1pt 0}{}}
    +\raisebox{-21pt}{\marginbox{1pt 0}{
\begingroup%
  \makeatletter%
  \providecommand\color[2][]{%
    \errmessage{(Inkscape) Color is used for the text in Inkscape, but the package 'color.sty' is not loaded}%
    \renewcommand\color[2][]{}%
  }%
  \providecommand\transparent[1]{%
    \errmessage{(Inkscape) Transparency is used (non-zero) for the text in Inkscape, but the package 'transparent.sty' is not loaded}%
    \renewcommand\transparent[1]{}%
  }%
  \providecommand\rotatebox[2]{#2}%
  \newcommand*\fsize{\dimexpr\f@size pt\relax}%
  \newcommand*\lineheight[1]{\fontsize{\fsize}{#1\fsize}\selectfont}%
  \ifx\svgwidth\undefined%
    \setlength{\unitlength}{76.72207917bp}%
    \ifx\svgscale\undefined%
      \relax%
    \else%
      \setlength{\unitlength}{\unitlength * \real{\svgscale}}%
    \fi%
  \else%
    \setlength{\unitlength}{\svgwidth}%
  \fi%
  \global\let\svgwidth\undefined%
  \global\let\svgscale\undefined%
  \makeatother%
  \begin{picture}(1,0.59446464)%
    \lineheight{1}%
    \setlength\tabcolsep{0pt}%
    \put(0,0){\includegraphics[width=\unitlength,page=1]{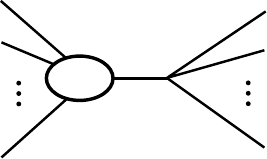}}%
    \put(0.20611213,0.26265337){\color[rgb]{0,0,0}\makebox(0,0)[lt]{\lineheight{0}\smash{\begin{tabular}[t]{l}$-a$\end{tabular}}}}%
    \put(0,0){\includegraphics[width=\unitlength,page=2]{signed-vtx-na-nb.pdf}}%
    \put(0.55090853,0.26443774){\color[rgb]{0,0,0}\makebox(0,0)[lt]{\lineheight{0}\smash{\begin{tabular}[t]{l}$-b$\end{tabular}}}}%
  \end{picture}%
\endgroup%
}}
    =
    2\raisebox{-21pt}{\marginbox{1pt 0}{}}
    -\raisebox{-21pt}{\marginbox{1pt 0}{}}
    -\raisebox{-21pt}{\marginbox{1pt 0}{
\begingroup%
  \makeatletter%
  \providecommand\color[2][]{%
    \errmessage{(Inkscape) Color is used for the text in Inkscape, but the package 'color.sty' is not loaded}%
    \renewcommand\color[2][]{}%
  }%
  \providecommand\transparent[1]{%
    \errmessage{(Inkscape) Transparency is used (non-zero) for the text in Inkscape, but the package 'transparent.sty' is not loaded}%
    \renewcommand\transparent[1]{}%
  }%
  \providecommand\rotatebox[2]{#2}%
  \newcommand*\fsize{\dimexpr\f@size pt\relax}%
  \newcommand*\lineheight[1]{\fontsize{\fsize}{#1\fsize}\selectfont}%
  \ifx\svgwidth\undefined%
    \setlength{\unitlength}{70.72207994bp}%
    \ifx\svgscale\undefined%
      \relax%
    \else%
      \setlength{\unitlength}{\unitlength * \real{\svgscale}}%
    \fi%
  \else%
    \setlength{\unitlength}{\svgwidth}%
  \fi%
  \global\let\svgwidth\undefined%
  \global\let\svgscale\undefined%
  \makeatother%
  \begin{picture}(1,0.6448994)%
    \lineheight{1}%
    \setlength\tabcolsep{0pt}%
    \put(0,0){\includegraphics[width=\unitlength,page=1]{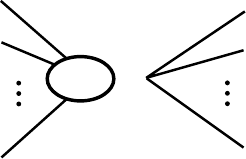}}%
    \put(0.22698011,0.28291349){\color[rgb]{0,0,0}\makebox(0,0)[lt]{\lineheight{0}\smash{\begin{tabular}[t]{l}$-a$\end{tabular}}}}%
    \put(0,0){\includegraphics[width=\unitlength,page=2]{signed-vtx-nasnb.pdf}}%
    \put(0.55278013,0.28348051){\color[rgb]{0,0,0}\makebox(0,0)[lt]{\lineheight{0}\smash{\begin{tabular}[t]{l}$-b$\end{tabular}}}}%
  \end{picture}%
\endgroup%
}}$
  \item $\displaystyle\raisebox{-21pt}{\marginbox{3pt 0}{
\begingroup%
  \makeatletter%
  \providecommand\color[2][]{%
    \errmessage{(Inkscape) Color is used for the text in Inkscape, but the package 'color.sty' is not loaded}%
    \renewcommand\color[2][]{}%
  }%
  \providecommand\transparent[1]{%
    \errmessage{(Inkscape) Transparency is used (non-zero) for the text in Inkscape, but the package 'transparent.sty' is not loaded}%
    \renewcommand\transparent[1]{}%
  }%
  \providecommand\rotatebox[2]{#2}%
  \newcommand*\fsize{\dimexpr\f@size pt\relax}%
  \newcommand*\lineheight[1]{\fontsize{\fsize}{#1\fsize}\selectfont}%
  \ifx\svgwidth\undefined%
    \setlength{\unitlength}{54.22198295bp}%
    \ifx\svgscale\undefined%
      \relax%
    \else%
      \setlength{\unitlength}{\unitlength * \real{\svgscale}}%
    \fi%
  \else%
    \setlength{\unitlength}{\svgwidth}%
  \fi%
  \global\let\svgwidth\undefined%
  \global\let\svgscale\undefined%
  \makeatother%
  \begin{picture}(1,0.84114524)%
    \lineheight{1}%
    \setlength\tabcolsep{0pt}%
    \put(0,0){\includegraphics[width=\unitlength,page=1]{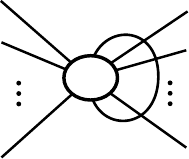}}%
    \put(0.37595277,0.37498266){\color[rgb]{0,0,0}\makebox(0,0)[lt]{\lineheight{0}\smash{\begin{tabular}[t]{l}$ab$\end{tabular}}}}%
  \end{picture}%
\endgroup%
}}
    =
    \frac{1}{2}N^2\left(
      \raisebox{-21pt}{\marginbox{3pt 0}{}}
      +
      \raisebox{-21pt}{\marginbox{3pt 0}{}}\right)
    -  \raisebox{-21pt}{\marginbox{3pt 0}{}}$.
  \end{enumerate}
\end{corollary}

If we were to define partial duality of signed graphs like so, where $e$ is the loop edge,
\begin{align*}
  \raisebox{-21pt}{\marginbox{3pt 0}{}}
  &\xleftrightarrow{\delta_e}
    \frac{1}{2}\left(\raisebox{-21pt}{\marginbox{3pt 0}{}}
    +\raisebox{-21pt}{\marginbox{3pt 0}{}}
    \right),
\end{align*}
then we could restate the corollary in the following form:
\begin{corollary}
  Let $G$ be a signed virtual graph and $e\in G$ a loop edge.
  Then, with partial duality extended as above,
  \begin{enumerate}
  \item $W_{\mathfrak{sl}(N)}(G)=N^2 W_{\mathfrak{sl}(N)}(\delta_e G-e)-W_{\mathfrak{sl}(N)}(G-e)$ and
  \item $W_{\mathfrak{sl}(N)}(\delta_e G)=W_{\mathfrak{sl}(N)}(G-e)- W_{\mathfrak{sl}(N)}(\delta_e G-e)$,
  \end{enumerate}
  where in $\delta_e G-e$ we delete $e$ from both graphs.
\end{corollary}

These contraction-deletion-like rules along with the following relations give a recursive procedure to calculate $W_{\mathfrak{sl}(N)}$ for any signed virtual graph:
\begin{itemize}
\item $W_{\mathfrak{sl}(N)}(\emptyset)=1$.
\item If $v\in V(G)$ is an isolated vertex, $W_{\mathfrak{sl}(N)}(G)=(1+s(v))W_{\mathfrak{sl}(N)}(G-v)$.
\item $W_{\mathfrak{sl}(N)}(G_1\amalg G_2)=W_{\mathfrak{sl}(N)}(G_1)W_{\mathfrak{sl}(N)}(G_2)$.
\end{itemize}
Additional relations are:
\begin{itemize}
\item If $v\in V(G)$ is a vertex of degree $1$, $W_{\mathfrak{sl}(N)}(G)=0$.
\item Subdividing an edge with a vertex of sign $a$ multiplies $W_{\mathfrak{sl}(N)}$ by $1+a$.
\item If $V\in V(G)$, $W_{\mathfrak{sl}(N)}(\sigma_v G)=s(v)W_{\mathfrak{sl}(N)}(G)$.
\end{itemize}

\begin{figure}[htb]
  \centering
  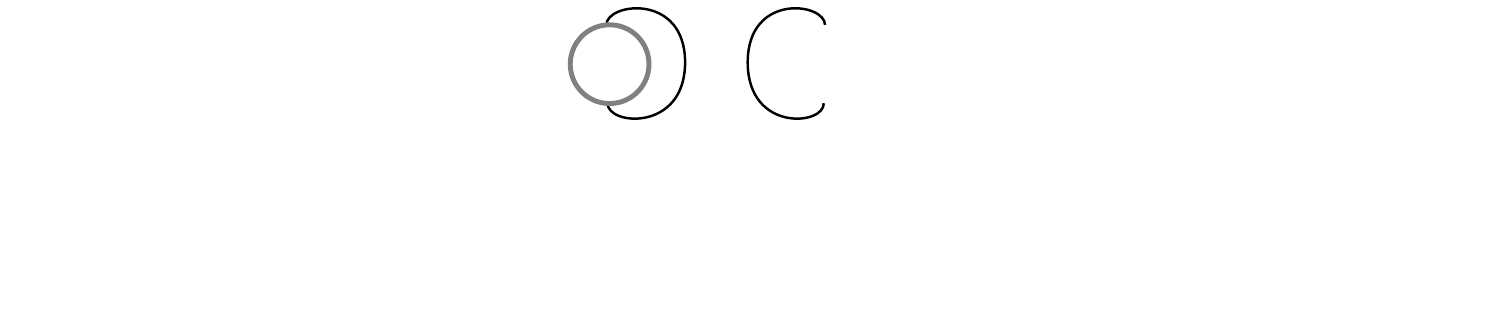
  \caption{$W_{\mathfrak{sl}(N)}$ relations for connect sums.}
  \label{fig:connect-sums-sln}
\end{figure}

\begin{proposition}
  If $G_1$ and $G_2$ are virtual graphs,
  \[ 2(N^2-1) W_{\mathfrak{sl}(N)}(G_1\csum_2 G_2) = W_{\mathfrak{sl}(N)}(G_1\amalg G_2), \]
  where the connect sum is performed at degree-$2$ vertices of positive sign.
  For $1\leq i\leq 2$, choose degree-$3$ vertices $v_i\in G_i$ and let $G_i^a$ for $a\in\{\pm 1\}$ be $G_i$ but with $s(v_i)=a$.
  Then,
  \[ p^+p^-W_{\mathfrak{sl}(N)}(G_1\csum_3 G_2) = p^- W_{\mathfrak{sl}(N)}(G_1^+\amalg G_2^+) + p^+ W_{\mathfrak{sl}(N)}(G_1^-\amalg G_2^-), \]
  where $p^s$, with $s\in\{\pm 1\}$, is the $W_{\mathfrak{sl}(N)}$ polynomial of the theta graph with both vertices given sign $s$: $p^+=2(N^2-1)(N^2-4)$ and $p^-=2N^2(N^2-1)$.
  See Figure~\ref{fig:connect-sums-sln} for an illustration.
\end{proposition}
\begin{proof}
  Applying Theorems \ref{thm:sl-edge-contr-del} and \ref{thm:sl-loop-contr-del} to a graph with two or three boundary edges allows one to reduce to the case of a linear combination of graphs with a single interior vertex, since transposing the edges of a trivalent vertex is the same as flipping it.
  In the case of two boundary edges, the interior vertex must be positively signed to be nonzero, and the edge-connect-sum relation follows from expanding both sides.
  In the case of three boundary edges, the interior vertex is either positive or negative.
  The pairing between a positive and negative trivalent vertex is $0$, and again the result follows by expanding both sides.
  Alternatively, one may expand in $\Sc^{N^2}$.
\end{proof}

Bar-Natan in \cite{Bar-Natan1997} cites Whitney that for connected planar bridgeless cubic graphs, one can get between any two planar embeddings by repeatedly flipping over subgraphs in edge-connect-sum decompositions.
For convenience, we give a self-contained proof of this result for graphs of arbitrary degree.
\begin{lemma}
  \label{lem:whitney}
  Let $G$ be a connected planar bridgeless virtual graph, and let $\mathcal{P}=\{W\subset V(G):\sigma_W(G)\text{ is planar}\}$.
  Consider an equivalence relation generated as follows: call $W,W'\in\mathcal{P}$ equivalent if $W$ is the symmetric difference of $W'$ and $U$, where $U\subset V(G)$ is such that $U$ and $V(G)-U$ are the vertex sets of an edge-connect-sum decomposition of $G$ (allowing $U$ to be empty).
  Then $\mathcal{P}$ has exactly one equivalence class.
  In other words, all planar embeddings of $G$ modulo vertex flips are related by repeatedly flipping over subgraphs in edge-connect-sum decompositions and the whole graph.
\end{lemma}
\begin{proof}
  Let $G$ and $\mathcal{P}$ be as in the hypothesis, and let $W\in\mathcal{P}$ be arbitrary.
  Consider an embedding of $\sigma_W(G)$ in $S^2$.
  Take a closed regular neighborhood of the subgraph induced by $W$ in $S^2$.
  We will induct on the number of boundary components of the neighborhood to show $W$ is equivalent to $\emptyset$.
  If there are no boundary components, then either $W=\emptyset$ or $W=V(G)$, which are equivalent.

  Let $D$ be a connected component of the neighborhood that is homeomorphic to a disk, where if there are none there is one in the complement, which we may take instead since $W$ is equivalent to $V(G)-W$.

  The region just outside $D$ is the graph in the prime orientation, where inside $D$ is the graph in the reversed orientation.
  Hence, flipping the portion of the graph within $D$ over will result in a planar virtual graph.
  If there were more than two edges of $\sigma_W(G)$ incident to the boundary of $D$, then flipping the disk portion over would result in a nonplanar virtual graph.
  Since there are no bridge edges and $G$ is connected, there are exactly two edges incident to the boundary, hence $D$ demonstrates $G$ as an edge connect sum.
  If $U$ is the set of vertices of $G$ within $D$, then $W$ is equivalent to $W-U$, whose corresponding neighborhood has one fewer boundary edge.
\end{proof}

Let $g_{\mathrm{min}}(G)$ be the minimal genus over all rotation systems of a cubic virtual graph $G$:
\[g_{\mathrm{min}}(G)=\min_{W\subset V(G)}g(\sigma_W(G)).\]
It follows from Theorem~\ref{thm:s-poly-degree} that $\deg W_{\mathfrak{sl}(N)}(G)\leq 2b_1(G)-2g_{\mathrm{min}}(G)$.

The following theorem extends Bar-Natan's result\cite{Bar-Natan1997} that for cubic $G$ maximal degree is achieved if and only if the underlying graph is bridgeless and planar.
\begin{theorem}
  \label{thm:sl-planar}
  If $G$ is a bridgeless signed virtual graph with $s(v)=(-1)^{\deg(v)}$ for all $v\in V(G)$, then $\deg W_{\mathfrak{sl}(N)}(G)=2b_1(G)$ if and only if there is some $W\subset V(G)$ with $\sigma_W(G)$ planar.
\end{theorem}
\begin{proof}
  By multiplicativity under disjoint union, we may assume $G$ is connected.
  In the Definition~\ref{lem:sl-as-s-extended} expansion, the only $S_{\sigma_W(G)}(N^2)$ terms with maximal degree are those for which $\sigma_W(G)$ is planar.
  We will show that $\prod_{v\in W}s(v)$ has the same value for all $W\subset V(G)$ with $\sigma_W(G)$ planar.
  By Lemma~\ref{lem:whitney}, we may reduce to the case of flipping over a single edge-connect-summand.
  Let $W\subset V(G)$ with $\sigma_W(G)$ planar, and let $U\subset V(G)$ be the summand.

  Within the edge-connect-summand, $\sum_{v\in U}\deg(v)=2(e+1)$, with $e$ the number of edges in the subgraph induced by $U$, where the additional $1$ counts the two half edges leaving the summand.
  With $W'$ being the symmetric difference of $W$ and $U$,
  \begin{align*}
    \prod_{v\in W'}s(v)&=\left(\prod_{v\in W}(-1)^{\deg(v)}\right)
                         \left(\prod_{v\in U}(-1)^{\deg(v)}\right)\\
                       &=\left(\prod_{v\in W}(-1)^{\deg(v)}\right)\left((-1)^{\sum_{v\in U}\deg(v)}\right)
                       =\prod_{v\in W}(-1)^{\deg(v)},
  \end{align*}
  since we established the sum of degrees is even.
  Therefore, the absolute value of the coefficient of the $N^{2b_1(G)}$ term in $W_{\mathfrak{sl}(N)}(G)$ is the number of $W\subset V(G)$ such that $\sigma_W(G)$ is planar.

  The converse is Theorem~\ref{thm:s-poly-degree}, since planar and bridgeless implies there being no coloops.
\end{proof}

\begin{lemma}
  \label{lemma:sl2-spoly4}
  If $G$ is a signed virtual graph, then if the vertices have sign $s(v)=(-1)^{\deg(v)}$ for all $v\in V(G)$, $W_{\mathfrak{sl}(2)}(G)=2^{\abs{V(G)}}S_G(4)$, and otherwise $W_{\mathfrak{sl}(2)}(G)=0$.
\end{lemma}
\begin{proof}
  Using Lemma~\ref{lemma:q4-vtx-flip}, at $Q=4$ we can obtain
  \begin{align*}
    \raisebox{-21pt}{\marginbox{4pt 0}{}}
    =
    \raisebox{-21pt}{\marginbox{4pt 0}{}}
    +a
    \raisebox{-21pt}{\marginbox{4pt 0}{}}
    =(1+a(-1)^{\deg(v)})\raisebox{-21pt}{\marginbox{4pt 0}{}}.
  \end{align*}
  Thus, if there is a vertex $v$ such that $s(v)\neq (-1)^{\deg(v)}$, the evaluation is $0$.
  Otherwise, each vertex contributes a factor of $2$, hence we have $W_{\mathfrak{sl}(2)}(G)=2^{\abs{V(G)}}S_G(4)$.
\end{proof}

The following theorem is a generalization of Penrose\cite{Penrose1971}, and all the evaluations give a normalization of what Jaeger calls the \emph{Penrose number}\cite{Jaeger1989}.
\begin{theorem}
  \label{thm:penrose-spoly-equivalence}
  For a virtual graph $G$, the evaluations $W_{\mathfrak{so}(3)}(G)$, $W_{\mathfrak{so}(-2)}(G)$, $W_{\mathfrak{sl}(\pm 2)}(G)$, and $S_{G}(4)$ are equal up to a suitable normalization, where the vertices are given the sign $s(v)=(-1)^{\deg(v)}$ for all $v\in V(G)$.
\end{theorem}
\begin{proof}
  Lemma~\ref{lemma:sl2-spoly4} was that $W_{\mathfrak{sl}(2)}(G)$ and $S_G(4)$ are equal up to normalization.
  From Definition~\ref{lem:sl-as-s-extended}, $W_{\mathfrak{sl}(2)}(G)=W_{\mathfrak{sl}(-2)}(G)$.
  The algebra $Br^{-2}_2$ has $\raisebox{-6pt}{\marginbox{1.5pt 0pt}{}}+\raisebox{-6pt}{\marginbox{1.5pt 0pt}{}}+\raisebox{-6pt}{\marginbox{1.5pt 0pt}{}}$ in its trace radical, so in the quotient there is the relation $\raisebox{-6pt}{\marginbox{1.5pt 0pt}{}}=-\raisebox{-6pt}{\marginbox{1.5pt 0pt}{}}-\raisebox{-6pt}{\marginbox{1.5pt 0pt}{}}$, which Penrose calls the \emph{binor identity}.
  Since $\raisebox{-6pt}{\marginbox{1.5pt 0pt}{}}\circ\raisebox{-6pt}{\marginbox{1.5pt 0pt}{}}=0$, we have the relation $W_{\mathfrak{sl}(-2)}(\tau_v G)=-W_{\mathfrak{sl}(-2)}(G)$, hence Lemma~\ref{lem:so-as-sl} implies $W_{\mathfrak{so}(-2)}=2^{\abs{E(G)}}W_{\mathfrak{sl}(-2)}(G)$.
  Finally, $\mathfrak{so}(3)\cong\mathfrak{sl}(2)$ as Lie algebras.
\end{proof}

\begin{remark}
  There is also a functor $\VG^{\C(Q)}\to\Br^{-Q^{1/2}}$ to calculate the $S$-polynomial, with the edges being replaced instead by some normalization of $P^{(2)}=\raisebox{-6pt}{\marginbox{1.5pt 0pt}{}}+Q^{-1/2}\raisebox{-6pt}{\marginbox{1.5pt 0pt}{}}$.
  One can extend the $S$-polynomial to graphs embedded in unorientable surfaces by replacing half twists with $\raisebox{-6pt}{\marginbox{1.5pt 0pt}{}}$ in either expansion, yielding distinct invariants.
\end{remark}

\begin{remark}
  For the $\Br^{-Q^{1/2}}$ expansion of the $S$-polynomial, at $Q=4$ the binor identity implies that the extension of the $S$-polynomial that replaces half twists by $\raisebox{-6pt}{\marginbox{1.5pt 0pt}{}}$ has the property that the insertion of a half twist multiplies the polynomial by $-1$.
  Thus, if $G$ is a cubic virtual graph whose underlying graph is planar, $S_G(4)=\pm F_G(4)$ whether or not $G$ has this sort of half twist (which are presumed to be ignored by $F$).
\end{remark}

\subsection{Cellular embedding polynomial}

Recall the formulation of $W_{\mathfrak{sl}(N)}(G)$ when $G$ is a cubic graph, where vertices are replaced by $\raisebox{-8pt}{\marginbox{1.5pt 0pt}{}}-\raisebox{-8pt}{\marginbox{1.5pt 0pt}{}}$ and edges by $\raisebox{-6pt}{\marginbox{1.5pt 0pt}{}}$.
As suggested in \cite{Bar-Natan1997}, by thinking of $\raisebox{-8pt}{\marginbox{1.5pt 0pt}{}}$ and $\raisebox{-8pt}{\marginbox{1.5pt 0pt}{}}$ as being two rotations for a vertex, we can identify the coefficient of $N^{2k}$ in $W_{\mathfrak{sl}(N)}(G)$ as a signed count of the number of genus-$(b_0(G)+\frac{1}{4}\abs{V(G)}-k)$ virtual graphs with the same underlying graph as $G$, with the sign being determined by the parity of the number of vertices given opposite rotation.

With this normalization, the evaluation $C_G(x)=x^{b_0(G)+\frac{1}{4}\abs{V(G)}}W_{\mathfrak{sl}(x^{-1/2})}(G)$ is a polynomial such that the coefficient of the $x^g$ term is for the genus-$g$ virtual graphs, and we call it the \emph{cellular embedding polynomial}.
For the cubic case, a restatement of Theorem~\ref{thm:sl-planar} is that the underlying graph of a connected cubic virtual graph $G$ is bridgeless and planar if and only if $C_G(0)\neq 0$.

One can define a kind of ``nonplanar'' algebra for $C_G$ by taking as generators compact connected surfaces whose boundary is partitioned into labeled arcs, with the relation that $[\Sigma\csum T^2]=x[\Sigma]$.
Then $C_G$ comes from a map from cubic virtual graphs by sending cubic vertices to a difference of two triangles with opposite orientations.
The second author calculated the cellular embedding polynomial in this way for all connected cubic graphs with up to $22$ vertices and girth at least $3$.
There are $471{,}932$ such polynomials, $684$ of which are for planar graphs.
Of the planar graphs, the nearest roots within $1/4$ of $1/4$ are all real, and the nearest occurs in $2 (1-x) (1 + 16 x + 87 x^2 - 504 x^3 + 368 x^4)$, the polynomial of a graph with $22$ vertices.

\section{Yamada polynomial}
\label{sec:yamada-polynomial}

Yamada introduced a one-variable Laurent polynomial invariant $R_G(q)$ for spatial graphs $G$ in \cite{Yamada1989}.
It is the $\mathcal{U}_q(\mathfrak{sl}_2)$ Reshetikhin-Turaev invariant of cubic spatial graphs, coloring the edges with the $3$-dimensional irreducible representation $V_2$ with the unique nontrivial intertwiner $V_2\otimes V_2\to V_2$ at vertices, extended to arbitrary degree through contraction-deletion.

There are a few normalizations of $R_G(q)$ in the literature, and we choose one implicitly in the following definition, differing by a factor of $(-1)^{\abs{V(G)}-\abs{E(G)}}$ from the original:
\begin{definition}
  \label{def:yamada-poly}
  Let $G$ be a spatial graph.
  The Laurent polynomial $R_G(q)$ is determined by the following properties:
  \begin{enumerate}
  \item If $G$ has a diagram with no (classical) crossings, then \[R_G(q)=F_G((q^{1/2}+q^{-1/2})^2),\] where $F_G$ is the flow polynomial of this diagram.
  \item There is the following local relation:
    \begin{equation*}
      \gRight_q=q\,\gUp_q+q^{-1}\,\gSide_q-\gX_q,
    \end{equation*}
    where all four disks represent diagrams that differ only within the disk in the way represented, and where $G_q$ is shorthand for $R_G(q)$.
  \end{enumerate}
\end{definition}

One may evaluate the Yamada polynomial of a spatial graph whose diagram has $n$ crossings by expanding all the crossings with the local relation to get $3^n$ planar graphs, and then evaluating the flow polynomial for each expansion.

If $G$ is a cubic flat vertex graph, then $R_G(q)$ is well-defined up to a power of $q^2$.
If it is merely a cubic pliable vertex graph, then it is well-defined up to a power of $-q$.
Thus, if $G$ is pliable isotopic to a planar graph, $R_G(q)=(-q)^k F_G((q^{1/2}+q^{-1/2})^2)$, for some power of $k$.

There is an extension of the Yamada polynomial to virtual spatial graphs by Fleming and Mellor in \cite{Fleming2007}, which we will denote by $R^F_G(q)$, where the flow polynomial is used for the expansions.
This polynomial has the weakness that it is unable to distinguish the graphs in Figure~\ref{fig:inequivalent-theta-graphs}.
We suggest a different extension of the Yamada polynomial for virtual spatial graphs that is, however, able to distinguish these graphs.
This is essentially a quantum virtual link invariant as Kauffman defined them in \cite{Kauffman1999}, but for virtual spatial graphs.
\begin{definition}
  Let $G$ be a virtual spatial graph. The Laurent polynomial $R^S_G(q)$ is determined by the following properties:
  \begin{enumerate}
  \item If $G$ is a virtual graph (that is, if there is a diagram for $G$ with no classical crossings), then \[R^S_G(q)=S_G((q^{1/2}+q^{-1/2})^2),\] where $S_G$ is the $S$-polynomial of this diagram.
  \item There is the following local relation:
    \begin{equation*}
      \gRight_q=q\,\gUp_q+q^{-1}\,\gSide_q-\gX_q
    \end{equation*}
    with the same convention as in Definition~\ref{def:yamada-poly}, but with $R^S$ instead of $R$.
  \end{enumerate}
\end{definition}
This is an invariant because the $S$-polynomial is locally the flow polynomial and because the Yamada polynomial is an invariant.
That is, the virtual crossings can be ``moved'' away from the region where a Reidemeister or rigid vertex isotopy move occurs.
Virtual graph moves come for free from the definition of $S_G$.
Well-definedness can be observed by writing out a state sum or by thinking about the local relation as corresponding to an expansion in $\Sc$ and therefore in $\Br$.

We can use the polynomials together to extend \cite[Proposition 5.1]{Miyazawa2007} to all virtual spatial graphs, rather than just virtual links:
\begin{theorem}
  \label{thm:R-polys-detect}
  If $G$ is a virtual spatial graph and $R^F_G(q)\neq R^S_G(q)$, then $G$ is not equivalent to a classical spatial graph.
  If additionally $G$ is cubic, then $G$ is not pliable vertex isotopy equivalent to a classical spatial graph.
\end{theorem}
\begin{proof}
  If $G$ were equivalent to a classical spatial graph, both of these polynomials would coincide with $R_G(q)$.
  In the case $G$ is cubic, both $R^F$ and $R^S$ have the same relations for pliable vertex isotopy as does $R$, and so both sides accumulate the same number of factors of $-q$ in a sequence of pliable vertex isotopy moves.
\end{proof}

\begin{example}
  For example, if $G$ is the toroidal theta graph in Figure~\ref{fig:inequivalent-theta-graphs},
  \begin{align*}
    R^F_G(q)&=(q+q^{-1})(q+1+q^{-1})\\
    R^S_G(q)&=-2(q+1+q^{-1}),
  \end{align*}
  so $G$ is not pliable vertex isotopy equivalent to a classical spatial graph.
  Note also that $R^F$ is the same for both theta graphs.
\end{example}

\begin{example}
  Consider $D$ from \cite{Kauffman1999} (see Figure~\ref{fig:kauffman-knot-D}), a virtual knot with trivial virtual Jones polynomial that Kauffman's $Z$ refinement was unable to tell whether or not was classical.
  \begin{align*}
    R^F_D(q) &= q^2(q+1+q^{-1})\\
    R^S_D(q) &= -(q-1-q^{-1})(q^2+q-q^{-1})(q+1+q^{-1}).
  \end{align*}
\end{example}

\begin{example}
  The virtual knot $K$ in Figure~\ref{fig:virtual-toroidal-triv-invt} has $R^F_K(q)=R^S_K(q)=q+1+q^{-1}$.
  $K$ is nonclassical\cite{Kauffman1999}, hence Theorem~\ref{thm:R-polys-detect} is not a necessary condition for non-classicality, and there are nontrivial virtual knots with the same $R^F$ and $R^S$ polynomials as the unknot's.
\end{example}

\begin{example}
  If $U$ is the unlink of two circles, $R^F_U(q)=R^S_U(q)=(q+1+q^{-1})^2$.
  If $L$ is the ``virtual unlink'' in Figure~\ref{fig:links-two-circles},
  \begin{align*}
    R^F_L(q) &= -(q+1+q^{-1})\\
    R^S_L(q) &= (q+1+q^{-1})^2,
  \end{align*}
  hence $R^F$ can distinguish $L$ from $U$, even up to framing change.
  Paired with the theta graph example, we see that neither $R^F$ nor $R^S$ is a finer invariant than the other.
\end{example}

\begin{figure}[htp]
  \centering
  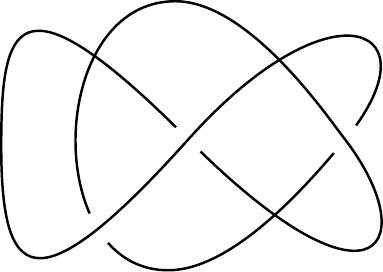
  \caption{Virtual knot $D$ from \cite{Kauffman1999} with trivial virtual Jones polynomial.}
  \label{fig:kauffman-knot-D}
\end{figure}
\begin{figure}[htp]
  \centering
  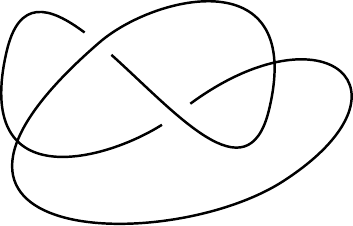
  \caption{A virtual knot of virtual genus $1$ for which Theorem~\ref{thm:R-polys-detect} cannot detect non-classicality.}
  \label{fig:virtual-toroidal-triv-invt}
\end{figure}
\begin{figure}[htp]
  \centering
  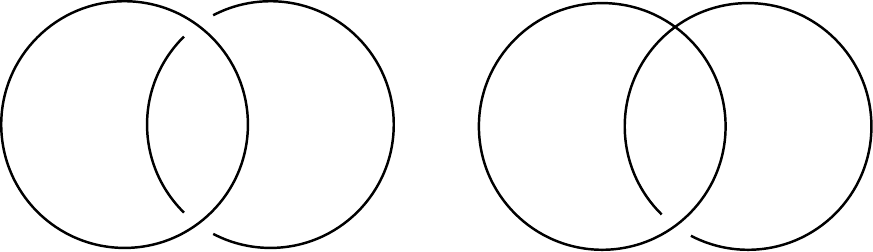
  \caption{Unlink and ``virtual unlink,'' distinguishable by $R^F$ and not $R^S$.}
  \label{fig:links-two-circles}
\end{figure}

\begin{remark}
  \label{rem:alg-virt-graph-genus}
  This theorem gives an algebraic proof of a weaker form of Corollary~\ref{cor:virtual-graph-genus}: if $G$ is a virtual graph with no bridge edges, then while Theorem~\ref{thm:s-poly-degree} implies $\deg S_G(Q)\leq b_1(G)-g(G)$, it is well known that $\deg F_G(Q)=b_1(G)$, so if $g(G)\neq 0$, Theorem~\ref{thm:R-polys-detect} implies that the virtual genus of $G$ as a virtual spatial graph is nonzero.
\end{remark}

\begin{figure}[htp]
  \centering
\begingroup%
  \makeatletter%
  \providecommand\color[2][]{%
    \errmessage{(Inkscape) Color is used for the text in Inkscape, but the package 'color.sty' is not loaded}%
    \renewcommand\color[2][]{}%
  }%
  \providecommand\transparent[1]{%
    \errmessage{(Inkscape) Transparency is used (non-zero) for the text in Inkscape, but the package 'transparent.sty' is not loaded}%
    \renewcommand\transparent[1]{}%
  }%
  \providecommand\rotatebox[2]{#2}%
  \newcommand*\fsize{\dimexpr\f@size pt\relax}%
  \newcommand*\lineheight[1]{\fontsize{\fsize}{#1\fsize}\selectfont}%
  \ifx\svgwidth\undefined%
    \setlength{\unitlength}{310.97474612bp}%
    \ifx\svgscale\undefined%
      \relax%
    \else%
      \setlength{\unitlength}{\unitlength * \real{\svgscale}}%
    \fi%
  \else%
    \setlength{\unitlength}{\svgwidth}%
  \fi%
  \global\let\svgwidth\undefined%
  \global\let\svgscale\undefined%
  \makeatother%
  \begin{picture}(1,0.12974238)%
    \lineheight{1}%
    \setlength\tabcolsep{0pt}%
    \put(0.1888963,0.06219678){\color[rgb]{0,0,0}\makebox(0,0)[lt]{\lineheight{0}\smash{\begin{tabular}[t]{l}$\leftrightarrow$\end{tabular}}}}%
    \put(0,0){\includegraphics[width=\unitlength,page=1]{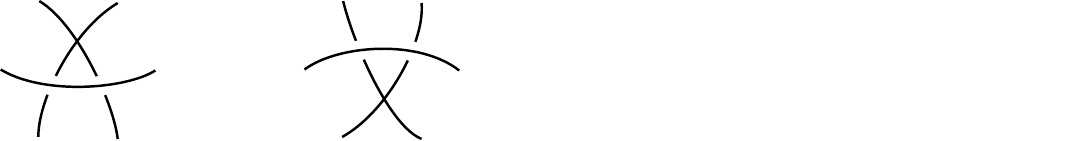}}%
    \put(0.7628981,0.06219678){\color[rgb]{0,0,0}\makebox(0,0)[lt]{\lineheight{0}\smash{\begin{tabular}[t]{l}$\leftrightarrow$\end{tabular}}}}%
    \put(0,0){\includegraphics[width=\unitlength,page=2]{forbidden-moves.pdf}}%
  \end{picture}%
\endgroup%

  \caption{The so-called ``forbidden moves'' for virtual spatial graphs, which can change the graph type.  They leave $R^S_G(q)$ invariant at $q=\pm 1,e^{\pm 2\pi i/3}$ and $R^F_G(q)$ invariant at $q=\pm 1,\pm i,e^{\pm 2\pi i/3}$, by analysis of negligible elements.}
  \label{fig:forbidden-moves}
\end{figure}

For the purpose of comparing $R^F_G(1)$ and $F_G(4)$, we introduce an invariant of virtual spatial graphs that is related to the planarity obstruction of \cite{Tutte1970,Sarkaria1991} and gives an algebraic characterization of a version of \cite[Proposition~1]{Fleming2007}.
The invariant is insensitive to crossing changes ($\raisebox{-6pt}{\marginbox{1.5pt 0pt}{}}\leftrightarrow\raisebox{-6pt}{\marginbox{1.5pt 0pt}{}}$) but is sensitive to the rotation system.
For a virtual spatial graph $G$ with oriented edges, let $G'$ denote its underlying ribbon graph, meaning the virtual graph obtained by replacing classical crossings with virtual crossings.
Have the group $C_2=\langle\tau:\tau^2=1\rangle$ act on $G'\times G'$ by $\tau(x,y)=(y,x)$, and let $\Sym^2G'$ denote the symmetric product $(G'\times G')/C_2$.
After deformation retracting the quotients of the $2$-cells intersecting the diagonal, we can describe the cell structure of $\Sym^2G'$ as follows: the $0$-cells correspond to unordered pairs of vertices, the $1$-cells to $V(G)\times E(G)$, and the $2$-cells to unordered pairs of distinct edges.
Write elementary $2$-cocycles of $\Sym^2G'$ as $e\wedge f$ for $e,f\in E(G')$, which is dual to the quotient of the cell $e\times f$.
This notation is indeed meant to suggest an exterior square: $e\wedge e=0$ since there are no diagonal $2$-cells, and $e\wedge f=-f\wedge e$ since $\tau$ reverses orientation.
Coboundaries are generated by elements of the form $\delta(x,e)=\sum_i \pm f_i\wedge e$ for all edges $f_i$ incident to vertex $x$, with sign determined by the orientation of $f_i$ at $x$.

\begin{definition}
  Let $G$ be a virtual spatial graph with oriented edges and consider a particular diagram for it.
  For a ring $R$, let $\overline{\mathfrak{o}}(G;R)\in H^2(\Sym^2G';R)$ be the cohomology class with representative cocycle $\alpha$ such that $\alpha(e\times f)$ is the algebraic intersection number of edges $e$ and $f$ at the classical crossings.
  Concretely, associate to each crossing $\raisebox{-6pt}{\marginbox{1.5pt 0pt}{}}$ and $\raisebox{-6pt}{\marginbox{1.5pt 0pt}{}}$ the value $e\wedge f\in H^2(\Sym^2G')$, where $e,f\in E(G')$ are the two edges taking part in the crossing, with $e$ being the edge pointing toward the top-right --- note that the sign of the crossing is taken into consideration but not its type.
  Then $\overline{\mathfrak{o}}(G;R)$ is the sum of these values over all crossings.
\end{definition}

\begin{remark}
  One could equivalently place the invariant in $H_{C_2}^2(G'\times G';R)$, the cohomology of $C_2$-equivariant cocycles, with $R$ as a trivial $C_2$-module.
\end{remark}

\begin{lemma}
  $\overline{\mathfrak{o}}(G;\mathbb{Z})$ is a complete invariant of virtual spatial graphs with oriented edges up to (a) move I, (b) the ``forbidden moves'' in Figure~\ref{fig:forbidden-moves}, and (c) crossing change ($\raisebox{-6pt}{\marginbox{1.5pt 0pt}{}}\leftrightarrow \raisebox{-6pt}{\marginbox{1.5pt 0pt}{}}$).
\end{lemma}
\begin{proof}
  First is to show that $\overline{\mathfrak{o}}(G;\mathbb{Z})$ is an invariant.
  Move I is that $e\wedge e=0$ for $e\in E(G')$.
  Move II is that $e\wedge f+ f\wedge e=0$ for $e,f\in E(G')$.
  Move III is that all three edges take part in the same crossings with the same sign.
  Move IV is that for a vertex $v$ whose incident edges are $e_1,\dots,e_n$ with an edge $f$ crossing all these edges with the same sign as in the diagram, $e_1\wedge f+\dots+e_n\wedge f=(e_1+\dots+e_n)\wedge f=(\delta v)\wedge f=\delta(v, f)$, which is zero in $H^2(\Sym^2G')$.
  As was the case for move III, the forbidden moves do not change which edges cross and with what sign.
  Lastly, for crossing change, $\overline{\mathfrak{o}}(G;\mathbb{Z})$ only uses the sign of the crossing, not the crossing type.

  Completeness follows from \cite[Proposition~1]{Fleming2007}:
  The forbidden moves allow crossings along an edge to commute\cite{Nelson2001} in the sense that
  \begin{equation*}
    \raisebox{-17.5pt}{\marginbox{4pt 0}{
\begingroup%
  \makeatletter%
  \providecommand\color[2][]{%
    \errmessage{(Inkscape) Color is used for the text in Inkscape, but the package 'color.sty' is not loaded}%
    \renewcommand\color[2][]{}%
  }%
  \providecommand\transparent[1]{%
    \errmessage{(Inkscape) Transparency is used (non-zero) for the text in Inkscape, but the package 'transparent.sty' is not loaded}%
    \renewcommand\transparent[1]{}%
  }%
  \providecommand\rotatebox[2]{#2}%
  \newcommand*\fsize{\dimexpr\f@size pt\relax}%
  \newcommand*\lineheight[1]{\fontsize{\fsize}{#1\fsize}\selectfont}%
  \ifx\svgwidth\undefined%
    \setlength{\unitlength}{108.75000189bp}%
    \ifx\svgscale\undefined%
      \relax%
    \else%
      \setlength{\unitlength}{\unitlength * \real{\svgscale}}%
    \fi%
  \else%
    \setlength{\unitlength}{\svgwidth}%
  \fi%
  \global\let\svgwidth\undefined%
  \global\let\svgscale\undefined%
  \makeatother%
  \begin{picture}(1,0.34482759)%
    \lineheight{1}%
    \setlength\tabcolsep{0pt}%
    \put(0.44314891,0.16107191){\color[rgb]{0,0,0}\makebox(0,0)[lt]{\lineheight{0}\smash{\begin{tabular}[t]{l}$\leftrightarrow$\end{tabular}}}}%
    \put(0,0){\includegraphics[width=\unitlength,page=1]{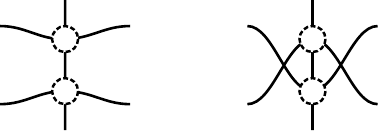}}%
  \end{picture}%
\endgroup%
}},
  \end{equation*}
  with the dotted circles containing classical crossings of either type, and each crossing slides across the other.
  Using this operation, all self-crossings along an edge can be removed using move I.
  Furthermore, between a pair of edges, crossings of opposite sign can be removed using move II.
  Once in this form, $\overline{\mathfrak{o}}(G;\mathbb{Z})$ gives the number of times pairs of edges cross and with what sign, up to move IV.
\end{proof}

\begin{lemma}
  \label{lem:vsg-z2-invt}
  For a virtual spatial graph $G$ up to move I, the forbidden moves, crossing change, and the virtualization move
  \begin{equation*}
    \raisebox{-17.5pt}{\marginbox{4pt 0}{
\begingroup%
  \makeatletter%
  \providecommand\color[2][]{%
    \errmessage{(Inkscape) Color is used for the text in Inkscape, but the package 'color.sty' is not loaded}%
    \renewcommand\color[2][]{}%
  }%
  \providecommand\transparent[1]{%
    \errmessage{(Inkscape) Transparency is used (non-zero) for the text in Inkscape, but the package 'transparent.sty' is not loaded}%
    \renewcommand\transparent[1]{}%
  }%
  \providecommand\rotatebox[2]{#2}%
  \newcommand*\fsize{\dimexpr\f@size pt\relax}%
  \newcommand*\lineheight[1]{\fontsize{\fsize}{#1\fsize}\selectfont}%
  \ifx\svgwidth\undefined%
    \setlength{\unitlength}{108.7499991bp}%
    \ifx\svgscale\undefined%
      \relax%
    \else%
      \setlength{\unitlength}{\unitlength * \real{\svgscale}}%
    \fi%
  \else%
    \setlength{\unitlength}{\svgwidth}%
  \fi%
  \global\let\svgwidth\undefined%
  \global\let\svgscale\undefined%
  \makeatother%
  \begin{picture}(1,0.3448276)%
    \lineheight{1}%
    \setlength\tabcolsep{0pt}%
    \put(0,0){\includegraphics[width=\unitlength,page=1]{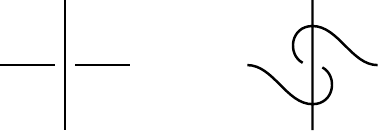}}%
    \put(0.44314885,0.16107194){\color[rgb]{0,0,0}\makebox(0,0)[lt]{\lineheight{0}\smash{\begin{tabular}[t]{l}$\leftrightarrow$\end{tabular}}}}%
  \end{picture}%
\endgroup%
}},
  \end{equation*}
  $\overline{\mathfrak{o}}(G;\mathbb{Z}/2\mathbb{Z})$ is a complete invariant.
\end{lemma}
\begin{proof}
  The virtualization move changes the sign of a crossing, allowing one to cancel out pairs of crossings between edge pairs when the virtual spatial graph is in the form in the proof of the preceding lemma.
\end{proof}

\begin{remark}
  If we wanted an invariant of pliable vertex isotopy (move VI), we could instead consider the restriction $\mathfrak{o}(G;R)\in H^2(\Sym^2G'-\Delta(G);R)$, with $\Delta:G\to G\times G$ being the diagonal map.
  The restriction to this cohomology group sends $e\wedge f$ to $0$ if the edges $e$ and $f$ are incident.
\end{remark}

\begin{remark}
  One could also define an invariant $\overline{\mathfrak{o}}(G;R)\in H^2(\Sym^2G;R)$ for a virtual graph $G$, where we take any diagram for $G$ and follow the same procedure with the virtual crossings.
  It can be shown that this is independent of the diagram.
  As an obstruction, $\overline{\mathfrak{o}}(G;\mathbb{Z}/2\mathbb{Z})$ is zero if and only if $g(G)=0$, similarly to \cite{Tutte1970,Sarkaria1991}.
\end{remark}

\begin{theorem}
  Let $G$ be a virtual spatial graph and $G'$ be its underlying ribbon graph.
  Then there are the relations $R^F_G(-1)=F_{G'}(0)=S_{G'}(0)=R^S_G(-1)$, and $R^S_G(1)=S_{G'}(4)$.

  If $\overline{\mathfrak{o}}(G;\mathbb{Z}/2\mathbb{Z})=0$, then $R^F_G(1)=F_{G'}(4)$.
  If additionally there is some $W\subset V(G')$ such that $\sigma_W G'$ is planar, then $R^F_G(1)=(-1)^{\sum_{v\in W}\deg(v)}R^S_G(1)$.
\end{theorem}
\begin{proof}
  At $q=\pm1$, the algebra $\Sc_2^{(q^{1/2}+q^{-1/2})^2}$ has the negligible element
  \[
    q \raisebox{-6pt}{\marginbox{1.5pt 0pt}{}} + q^{-1}\raisebox{-6pt}{\marginbox{1.5pt 0pt}{}} - \raisebox{-6pt}{\marginbox{1.5pt 0pt}{
\begingroup%
  \makeatletter%
  \providecommand\color[2][]{%
    \errmessage{(Inkscape) Color is used for the text in Inkscape, but the package 'color.sty' is not loaded}%
    \renewcommand\color[2][]{}%
  }%
  \providecommand\transparent[1]{%
    \errmessage{(Inkscape) Transparency is used (non-zero) for the text in Inkscape, but the package 'transparent.sty' is not loaded}%
    \renewcommand\transparent[1]{}%
  }%
  \providecommand\rotatebox[2]{#2}%
  \newcommand*\fsize{\dimexpr\f@size pt\relax}%
  \newcommand*\lineheight[1]{\fontsize{\fsize}{#1\fsize}\selectfont}%
  \ifx\svgwidth\undefined%
    \setlength{\unitlength}{19.06299186bp}%
    \ifx\svgscale\undefined%
      \relax%
    \else%
      \setlength{\unitlength}{\unitlength * \real{\svgscale}}%
    \fi%
  \else%
    \setlength{\unitlength}{\svgwidth}%
  \fi%
  \global\let\svgwidth\undefined%
  \global\let\svgscale\undefined%
  \makeatother%
  \begin{picture}(1,0.99999993)%
    \lineheight{1}%
    \setlength\tabcolsep{0pt}%
    \put(0,0){\includegraphics[width=\unitlength,page=1]{s2-v4.pdf}}%
  \end{picture}%
\endgroup%
}} - \raisebox{-6pt}{\marginbox{1.5pt 0pt}{}},
  \]
  hence $R^S$ has the relation $\gRight_q=\gV_q$ at $q=\pm 1$.
  Thus $R^S_G(1)=S_{G'}(4)$ and $R^S_G(-1)=S_{G'}(0)$.
  Similarly, at $q=-1$, $\Flow_2^{(q^{1/2}+q^{-1/2})^2}$ has a negligible element of the same form, hence $R^F_G(-1)=F_{G'}(0)$.
  Proposition~\ref{prop:s0-f0} gives $S_{G'}(0)=F_{G'}(0)$.

  For $q=1$, we aim to apply Lemma~\ref{lem:vsg-z2-invt}.
  $R^F(1)$ is invariant under move I.
  The ninety-degree rotational symmetry of the right-hand side of the crossing relations for the $R^F$ polynomial gives $\gRight_1=\gLeft_1$, and by inspecting the expansion we can see it is invariant under the virtualization move.
  Through the analysis of negligible elements, $R^F(1)$ is invariant under the forbidden moves of Figure~\ref{fig:forbidden-moves}.
  Hence by Lemma~\ref{lem:vsg-z2-invt}, if $\overline{\mathfrak{o}}(G;\mathbb{Z}/2\mathbb{Z})=0$, there is a diagram for $G$ up to these moves with no classical crossings.
  By the definition of $R^F$, $R^F_G(1)=F_{G'}(4)$.

  If there is a $W\subset V(G')$ such that $\sigma_W G'$ is planar, then Proposition~\ref{prop:s4-f4} gives $(-1)^n S_{G'}(4)=F_{G'}(4)$ with $n=\sum_{v\in W}\deg(v)$.
\end{proof}
\begin{remark}
  If $G$ is cubic and $\mathfrak{o}(G;\mathbb{Z}/2\mathbb{Z})=0$, then $R^F_G(1)=\pm F_{G'}(4)$ since flipping a trivalent vertex negates the $R^F$ polynomial.
\end{remark}

\begin{example}
  For the ``virtual unlink'' $L$ in Figure~\ref{fig:links-two-circles}, $R_L^F(1)=-3$ yet $F_L(4)=9$.
\end{example}

\subsection{The golden inequality conjecture for virtual spatial graphs}
\label{sec:golden-conj-vsg}

In \cite{Agol2018}, Agol and Krushkal establish an extension of Tutte's golden identity for the flow polynomial of planar cubic graphs to cubic spatial graphs.
With the normalization used in this paper, if $G$ is a cubic spatial graph,
\begin{equation*}
  R_G(e^{\pi i/5})= \phi^{\abs{E(G)}} R_G(e^{-2\pi i/5})^2,
\end{equation*}
where $\phi=\frac{1}{2}(1+\sqrt{5})$ is the golden ratio.
They conjecture that for abstract cubic graphs, Tutte's golden identity holds if and only if the graph is planar.

The virtual spatial graphs in Figure~\ref{fig:virtual-spatial-golden} are a handcuff graph $H$ and a theta graph $\Theta$ having the following Yamada polynomials:
\begin{align*}
  R^F_H(q) &= -q^{-2}(q-1)(q+1)^2(q+q^{-1})(q-1+q^{-1})(q+1+q^{-1}) \\
  R^S_H(q) &= -q^{-2}(q-1)(q+1)^2(q+1+q^{-1})(q^2-2q+4-2q^{-1}+q^{-2}) \\
  R^F_\Theta(q) &= q^{-2}(q+q^{-1})(q+1+q^{-1}) \\
  R^S_\Theta(q) &= -q^{-1}(q^2-3)(q+1+q^{-1})
\end{align*}
By Theorem~\ref{thm:R-polys-detect}, neither virtual spatial graph is equivalent to a virtual graph, hence it is not the case that the virtual genus is $0$ if and only if $R^F_G(e^{\pi i/5})= \phi^{\abs{E(G)}} R^F_G(e^{-2\pi i/5})^2$.

\begin{figure}[htp]
  \centering
  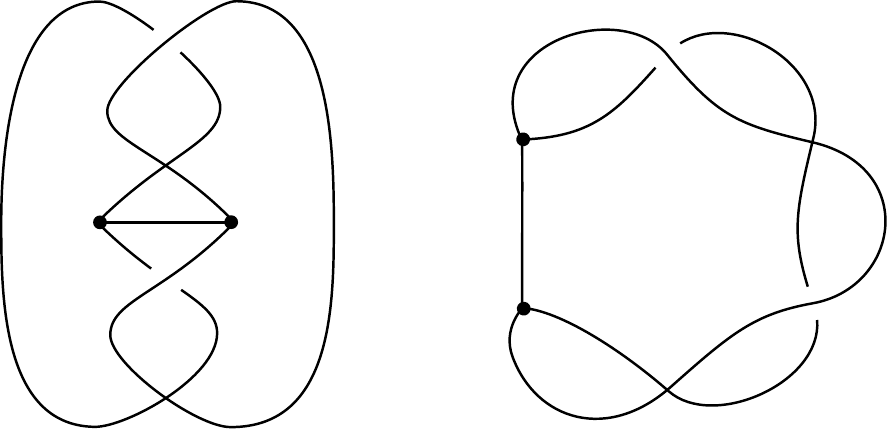
  \caption{Non-classical virtual spatial graphs with $R^F_G(e^{\pi i/5})= \phi^{\abs{E(G)}} R^F_G(e^{-2\pi i/5})^2$.}
  \label{fig:virtual-spatial-golden}
\end{figure}

\section*{Acknowledgments}

We would like to thank Noah Snyder for helpful advice and for introducing us to the connection between diagrams on surfaces and virtual crossings, Slava Krushkal for the suggestion to study chromatic identities for surface graphs, Paul Fendley for encouragement, and Ian Agol for valuable discussion.
K.\ M. was supported by NSF RTG grant DMS-1344991 and the Simons Foundation.

\let\MRhref\undefined
\bibliographystyle{hamsplain}
\bibliography{sources}

\providecommand{\bysame}{\leavevmode\hbox to3em{\hrulefill}\thinspace}
\providecommand{\MR}{\relax\ifhmode\unskip\space\fi MR }
\providecommand{\MRhref}[2]{%
  \href{http://www.ams.org/mathscinet-getitem?mr=#1}{#2}
}
\providecommand{\href}[2]{#2}
\begin{thebibliography}{10}

\bibitem{Agol2018}
Ian Agol and Vyacheslav Krushkal, \emph{Structure of the flow and yamada
  polynomials of cubic graphs},
  \href{http://arxiv.org/abs/1801.00502v1}{arXiv:1801.00502v1 [math.CO]}.

\bibitem{Agol2016}
Ian Agol and Vyacheslav Krushkal, \emph{Tutte relations, {TQFT}, and planarity
  of cubic graphs}, Illinois Journal of Mathematics \textbf{60} (2016), no.~1,
  273--288, \href{http://arxiv.org/abs/1512.07339v1}{arXiv:1512.07339v1
  [math.CO]}. \MR{3665181}

\bibitem{Aigner1997}
Martin Aigner, \emph{The {P}enrose polynomial of a plane graph}, Mathematische
  Annalen \textbf{307} (1997), no.~2, 173--189. \MR{1428870}

\bibitem{Bar-Natan1997}
Dror Bar-Natan, \emph{Lie algebras and the four color theorem}, Combinatorica
  \textbf{17} (1997), no.~1, 43--52,
  \href{http://arxiv.org/abs/q-alg/9606016v1}{arXiv:q-alg/9606016v1}.
  \MR{1466574}

\bibitem{Bollobas2001}
B{\'{e}}la Bollob{\'{a}}s and Oliver Riordan, \emph{A polynomial invariant of
  graphs on orientable surfaces}, Proc. London Math. Soc. (3) \textbf{83}
  (2001), no.~3, 513--531. \MR{1851080}

\bibitem{Brauer1937}
Richard Brauer, \emph{On algebras which are connected with the semisimple
  continuous groups}, Ann. of Math. (2) \textbf{38} (1937), no.~4, 857--872.
  \MR{1503378}

\bibitem{Carter2002}
J.~Scott Carter, Seiichi Kamada, and Masahico Saito, \emph{Stable equivalence
  of knots on surfaces and virtual knot cobordisms}, J. Knot Theory
  Ramifications \textbf{11} (2002), no.~3, 311--322,
  \href{http://arxiv.org/abs/math/0008118v1}{arXiv:math/0008118v1}.
  \MR{1905687}

\bibitem{Chmutov2009}
Sergei Chmutov, \emph{Generalized duality for graphs on surfaces and the signed
  {B}ollob\'{a}s-{R}iordan polynomial}, J. Combin. Theory Ser. B \textbf{99}
  (2009), no.~3, 617--638,
  \href{http://arxiv.org/abs/0711.3490v3}{arXiv:0711.3490v3 [math.CO]}.
  \MR{2507944}

\bibitem{Duzhin1998}
S.~V. Duzhin, A.~I. Kaishev, and S.~V. Chmutov, \emph{The algebra of
  {$3$}-graphs}, Tr. Mat. Inst. Steklova \textbf{221} (1998), 168--196.
  \MR{1683693}

\bibitem{Ellis-Monaghan2013}
Joanna~A. Ellis-Monaghan and Iain Moffatt, \emph{A {P}enrose polynomial for
  embedded graphs}, European Journal of Combinatorics \textbf{34} (2013),
  no.~2, 424--445. \MR{2994409}

\bibitem{Fendley2009}
Paul Fendley and Vyacheslav Krushkal, \emph{Tutte chromatic identities from the
  {T}emperley-{L}ieb algebra}, Geom. Topol. \textbf{13} (2009), 709--741,
  \href{http://arxiv.org/abs/0711.0016v3}{arXiv:0711.0016v3}. \MR{2469528}

\bibitem{Fleming2007}
Thomas Fleming and Blake Mellor, \emph{Virtual spatial graphs}, Kobe J. Math.
  \textbf{24} (2007), no.~2, 67--85,
  \href{http://arxiv.org/abs/math/0510158v2}{arXiv:math/0510158v2}.
  \MR{2488753}

\bibitem{Jacobsen2013}
Jesper~Lykke Jacobsen and Jes\'{u}s Salas, \emph{A generalized {B}eraha
  conjecture for non-planar graphs}, Nuclear Phys. B \textbf{875} (2013),
  no.~3, 678--718, \href{http://arxiv.org/abs/1303.5210v2}{arXiv:1303.5210v2}.
  \MR{3102902}

\bibitem{Jaeger1989}
Fran\c{c}ois Jaeger, \emph{On the {P}enrose number of cubic diagrams}, Discrete
  Mathematics \textbf{74} (1989), no.~1-2, 85--97, Graph colouring and
  variations. \MR{989125}

\bibitem{Kauffman1999}
Louis~H. Kauffman, \emph{Virtual knot theory}, European Journal of
  Combinatorics \textbf{20} (1999), no.~7, 663--690. \MR{1721925}

\bibitem{Kauffman2015}
Louis~Hirsch Kauffman and Vassily~Olegovich Manturov, \emph{Graphical
  constructions for the {${\rm sl}(3)$}, {$C_2$} and {$G_2$} invariants for
  virtual knots, virtual braids and free knots}, Journal of Knot Theory and its
  Ramifications \textbf{24} (2015), no.~6, 1550031, 47. \MR{3358442}

\bibitem{King2016}
Oliver King, Paul Martin, and Alison Parker, \emph{On central idempotents in
  the {B}rauer algebra}, September 2016,
  \href{http://arxiv.org/abs/1609.01183v1}{arXiv:1609.01183v1}.

\bibitem{Krushkal2011}
Vyacheslav Krushkal, \emph{Graphs, links, and duality on surfaces}, Combin.
  Probab. Comput. \textbf{20} (2011), 267--287,
  \href{http://arxiv.org/abs/0903.5312v3}{arXiv:0903.5312v3}. \MR{2769192}

\bibitem{Kuperberg2003}
Greg Kuperberg, \emph{What is a virtual link?}, Algebraic \& Geometric Topology
  \textbf{3} (2003), 587--591. \MR{1997331}

\bibitem{Lehrer2012}
Gustav Lehrer and Ruibin Zhang, \emph{The second fundamental theorem of
  invariant theory for the orthogonal group}, Annals of Mathematics. Second
  Series \textbf{176} (2012), no.~3, 2031--2054. \MR{2979865}

\bibitem{Miyazawa2007}
Yasuyuki Miyazawa, \emph{The {Y}amada polynomial for virtual graphs},
  Intelligence of low dimensional topology 2006, Ser. Knots Everything,
  vol.~40, World Sci. Publ., Hackensack, NJ, 2007, pp.~205--212. \MR{2371727}

\bibitem{Murakami1993}
Jun Murakami, \emph{The {Y}amada polynomial of spacial graphs and knit
  algebras}, Communications in Mathematical Physics \textbf{155} (1993), no.~3,
  511--522. \MR{1231641}

\bibitem{Nelson2001}
Sam Nelson, \emph{Unknotting virtual knots with {G}auss diagram forbidden
  moves}, Journal of Knot Theory and its Ramifications \textbf{10} (2001),
  no.~6, 931--935. \MR{1840276}

\bibitem{Penrose1971}
Roger Penrose, \emph{Applications of negative dimensional tensors},
  Combinatorial {M}athematics and its {A}pplications ({P}roc. {C}onf.,
  {O}xford, 1969), Academic Press, London, 1971, pp.~221--244. \MR{0281657}

\bibitem{Reshetikhin1990}
N.~Y. Reshetikhin and V.~G. Turaev, \emph{Ribbon graphs and their invaraints
  derived from quantum groups}, Comm. Math. Phys. \textbf{127} (1990), no.~1,
  1--26. \MR{1036112}

\bibitem{Sarkaria1991}
K.~S. Sarkaria, \emph{A one-dimensional {W}hitney trick and {K}uratowski's
  graph planarity criterion}, Israel Journal of Mathematics \textbf{73} (1991),
  no.~1, 79--89. \MR{1119929}

\bibitem{Tuba2005}
Imre Tuba and Hans Wenzl, \emph{On braided tensor categories of type {$BCD$}},
  Journal f\"ur die Reine und Angewandte Mathematik. [Crelle's Journal]
  \textbf{581} (2005), 31--69. \MR{2132671}

\bibitem{Tutte1970}
W.~T. Tutte, \emph{Toward a theory of crossing numbers}, J. Combinatorial
  Theory \textbf{8} (1970), 45--53. \MR{0262110}

\bibitem{Wenzl1988}
Hans Wenzl, \emph{On the structure of {B}rauer's centralizer algebras}, Ann. of
  Math. (2) \textbf{128} (1988), no.~1, 173--193. \MR{951511}

\bibitem{Yamada1989}
Shuji Yamada, \emph{An invariant of spatial graphs}, Journal of Graph Theory
  \textbf{13} (1989), no.~5, 537--551. \MR{1016274}

\end{thebibliography}

\end{document}